\documentclass{mcom-l} 

\copyrightinfo{}{}

\usepackage[T1]{fontenc}
\usepackage[utf8]{inputenc}

\usepackage{graphicx}
\usepackage{booktabs}
\usepackage{xcolor}
\usepackage{hyperref}

\usepackage{amsmath}
\usepackage{amsfonts}
\usepackage{amsthm}
\usepackage{mathtools}
\usepackage{braket} 
\usepackage{stmaryrd} 
\usepackage{empheq} 
\usepackage{mathrsfs}

\usepackage[notocbasic]{nomencl}
\makenomenclature
\usepackage{etoolbox}
\renewcommand\nomgroup[1]{%
  \item[\em
  \ifstrequal{#1}{A}{Geometric setting}{%
  \ifstrequal{#1}{B}{Function spaces}{%
  \ifstrequal{#1}{C}{Duality pairings}{%
  \ifstrequal{#1}{D}{Trace operators}{%
  \ifstrequal{#1}{E}{Other operators}{}}}}}
]}

\newtheorem{theorem}{Theorem}[section]
\newtheorem{lemma}[theorem]{Lemma}
\newtheorem{proposition}[theorem]{Proposition}
\newtheorem{corollary}[theorem]{Corollary}

\theoremstyle{definition}

\newtheorem{example}[theorem]{Example}

\theoremstyle{remark}
\newtheorem{remark}[theorem]{Remark}

\newcommand{\imagi}{\imath} 

\numberwithin{equation}{section}

\usepackage{accents}
\newlength{\dhatheight}
\newcommand{\doublehat}[1]{%
    \settoheight{\dhatheight}{\ensuremath{\hat{#1}}}%
    \addtolength{\dhatheight}{-0.2ex}%
    \hat{\vphantom{\rule{1pt}{\dhatheight}}%
    \smash{\hat{#1}}}}
\newcommand{\wdoublehat}[1]{%
    \settoheight{\dhatheight}{\ensuremath{\widehat{#1}}}%
    \addtolength{\dhatheight}{-0.35ex}%
    \widehat{\vphantom{\rule{1pt}{\dhatheight}}%
    \smash{\widehat{#1}}}}
 
\makeatletter
\@ifpackageloaded{stix}{%
}{%
  \DeclareFontEncoding{LS2}{}{\noaccents@}
  \DeclareFontSubstitution{LS2}{stix}{m}{n}
  \DeclareSymbolFont{stix@largesymbols}{LS2}{stixex}{m}{n}
  \SetSymbolFont{stix@largesymbols}{bold}{LS2}{stixex}{b}{n}
  \DeclareMathDelimiter{\lBrace}{\mathopen} {stix@largesymbols}{"E8}%
                                            {stix@largesymbols}{"0E}
  \DeclareMathDelimiter{\rBrace}{\mathclose}{stix@largesymbols}{"E9}%
                                            {stix@largesymbols}{"0F}
}
\makeatother

\DeclareFontFamily{U}{mathx}{\hyphenchar\font45}
\DeclareFontShape{U}{mathx}{m}{n}{
      <5> <6> <7> <8> <9> <10>
      <10.95> <12> <14.4> <17.28> <20.74> <24.88>
      mathx10
      }{}
\DeclareSymbolFont{mathx}{U}{mathx}{m}{n}
\DeclareFontSubstitution{U}{mathx}{m}{n}
\DeclareMathAccent{\widecheck}{0}{mathx}{"71}

\DeclareMathOperator{\Realpart}{Re}
\renewcommand{\Re}{\Realpart}

\DeclareMathOperator{\Imagpart}{Im}
\renewcommand{\Im}{\Imagpart}

\newcommand{\nc}{\newcommand}
\nc{\mrm}{\mathrm}
\nc{\mH}{\mrm{H}}
\nc{\mL}{\mrm{L}}
\nc{\mbH}{\mathbb{H}}
\nc{\mbX}{\mathbb{X}}
\nc{\fru}{\mathfrak{u}}
\nc{\frv}{\mathfrak{v}}
\nc{\frw}{\mathfrak{w}}
\nc{\RR}{\mathbb{R}}
\nc{\CC}{\mathbb{C}}
\nc{\bx}{\boldsymbol{x}}
\nc{\by}{\boldsymbol{y}}
\nc{\bn}{\boldsymbol{n}}
\nc{\bu}{\boldsymbol{u}}
\nc{\bv}{\boldsymbol{v}}
\nc{\bp}{\boldsymbol{p}}
\nc{\bq}{\boldsymbol{q}}
\nc{\bvphi}{\boldsymbol{\varphi}}

\nc{\ub}{\textsc{ub}}
\nc{\inc}{\textup{inc}}
\nc{\loc}{\textup{loc}}
\nc{\dir}{\textsc{d}}
\nc{\neu}{\textsc{n}}
\nc{\lbr}{\lbrack}
\nc{\rbr}{\rbrack}

\nc{\comp}{\mathsf{c}}
\nc{\Tr}{\mathsf{T}}
\nc{\Id}{\mathsf{Id}}

\nc{\XSigmaGamma}{\mbX(\Omega_\Sigma,\Gamma)} 
\nc{\XSigmaGammaM}{\mbX_{\mathsf{M}}(\Omega_\Sigma,\Gamma)} 

\begin{document}

\title[Multi-domain FEM-BEM coupling for acoustic scattering]{Multi-domain FEM-BEM coupling for acoustic scattering}


\author{Marcella Bonazzoli} \address{Inria, UMA, ENSTA Paris, Institut Polytechnique de Paris, 91120 Palaiseau, France}
\email{marcella.bonazzoli@inria.fr}

\author{Xavier Claeys} \address{Sorbonne Universit\'e, CNRS, Universit\'e de Paris, Inria, Laboratoire Jacques-Louis Lions, F-75005 Paris, France}
\email{xavier.claeys@sorbonne-universite.fr} \thanks{This work was supported by the French National Research Agency (ANR) in the framework of the project NonlocalDD, ANR-15-CE23-0017-01.}


\date{}

\dedicatory{}

\begin{abstract}
We model time-harmonic acoustic scattering by an object composed of piece-wise homogeneous parts and an arbitrarily heterogeneous part. We propose and analyze new formulations that couple, adopting a Costabel-type approach, boundary integral equations for the homogeneous subdomains with volume variational formulations for the heterogeneous subdomain. This is an extension of the Costabel FEM-BEM coupling to a multi-domain configuration, with  cross-points allowed, i.e.~points where three or more subdomains are adjacent. While generally just the exterior unbounded subdomain is treated with the BEM, here we wish to exploit the advantages of BEM whenever it is applicable, that is, for all the homogeneous parts of the scattering object. Our formulation is based on the multi-trace formalism, which initially was introduced for acoustic scattering by piece-wise homogeneous objects. Instead, here we allow the wavenumber to vary arbitrarily in a part of the domain.
We prove that the bilinear form associated with the proposed formulation satisfies a G\r{a}rding coercivity inequality, which ensures stability of the variational problem if it is uniquely solvable. We identify conditions for injectivity and construct modified versions immune to spurious resonances.    
\end{abstract}

\maketitle

\section{Introduction}

\noindent 
The efficient simulation of wave propagation problems in time-harmonic regime remains a computational challenge that is still the subject of  intensive research effort. Propagation media are generally heterogeneous, which is reflected by arbitrarily varying coefficients in the equations. Classical numerical methods to perform simulations in heterogeneous media usually rely on volume-type discretization schemes such as finite elements.
In many situations of practical relevance, material coefficients are piece-wise constant in certain parts of the computational domain, and this feature can be exploited to reformulate the problem by means of boundary integral operators as an equation defined only on the boundary, called Boundary Integral Equation (BIE). Indeed, boundary element methods, which are discretization schemes for BIEs, yield a significant reduction in the number of unknowns, higher accuracy at least away from the boundary, and better robustness to high frequency compared with finite elements. In addition, boundary integral operators can naturally deal with unbounded domains, provided that the boundary is bounded.

This is the general idea of Finite Element Method - Boundary Element Method (FEM-BEM) coupling, which aims at taking advantage of the versatility of the finite element method and the computational efficiency of the boundary element method. There already exists a well established literature on the numerical analysis of FEM-BEM coupling, in particular for time-harmonic acoustic problems, with several possible FEM-BEM strategies including the Johnson-N\'ed\'elec coupling \cite{MR583487}, the Bielak-McCamy coupling \cite{MR700668} or the symmetric Costabel coupling \cite{Costabel:lectures:1987,Costabel:symm:1987} (see e.g.~\cite{AFF:fembem:2013} for an overview of the three approaches). Another possible strategy relies on substructuring domain decomposition and FETI-BETI methods \cite{LangerSteinbach2005CBF,BendaliBoubendirEtAl2007FDD,MR2962738,CaudronAntoineEtAl2020OWC}. 
In the present contribution, we wish to focus on the Costabel coupling, which appears interesting from a numerical analysis perspective because it naturally leads to G\r{a}rding coercivity estimates. 

\begin{figure}
  \centering \includegraphics[width=0.45\linewidth]{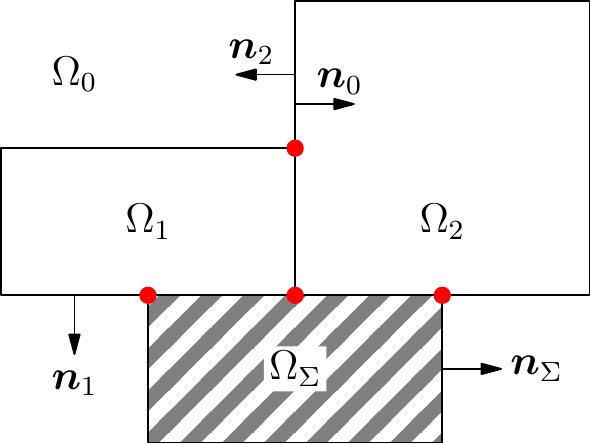}
\caption{Example of geometric setting: composite medium, with $\Omega_\Sigma$ arbitrarily heterogeneous. Cross-points (red dots) are allowed.}
\label{fig:geo}
\end{figure}
Except for those related to domain decomposition, many of the contributions dedicated to FEM-BEM coupling consider a simple geometric configuration where the computational domain is subdivided into two parts separated by a single interface: one interior heterogeneous part and one exterior homogeneous part. Multi-domain configurations with more than two subdomains are also of interest, and often involve the presence of \emph{cross-points}, i.e.~points where three subdomains or more are adjacent (see for instance the red points in Figure~\ref{fig:geo}). From a numerical standpoint, as was clearly shown in \cite[\S 4]{MR4106670} by detailed numerical examples, a careless treatment of cross-points may lead to a lack of consistency of standard linear solvers such as GMRes. At the continuous level, the presence of cross-points is problematic because in that case the interface shared by one subdomain with another can have a boundary (made of cross-points). So, the operators giving the restriction to the interface (between Dirichlet or Neumann trace spaces on the subdomain boundary) are not continuous, see e.g.~\cite[§6.2]{MR3202533}. 
This prevents writing in a proper function space framework the most natural multi-domain formulations that would use restriction operators. To avoid these, in the present contribution, we design and analyze new multi-domain FEM-BEM formulations by means of the Multi-Trace Formalism (MTF), which was introduced in \cite{ClHi:MTF:2013,MR3202533,ClHi:impenetrable:2015} for piece-wise constant coefficients. Indeed, MTF allows for a clean treatment of cross-points from the perspective of function spaces, and proves here to be perfectly fitted to the Costabel coupling. These new formulations satisfy G\r{a}rding inequalities, which, in case of injectivity, imply stability and quasi-optimal convergence results of conforming discretization methods. 

Unfortunately, like the classical Costabel coupling, also its multi-domain versions may be affected by the spurious resonances phenomenon, that is, the associated operator may be not injective, whereas the corresponding transmission problem is always well-posed. Therefore, we identify conditions for injectivity, and then we construct modified versions immune to spurious resonances. This is a generalization to multi-domain configurations of the strategy studied in \cite{HiMe:sfembem:2006} for the two-domain case. 

This article is organized as follows. First, we present the acoustic scattering transmission problem in Section~\ref{sec:setting}, we recall the definitions of trace spaces and operators in Section~\ref{TraceSpaceSingleDomain}, and classical results of potential and boundary integral operator theory in Section~\ref{sec:BIEreview}. Then in Section~\ref{sec:spaces} we introduce a functional setting suited for the multi-domain configuration. After revisiting the classical Costabel coupling in Section~\ref{sec:2costabel} for two subdomains, in Section~\ref{sec:STF} we propose a first multi-domain coupling formulation, called single-trace FEM-BEM formulation, followed by a combined field version in Section~\ref{sec:STFcomb} that is immune to spurious resonances. The single-trace FEM-BEM formulation is preparatory to the more flexible multi-trace FEM-BEM formulation, which is derived and analyzed in Section~\ref{sec:MTF}. Finally, a multi-trace combined field FEM-BEM formulation is designed in Section~\ref{sec:MTFcomb}. 


\nomenclature[A, 01]{$\Omega_j$}{Subdomains of $\RR^d$ with homogeneous medium (with $\Omega_0$ unbounded)}
\nomenclature[A, 02]{$n$}{Number of bounded homogeneous subdomains}
\nomenclature[A, 03]{$\Omega_\Sigma$}{Subdomain of $\RR^d$ with heterogeneous medium}
\nomenclature[A, 04]{$\Gamma_j$}{Homogeneous subdomain boundary $\partial \Omega_j$}
\nomenclature[A, 05]{$\Sigma$}{Heterogeneous subdomain boundary $\partial \Omega_\Sigma$}
\nomenclature[A, 06]{$\Gamma$}{The skeleton, that is the union of subdomain interfaces, see \eqref{eq:boundaries}}
\nomenclature[A, 07]{$\kappa_j$}{Wavenumber in $\Omega_j$ (positive constant)}
\nomenclature[A, 08]{$\kappa_\Sigma$}{Wavenumber in $\Omega_\Sigma$ (positive function)}

\nomenclature[D, 09]{$\gamma^{\Omega}_{\dir}, \gamma^{\Omega}_{\neu}$}{Interior Dirichlet and Neumann trace operators on $\partial \Omega$, denoted $\gamma^{j}_{\dir}, \gamma^{j}_{\neu}$ for $\Omega \equiv \Omega_j$}
\nomenclature[D, 10]{$\gamma^{\Omega}_{\dir,c},\gamma^{\Omega}_{\neu,c}$}{Exterior Dirichlet and Neumann trace operators on $\partial \Omega$, denoted $\gamma^{j}_{\dir,c},\gamma^{j}_{\neu,c}$ for $\Omega \equiv \Omega_j$}
\nomenclature[D, 11]{$\gamma^{\Omega}, \gamma^{\Omega}_{c}$}{Interior and exterior pairs of Dirichlet and Neumann trace operators on $\partial \Omega$, denoted $\gamma^{j}, \gamma^{j}_{c}$ for $\Omega \equiv \Omega_j$}
\nomenclature[B, 12]{$\mbH(\partial\Omega)$}{Space of pairs of Dirichlet and Neumann traces on $\partial \Omega$, see \eqref{DefPairOfTraces}}
\nomenclature[C, 13]{$\langle \,\cdot\,, \,\cdot\, \rangle_{\partial\Omega}$}{Duality pairing between Dirichlet and Neumann traces on $\partial\Omega$}
\nomenclature[C, 14]{$\lbr \,\cdot\,, \,\cdot\, \rbr_{\partial\Omega}$}{Self-duality pairing on $\mbH(\partial\Omega)$, see \eqref{eq:skewpairingloc}}
\nomenclature[D, 16]{$\lbr \gamma^{\Omega}\rbr$}{Jump of the interior and exterior trace operators across $\partial\Omega$, see \eqref{eq:jumpaverage}}
\nomenclature[D, 17]{$\{ \gamma^{\Omega}\}$}{Average of the interior and exterior trace operators across $\partial\Omega$, see \eqref{eq:jumpaverage}}
\nomenclature[E, 18]{$\mathsf{SL}_\kappa^\Omega$}{Single layer potential on $\partial\Omega$ ($\kappa$ constant wavenumber), see \eqref{eq:SLpot}, denoted $\mathsf{SL}_\kappa^j$ for $\Omega \equiv \Omega_j$}
\nomenclature[E, 19]{$\mathsf{DL}_\kappa^\Omega$}{Double layer potential on $\partial\Omega$ ($\kappa$ constant wavenumber), see \eqref{eq:DLpot}, denoted $\mathsf{DL}_\kappa^j$ for $\Omega \equiv \Omega_j$}
\nomenclature[E, 20]{$\mathsf{G}_\kappa^\Omega$}{Total potential on $\partial\Omega$ ($\kappa$ constant wavenumber), see \eqref{eq:totalpot}, denoted $\mathsf{G}_\kappa^j$ for $\Omega \equiv \Omega_j$}
\nomenclature[E, 21]{$\mathsf{A}_\kappa^\Omega$}{$2\times 2$ matrix of boundary integral operators (double layer, single layer, hypersingular and adjoint double layer operators), see \eqref{DefPMCHWTOp}, denoted $\mathsf{A}_\kappa^j$ for $\Omega \equiv \Omega_j$}

\nomenclature[E, 21a]{$\mathsf{A}$}{Block diagonal operator $\mathsf{A} \coloneqq \mathrm{diag}(\mathsf{A}^{0}_{\kappa_0},\dots,\mathsf{A}^{n}_{\kappa_n})$}
\nomenclature[B, 22]{$\mbH(\Gamma)$}{Multi-trace space: $\mbH(\Gamma) \coloneqq \mbH(\Gamma_0) \times \cdots \times \mbH(\Gamma_n)$, see \eqref{MultiTraceSpace}}
\nomenclature[D, 23]{$\gamma$}{Global trace operator defined in \eqref{GlobalTraceOperator}}
\nomenclature[C, 24]{$\lbr \,\cdot\,, \,\cdot\, \rbr_{\Gamma}$}{Self-duality pairing on $\mbH(\Gamma)$, see \eqref{eq:pairingHGamma}}
\nomenclature[B, 25]{$\mbX(\Gamma)$}{Single-trace space (subspace of $\mbH(\Gamma)$), see \eqref{SingleTraceSpace}}
\nomenclature[D, 26]{$\Tr, \Tr_\dir, \Tr_\neu$}{Traces on $\Sigma$ induced by a tuple in $\mbX(\Gamma)$, see Proposition~\ref{DefinitionTraceOnSigma}}
\nomenclature[B, 27]{$\widetilde{\mbX}(\Gamma)$}{Single-trace space with additional components on $\Sigma$, see \eqref{eq:defXtilde}}
\nomenclature[B, 28]{$\XSigmaGamma$}{Subspace of $\mH^{1}(\Omega_{\Sigma})\times\mbX(\Gamma)$ with Dirichlet conditions on $\Sigma$, see \eqref{eq:XsigmaGamma}}
\nomenclature[B, 29]{$\XSigmaGammaM$}{Subspace of $\mH^{1}(\Omega_{\Sigma})\times\mbX(\Gamma)$ with generalized Robin conditions on $\Sigma$, see \eqref{DefEspaceRobin}}

\nomenclature[E, 30]{$a_\Sigma$}{Helmholtz bilinear form on $\Omega_\Sigma$, see \eqref{eq:HelmOmegaSigma}}
\nomenclature[E, 31]{$F_\Sigma$}{Linear form for the source term on $\Omega_\Sigma$, see \eqref{eq:HelmOmegaSigma}}

\nomenclature[E, 50]{$\theta$}{Operator $\theta(v,q) \coloneqq (-v,q)$, for $(v,q) \in \mbH(\partial \Omega)$}
\nomenclature[E, 51]{$\Theta$}{Operator $\Theta(\frv) \coloneqq (\theta(\frv_0),\dots,\theta(\frv_n))$, for $\frv = (\frv_0,\dots,\frv_n)\in \mbH(\Gamma)$}

\nomenclature[B, 52]{$\widehat{\mbH}(\Gamma)$}{Multi-trace space with Neumann traces on $\Sigma$ and no components on $\Gamma_0$, see \eqref{eq:MTspaceNeu}}
\nomenclature[B, 53]{$\widecheck{\mbH}(\Gamma)$}{Multi-trace space with  Dirichlet traces on $\Sigma$ and no components on $\Gamma_0$}
\nomenclature[C, 54]{$\llbracket \,\cdot\,, \,\cdot\, \rrbracket$}{Duality pairing between $\widehat{\mbH}(\Gamma)$ and $\widecheck{\mbH}(\Gamma)$, see \eqref{eq:pairingHatCheck}}
\nomenclature[B, 55]{$\wdoublehat{\mbH}(\Gamma)$}{Multi-trace space with both Dirichlet and Neumann traces on $\Sigma$, and no components on $\Gamma_0$, see \eqref{eq:MTspaceFullSigma}}
\nomenclature[C, 56]{$\lBrace \,\cdot\,, \,\cdot\, \rBrace$}{Self-duality pairing on $\wdoublehat{\mbH}(\Gamma)$, see \eqref{eq:pairingFullSigma}}
\nomenclature[E, 21b]{$\wdoublehat{\mathsf{A}}$}{Full block operator, see \eqref{eq:defAhathat}}

\printnomenclature[1.8cm]

\section{The transmission problem}
\label{sec:setting}

\noindent 
We start by presenting the problem under study. We consider a non-overlapping domain decomposition
\begin{equation}\label{SubdomainPartition}
  \RR^d = \bigcup_{j=0}^n \overline\Omega_j \, \cup \,\overline\Omega_\Sigma,
\end{equation}
where each subdomain will only be assumed Lipschitz regular \cite[Def.~3.28]{McLean:book:2000} and connected, and all subdomains except  $\Omega_0$ are bounded.
In addition, $\RR^d\setminus\overline{\Omega}_\Sigma$ will also be assumed connected (so that $\Omega_\Sigma$ does not contain any hole). An example of such a configuration is given in Figure~\ref{fig:geo}. We emphasize that in such a geometrical setting the presence of cross-points (red points in Figure~\ref{fig:geo}) is allowed.

We consider a propagation medium whose effective wavenumber, described by a function $\kappa\colon\RR^d \to \RR_+$, varies in accordance with the subdomain decomposition in \eqref{SubdomainPartition}: we assume that
\begin{equation*}
    \kappa(\bx) = \kappa_j\;\;\forall\bx\in \Omega_j, \,j=0,\dots, n, \quad
    \text{with}\; \kappa_j\in(0,+\infty),
\end{equation*}
while in the subdomain $\Omega_\Sigma$ the wavenumber is not assumed constant and may vary: $\kappa(\bx) = \kappa_{\Sigma}(\bx)$, with $\kappa_{\Sigma}(\bx)>0, \;\forall\bx\in\Omega_\Sigma$.

Let the incident field $U_\inc\in \mH^{1}_{\mrm{loc}}(\RR^d)$ satisfy $\Delta U_{\inc}+\kappa_0^2U_{\inc} = 0$ in $\RR^d$, where $\mH^{1}_{\mrm{loc}}(\RR^d)$ is the set of functions whose restriction to any compact set $\omega \subset \RR^d$ belongs to $\mH^{1}(\omega)$. Let the source term $f\in \mL^2(\RR^d)$ be supported in $\Omega_{\Sigma}$. We are interested in solving the following problem modelling an acoustic wave propagating in a heterogeneous medium
\begin{equation}\label{InitBVP}
  \begin{aligned}
    & \text{Find}\; U\in \mH^1_{\loc}(\RR^d)\;\text{such that}\\
    & \Delta U + \kappa(\bx)^2 U = -f\quad \text{in}\;\RR^d\\
    & U-U_{\inc}\;\text{is $\kappa_0$-outgoing radiating.}
  \end{aligned}
\end{equation}
In this problem, the third condition is the classical Sommerfeld radiation condition, see e.g.~\cite[§2.6.5]{Nedelec:book:2001}: any function $V$ is said to be $k$-outgoing radiating if
$\lim_{\rho\to \infty} \int_{\partial B_\rho} \lvert \partial_\rho V - \imagi k V \rvert ^2 \,d\sigma_\rho = 0$, 
where $\imagi = \sqrt{-1}$, $B_\rho$ is the ball centered at the origin of radius $\rho$ and $\partial_\rho$ denotes the radial derivative. 
By standard results of scattering theory, Problem \eqref{InitBVP} admits a unique solution, see e.g.~\cite[Theorem~8.7]{ColtonKress2013IAE}.

To solve such a problem, a standard numerical approach would rely on finite elements. 
The computational efficiency could be improved by taking advantage of the piece-wise constant material characteristics in the subdomains $\Omega_j$. 
In the present contribution, we wish to develop a multi-domain FEM-BEM coupling strategy, where the wave equation is treated by means of boundary integral operators in those parts of the computational domain where material characteristics are constant.
Compared to most of the existing literature on FEM-BEM coupling, an important novelty in the present contribution lies in providing a rigorous analysis also in the presence of cross-points. 

Let us introduce notations for boundaries and interfaces:
\begin{equation}\label{eq:boundaries}
  \Gamma_j \coloneqq\partial\Omega_j, \, j=0,\dots,n, \quad \Sigma\coloneqq\partial\Omega_{\Sigma},\quad \Gamma\coloneqq\cup_{j=0}^{n}\Gamma_j \;\, \text{(the ``skeleton'')}.
\end{equation}
Note that $\Sigma \subset \Gamma$ because each point of $\Sigma$ belongs also to some $\Gamma_j$, $j=0,\dots,n$.
The first step toward a multi-domain FEM-BEM formulation of problem \eqref{InitBVP} consists in decomposing the wave equation according to \eqref{SubdomainPartition}, and imposing transmission conditions at interfaces:
\begin{equation}\label{eq:bvp}
  \begin{aligned}
    & \Delta U + \kappa_{\Sigma}^2(\bx)U = -f\quad \text{in}\;\Omega_{\Sigma}\\
    & \Delta U + \kappa_{j}^2 U = 0\quad \text{in}\;\Omega_{j}\\[8pt]
    & U\vert_{\Gamma_j} - U\vert_{\Gamma_k} = 0\\
    & \partial_{\bn_j}U\vert_{\Gamma_j} + \partial_{\bn_k}U\vert_{\Gamma_k} = 0\\[8pt]
    & U\vert_{\Gamma_j} - U\vert_{\Sigma} = 0\\
    & \partial_{\bn_j}U\vert_{\Gamma_j} + \partial_{\bn_\Sigma}U\vert_{\Sigma} = 0\\[8pt]
    & U-U_{\inc}\;\text{is $\kappa_0$-outgoing radiating}.
  \end{aligned}
\end{equation}
Here, all traces are taken from the interior of subdomains, and $\bn_j, j=0\dots n$ (resp.~$\bn_{\Sigma}$) are the unit normal vector fields on $\Gamma_j$ directed toward the exterior of $\Omega_j$ (resp.~$\Omega_\Sigma$).  
 Neumann traces are defined by $\partial_{\bn_j}U\vert_{\Gamma_j}\coloneqq\bn_{j}\cdot\nabla U\vert_{\Gamma_j}$
(resp. $\partial_{\bn_{\Sigma}}U\vert_{\Sigma}\coloneqq\bn_{\Sigma}\cdot\nabla U\vert_{\Sigma}$).

\section{Trace spaces and operators}
\label{TraceSpaceSingleDomain}

\noindent 
Discussing transmission conditions requires paying thorough attention to function spaces, trace spaces and operators. In all this section, $\Omega$ refers to a generic Lipschitz domain that is either bounded or such that $\RR^d\setminus\overline{\Omega}$ is bounded, and $\bn_\Omega$ is the unit normal vector field on $\partial\Omega$ systematically directed toward the exterior of $\Omega$. 

First of all, we use classical notations for the following elementary function spaces of volume functions:
\begin{equation}\label{SobolevSpaces}
  \begin{aligned}
    \mH^1(\Omega)
    & \coloneqq \set{ V \in \mL^2(\Omega) |
      \nabla V \in \mL^2(\Omega) },\\
    \mH(\textup{div},\Omega)
    & \coloneqq \set{ \bv \in
    \mL^2(\Omega)^d | \textup{div}(\bv) \in
    \mL^2(\Omega) },\\
    \mH^1(\Delta,\Omega) 
    &  \coloneqq \set{ V \in
      \mH^1(\Omega) | \Delta V \in \mL^2(\Omega) }.
  \end{aligned}
\end{equation}
They are equipped with their canonical norms
$\Vert V \Vert_{\mH^1(\Omega)}^2 \coloneqq \Vert V
\Vert_{\mL^2(\Omega)}^2 + \Vert \nabla V
\Vert_{\mL^2(\Omega)}^2$, 
$\Vert \bv\Vert_{\mH(\textup{div},\Omega)}^2 \coloneqq \Vert
\bv \Vert_{\mL^2(\Omega)}^2 + \Vert
\textup{div}(\bv) \Vert_{\mL^2(\Omega)}^2$, and
$\lVert V \rVert_{\mH^1(\Delta,\Omega)}^2 \coloneqq \lVert V
\rVert_{\mH^1(\Omega)}^2 + \lVert \Delta V
\rVert_{\mL^2(\Omega)}^2$. With these norms, the spaces \eqref{SobolevSpaces} admit a Hilbert structure.
If $\mH(\Omega)$ is any of the spaces above, we set
$\mH_\loc(\overline \Omega) \coloneqq 
\{V \mid \varphi V \in \mH(\Omega) \; \forall \varphi \in \mathscr{C}^\infty_{\comp}(\RR^d)\}$, 
where $\mathscr{C}^\infty_{\comp}(\RR^d)$ is the space of $\mathscr{C}^\infty$ functions with compact support.

We introduce the \emph{interior Dirichlet trace operator}  $\gamma_{\dir}^{\Omega}$
and the \emph{interior Neumann trace operator} $\gamma_{\neu}^{\Omega}$,
defined for smooth functions $ \varphi\in\mathscr{C}^{\infty}(\RR^d)$ by 
\begin{equation*}
 \gamma_{\dir}^{\Omega}(\varphi)\coloneqq\varphi\vert_{\partial\Omega},
 \quad \gamma_{\neu}^{\Omega}(\varphi)\coloneqq\bn_{\Omega}\cdot\nabla \varphi\vert_{\partial\Omega}.
\end{equation*}
These definitions are extended by density and continuity to trace operators $\gamma_{\dir}^{\Omega}: \mH^{1}_{\loc}(\overline{\Omega})\to \mH^{1/2}(\partial\Omega)$, $\gamma_{\neu}^{\Omega}\colon \mH^{1}_{\loc}(\Delta,\overline{\Omega})\to \mH^{-1/2}(\partial\Omega)$,  
where the Dirichlet trace space $\mH^{1/2}(\partial\Omega)$ is defined as the completion of $\{\varphi\vert_{\partial\Omega}, \varphi\in\mathscr{C}^{\infty}(\RR^d)\}$ with respect to the Slobodeckii norm (see e.g.~\cite[Chap.~2]{McLean:book:2000})
\begin{equation*}
  \Vert \varphi\Vert_{\mH^{1/2}(\partial\Omega)}^{2}\coloneqq\int_{\partial\Omega\times\partial\Omega}
  \frac{\vert \varphi(\bx) -  \varphi(\by)\vert^{2}}{\vert \bx-\by\vert^{d}} d\sigma(\bx,\by),
\end{equation*}
and the Neumann trace space $\mH^{-1/2}(\partial\Omega)$ is the dual space of $\mH^{1/2}(\partial\Omega)$.
The corresponding duality pairing will be denoted by 
$\langle p,v\rangle_{\partial\Omega} \equiv \langle v,p\rangle_{\partial\Omega} \coloneqq p(v)$ for $v\in \mH^{1/2}(\partial\Omega)$ and $p\in \mH^{-1/2}(\partial\Omega)$,
and we shall take
\begin{equation*}
  \Vert p\Vert_{\mH^{-1/2}(\partial\Omega)}\coloneqq\sup_{v\in \mH^{1/2}(\partial\Omega)\setminus\{0\}}
  \frac{\vert \langle  p,v\rangle_{\partial\Omega}\vert}{\Vert v\Vert_{\mH^{1/2}(\partial\Omega)}}
\end{equation*}
as norm for the Neumann trace space.
We also introduce operators and spaces for \emph{pairs} of Dirichlet and Neumann traces, defined by $\gamma^{\Omega}(V)\coloneqq(\gamma_{\dir}^{\Omega}(V),\gamma_{\neu}^{\Omega}(V))$
and 
\begin{equation}\label{DefPairOfTraces}
  \begin{aligned}
    & \gamma^{\Omega}\coloneqq (\gamma_{\dir}^{\Omega},\gamma_{\neu}^{\Omega})\colon \mH^{1}(\Delta,\overline{\Omega})\to \mbH(\partial\Omega)
    \quad \text{where}\\
    & \mbH(\partial\Omega)\coloneqq \mH^{1/2}(\partial\Omega)\times \mH^{-1/2}(\partial\Omega).
  \end{aligned}
\end{equation}
In contrast with Dirichlet and Neumann trace operators $\gamma_{\dir}^{\Omega},\gamma_{\neu}^{\Omega}$, the trace operator $\gamma^{\Omega}$ is not really standard, but we shall often use it for compact notation in our analysis. 
The space of pairs of Dirichlet-Neumann traces $\mbH(\partial\Omega)$ will be equipped with the Cartesian product norm $\Vert (v,q)\Vert_{\mbH(\partial\Omega)}^{2}\coloneqq
\Vert v\Vert_{\mH^{1/2}(\partial\Omega)}^{2}+\Vert q\Vert_{\mH^{-1/2}(\partial\Omega)}^{2}$. It is put in duality with itself through the following skew-symmetric bilinear pairing
\begin{equation}\label{eq:skewpairingloc}
  \lbr (u,p), (v,q)\rbr_{\partial\Omega}\coloneqq\langle u,q\rangle_{\partial\Omega} -  \langle p,v\rangle_{\partial\Omega}
\end{equation}
for all $(u,p),(v,q)\in \mbH(\partial\Omega)$. We underline that no complex conjugation comes into play in this definition. 
Note that throughout the paper Dirichlet traces are denoted by $u, v, w$ and Neumann traces by $p, q, r$, while capital letters like $U, V$ are used to indicate scalar functions on volume domains, and small bold letters like $\bv, \bp, \bq$ are used for vector fields. In this section and the following one, we use gothic symbols $\fru, \frv, \frw$ to denote pairs of Dirichlet-Neumann traces, that is elements of $\mbH(\partial\Omega)$. We have the inequality
$\vert  \lbr \fru,\frv\rbr_{\partial\Omega}\vert\leq \Vert \fru\Vert_{\mbH(\partial\Omega)} \Vert \frv\Vert_{\mbH(\partial\Omega)}$ for all $\fru,\frv\in \mbH(\partial\Omega)$. 

Setting $\theta(v, q) \coloneqq (-v, q)$, we state simple identities that will be used several times in the following: for all $\fru=(u,p)$, $\frv=(v,q) \in \mbH(\partial\Omega)$
\begin{align}
  & \lbr \fru,\theta(\frv)\rbr_{\partial\Omega} =
  \langle u,q\rangle_{\partial\Omega}+\langle p,v\rangle_{\partial\Omega}
  \label{eq:thetaIdentitybis3},\\
  &  \lbr \fru,\theta(\frv)\rbr_{\partial\Omega} 
  -  \lbr \fru, \frv\rbr_{\partial\Omega} = 
  2 \langle p,v\rangle_{\partial\Omega},\label{eq:thetaIdentity}\\
  & \lbr \fru,\theta(\frv)\rbr_{\partial\Omega} 
  + \lbr \fru, \frv\rbr_{\partial\Omega} = 
  2\langle u,q\rangle_{\partial\Omega}.\label{eq:thetaIdentitybis}
\end{align}

Together with the operators $\gamma^{\Omega}_{\dir}, \gamma^{\Omega}_{\neu}, \gamma^{\Omega}$, for which traces are taken from the interior of the domain $\Omega$, similar operators can be defined for traces taken from the \emph{exterior} of $\Omega$, and will be denoted by
\begin{equation*}
  \begin{aligned}
    & \gamma^{\Omega}_{\dir,c}\colon \mH^{1}_{\loc}(\RR^d\setminus\Omega)\to \mH^{1/2}(\partial\Omega),\\
    & \gamma^{\Omega}_{\neu,c}\colon \mH^{1}_{\loc}(\Delta,\RR^d\setminus\Omega)\to \mH^{-1/2}(\partial\Omega),\\
    & \gamma^{\Omega}_{c}\coloneqq(\gamma^{\Omega}_{\dir,c},\gamma^{\Omega}_{\neu,c})\colon\mH^{1}_{\loc}(\Delta,\RR^d\setminus\Omega)\to \mbH(\partial\Omega).
  \end{aligned}
\end{equation*}
When considering the trace operator $\gamma^{\Omega}_{\neu,c}$, the normal vector is still directed toward the exterior of $\Omega$.
Finally, we will also need \emph{jump and average traces}: 
\begin{equation}
\label{eq:jumpaverage}
\lbr \gamma^{\Omega}\rbr\coloneqq \gamma^{\Omega} - \gamma^{\Omega}_c, \qquad
\{\gamma^{\Omega}\}\coloneqq (\gamma^{\Omega}+\gamma^{\Omega}_c)/2.
\end{equation}

In the context of the multi-domain configuration \eqref{eq:boundaries}, for the sake of brevity, we shall write $\gamma^{j}_{\dir}$ (resp.~$\gamma^{j}_{\neu},\gamma^{j}, \gamma^{j}_{\dir,c},\gamma^{j}_{\neu,c},\gamma^{j}_{c}$) 
instead of $\gamma^{\Omega_{j}}_{\dir}$ (resp.~$\gamma^{\Omega_{j}}_{\neu},\gamma^{\Omega_{j}},\gamma^{\Omega_{j}}_{\dir,c}, \gamma^{\Omega_{j}}_{\neu,c},\gamma^{\Omega_{j}}_{c}$). We shall adopt a similar convention for traces on $\Sigma$, writing $\gamma^{\Sigma}_{*}$ instead of
 $\gamma^{\Omega_{\Sigma}}_{*}$ with $* = \dir,\neu$ and so on.

\section{Review of potential and boundary integral operators}
\label{sec:BIEreview}

\noindent 
In this section, we recall, using compact notation, classical results about boundary integral formulations for the Helmholtz equation in Lipschitz domains. For more details and proofs see for instance \cite[Chap.~3]{SaSw:book:2011}.
As in the previous section, here $\Omega$ denotes a generic Lipschitz domain, which is either bounded or the complement of a bounded domain. 

Let the function $\mathcal{G}_\kappa\colon \RR^d\setminus\{0\}\to \CC$ be the $\kappa$-outgoing radiating fundamental solution or Green kernel for the Helmholtz operator $-\Delta - \kappa^2$, for a given constant wavenumber $\kappa \in(0,+\infty)$. In particular for $\RR^d = \RR^3$ we have $\mathcal{G}_\kappa(\bx) = \exp(i\kappa\vert \bx\vert)/(4\pi\vert\bx\vert)$.
For any $\bx\in \RR^d\setminus \partial\Omega$, and any $\frv = (v,q)\in \mbH(\partial\Omega)$, define \emph{potential operators}\footnote{Note that the choice of sign in the double layer potential differs from the one usually adopted in the literature, in order to maintain symmetry in the definition of $\mathsf{G}_\kappa^\Omega$ (and consequently in the representation formula).} 
\begin{align}
\mathsf{SL}_\kappa^\Omega(q)(\bx) &\coloneqq \int_{\partial \Omega}   q(\by) \; \mathcal{G}_\kappa(\bx-\by) \,d \sigma(\by),\label{eq:SLpot} \displaybreak[0] \\
\mathsf{DL}_\kappa^\Omega(v)(\bx) &\coloneqq  \int_{\partial \Omega}   v(\by) \; \bn_\Omega(\by) \cdot
(\nabla \mathcal{G}_\kappa)(\bx-\by) \,d \sigma(\by) \label{eq:DLpot}\\
&= - \int_{\partial \Omega}   v(\by) \; \bn_\Omega(\by) \cdot \nabla_{\by}(\mathcal{G}_\kappa(\bx-\by)) \,d \sigma(\by), \notag \displaybreak[0] \\
\mathsf{G}_\kappa^{\Omega} (\frv)(\bx) &\coloneqq \mathsf{DL}_\kappa^{\Omega}(v)(\bx) + \mathsf{SL}_\kappa^{\Omega}(q)(\bx), \label{eq:totalpot}
\end{align}
where the first two operators are called single and double layer potentials. 
The total potential $\mathsf{G}_\kappa^\Omega$ maps continuously $\mbH(\partial \Omega)$ into\footnote{Here we consider that 
$V \in \mH^1_\loc(\Delta, \overline\Omega) \times
\mH^1_\loc(\Delta,\RR^d \backslash\Omega)$ 
if and only if 
$V|_{\Omega} \in \mH^1_\loc(\Delta, \overline\Omega)$ and $V|_{\RR^d \backslash \overline\Omega} \in \mH^1_\loc(\Delta,  \RR^d \backslash\Omega)$.} 
$\mH^1_\loc(\Delta, \overline\Omega) \times \mH^1_\loc(\Delta,  \RR^d \backslash\Omega)$
(see \cite[Thm.~3.1.16]{SaSw:book:2011}), so that the traces of $\mathsf{G}_\kappa^\Omega(\frv)$ are properly defined. 
This operator can be used to write a representation formula for the solution to the homogeneous Helmholtz equation in terms of the Dirichlet and Neumann traces of the solution (see \cite[Thm.~3.1.6]{SaSw:book:2011}): 
\begin{proposition}[Representation formulas]
Let $U \in \mH^1_\loc(\overline\Omega)$ satisfy $-\Delta U - \kappa^2U = 0$ in $\Omega$.
If $\Omega$ is unbounded, assume in addition that $U$ is $\kappa$-outgoing radiating. Then we have the representation formula
\begin{equation}
\label{eq:intRep}
\mathsf{G}_\kappa^{\Omega} (\gamma^\Omega(U))(\bx)= 1_{\Omega}(\bx)U(\bx).\quad\quad
\end{equation} 
Similarly, let $V \in \mH^1_\loc(\RR^d \backslash \Omega)$ satisfy $-\Delta V - \kappa^2V = 0$ in $\RR^d \backslash \overline \Omega$, as well as the Sommerfeld radiation condition if $\Omega$ is bounded.
Then we have 
\begin{equation}
\label{eq:extRep}
\mathsf{G}_\kappa^{\Omega} (\gamma_c^{\Omega}(V))(\bx)= - 1_{\RR^{d}\setminus\overline{\Omega}}(\bx)V(\bx).
\end{equation} 
\end{proposition}

\smallskip
\noindent Here, $1_{\Omega}$ (resp.~$1_{\RR^{d}\setminus\overline{\Omega}}$) is the characteristic function of $\Omega$ (resp.~$\RR^{d}\setminus\overline{\Omega}$). In addition to the representation formulas above, the potential operator $\mathsf{G}_\kappa^{\Omega}$ satisfies the so-called \emph{jump relations} \cite[Thm.~3.3.1]{SaSw:book:2011}, which describe the relationship between interior and exterior traces of $\mathsf{G}_\kappa^{\Omega}$. Here we express these relations through the following synthetic identity 
\begin{equation}
\label{eq:jumprels}
[\gamma^{\Omega}] \circ \mathsf{G}_\kappa^{\Omega} = \Id,
\end{equation} 
where $\Id$ is the identity map on $\mbH(\partial\Omega)$ and the jump $[\,\cdot\,]$ is defined in \eqref{eq:jumpaverage}. 

Any $U = \mathsf{G}_\kappa^{\Omega}(\fru)$ for $\fru\in \mbH(\partial \Omega)$ is a $\kappa$-outgoing radiating  solution to the homogeneous Helmholtz equation in $\Omega$ with wavenumber $\kappa$, hence we can apply to it the representation formula~\eqref{eq:intRep}. Taking the interior traces of
this formula leads to $\gamma^\Omega \circ \mathsf{G}_\kappa^\Omega(\gamma^\Omega\circ\mathsf{G}_\kappa^{\Omega}(\fru)) =
\gamma^\Omega\circ\mathsf{G}_\kappa^{\Omega}(\fru)$, and since $\fru$ was chosen arbitrarily in $\mbH(\partial \Omega)$, this finally rewrites 
\begin{equation}
\label{eq:CaldId}
(\gamma^\Omega \circ \mathsf{G}_\kappa^\Omega)^2 = (\gamma^\Omega \circ \mathsf{G}_\kappa^\Omega)
\end{equation}
which is a synthetic form of the four classical \emph{interior Cald\'eron identities}.
The operator $\gamma^\Omega \circ \mathsf{G}_\kappa^\Omega$ is a continuous projector, called the \emph{interior Calder\'on projector} of $\Omega$.  This actually provides a characterization of traces of solutions to the homogeneous Helmholtz equation, which are called Cauchy data (see \cite[§3.6]{SaSw:book:2011}): 
\begin{proposition}[Definition and characterization of Cauchy data]
\label{pr:chCauchy}
We define the space of \emph{Cauchy data} of $\Omega$
\begin{equation}
  \begin{aligned}
    \mathcal{C}_\kappa(\Omega) \coloneqq 
    & \{\gamma^\Omega(U)\in\mbH(\partial\Omega) \;|\;
     U \in \mH^1_\loc(\overline\Omega),
    -\Delta U - \kappa^2U = 0 \text{ in } \Omega,\\
    & \text{ and $U$ is $\kappa$-outgoing radiating if $\Omega$ is unbounded}\;\}.
  \end{aligned}
\end{equation}
The range of the {interior Calder\'on projector} $\gamma^\Omega \circ \mathsf{G}_\kappa^\Omega$ coincides with $\mathcal{C}_\kappa(\Omega)$. More precisely, for any $\fru \in \mbH(\partial \Omega)$ we have
$\gamma^\Omega \circ \mathsf{G}_\kappa^\Omega (\fru) = \fru
\iff\fru \in \mathcal{C}_\kappa(\Omega)$.
\end{proposition}
\noindent Analogous results, obtained by taking exterior traces of the representation
formula~\eqref{eq:extRep}, hold for exterior Cauchy data.  

Applying traces to potential operators yields \emph{boundary integral operators}: in our compact notation we will use 
\begin{equation}\label{DefPMCHWTOp}
\mathsf{A}_\kappa^\Omega \coloneqq \{\gamma^\Omega\} \circ \mathsf{G}_\kappa^\Omega, 
\end{equation}
where the average $\{\cdot\}$ is defined in \eqref{eq:jumpaverage}. 
The operator $\mathsf{A}_\kappa^\Omega$ continuously maps
$\mbH(\partial\Omega)$ into $\mbH(\partial\Omega)$.  It consists in a
$2\times 2$ matrix of boundary integral operators (double layer,
single layer, hypersingular and adjoint double layer operators, see
e.g.~\cite[\S 3.6]{SaSw:book:2011}). In this article, we shall not need to refer individually to any of its entries.  Simple consequences of the jump
relations~\eqref{eq:jumprels} are
\begin{align}
& \gamma^\Omega \circ \mathsf{G}_\kappa^\Omega = \mathsf{A}_\kappa^\Omega + \Id/2, \label{eq:Akj12} \\
& \gamma_c^\Omega \circ \mathsf{G}_\kappa^\Omega = \mathsf{A}_\kappa^\Omega - \Id/2. \label{eq:Akj12bis} 
\end{align}
So, identity \eqref{eq:CaldId} implies $(\mathsf{A}_\kappa^\Omega)^{2} = \Id/4$. The operator $\mathsf{A}_\kappa^\Omega$, for $\Omega = \Omega_j$, $j=0,\dots,n$, will play a pivotal role in our analysis.  We now
recall a few properties of $\mathsf{A}_\kappa^\Omega$, which are well established in the literature. First, this operator satisfies a generalized G\r{a}rding inequality:
\begin{proposition}[Generalized G\r{a}rding inequality]
\label{pr:gardingAj}
Recall the operator $\theta(v, q) \coloneqq (-v, q)$. There exist a compact operator $\mathcal{K}\colon \mbH(\partial \Omega) \to \mbH(\partial \Omega)$ and a constant $\alpha>0$ such that for all $ \fru \in \mbH(\partial \Omega)$ we have
\[
\Re \left \{ [(\mathsf{A}_\kappa^\Omega + \mathcal{K})\fru, \theta(\overline{\fru})]_{\partial\Omega} \right \}
\ge \alpha \lVert \fru \rVert_{\mbH(\partial \Omega)}^2.
\] 
\end{proposition}
\noindent
Although well known (see for example \cite[Thm.~3.9]{Pet:STF:1989}), the proof of this result is instructive, so we include it in Proposition~\ref{pr:gardingAjAppendix} in the appendix. 
Next, remarkable \emph{symmetry properties} were proved in \cite[Lemma 3.6--3.7]{ClHi:impenetrable:2015}: for any $ \fru, \frv \in \mbH(\partial \Omega)$
we have 
\begin{equation*}
      [\mathsf{A}_\kappa^{\Omega}(\fru), \frv]_{\partial\Omega} =
      [\mathsf{A}_\kappa^{\Omega}(\frv), \fru]_{\partial\Omega}.
\end{equation*}
Finally, we recall a useful result about the sign of the imaginary part of the quadratic form $\fru\mapsto [\mathsf{A}_\kappa^{\Omega}(\fru),\overline{\fru}]_{\partial\Omega}$:
\begin{proposition}\label{RadiationConditionConsequence}
  Assume that either $\Omega\subset \RR^d$ is bounded or $\RR^d\setminus \Omega$ is bounded.
  Then for all $\fru\in \mbH(\partial\Omega)$, we have
  $\Im \{[\mathsf{A}_\kappa^{\Omega}(\fru),\overline{\fru}]_{\partial\Omega}\}\geq 0$.
\end{proposition}
\noindent The proof of this result can be deduced for example from the positivity of the capacity operator stated in \cite[Thm.~5.3.5]{Nedelec:book:2001}. However, since we are not able to find a definitive proof in the current literature, we provide it in Proposition~\ref{RadiationConditionConsequence2} in the appendix.  

Once again, in the context of the multi-domain configuration \eqref{eq:boundaries}, we shall write $\mathsf{SL}_\kappa^{j}$, $\mathsf{DL}_\kappa^{j}$, $\mathsf{G}_\kappa^{j}$, $\mathsf{A}_\kappa^{j}$ (resp.~$\mathsf{SL}_\kappa^{\Sigma}$, $\mathsf{DL}_\kappa^{\Sigma}$, $\mathsf{G}_\kappa^{\Sigma}$, $\mathsf{A}_\kappa^{\Sigma}$) instead of $\mathsf{SL}_\kappa^{\Omega_j}$, $\mathsf{DL}_\kappa^{\Omega_j}$, $\mathsf{G}_\kappa^{\Omega_{j}}$, $\mathsf{A}_\kappa^{\Omega_{j}}$ (resp.~$\mathsf{SL}_\kappa^{\Omega_\Sigma}$, $\mathsf{DL}_\kappa^{\Omega_\Sigma}$, $\mathsf{G}_\kappa^{\Omega_{\Sigma}}$, $\mathsf{A}_\kappa^{\Omega_{\Sigma}}$).


\section{Trace spaces for multi-domain scattering}
\label{sec:spaces}

\noindent Based on previous contributions about multi-trace formalism \cite{ClHi:MTF:2013,ClHi:impenetrable:2015}, we introduce function spaces specific to multi-domain configurations. A natural trace space on the skeleton $\Gamma$ \eqref{eq:boundaries} is the \emph{multi-trace space} defined as the Cartesian product of local trace spaces on the homogeneous subdomains boundary:
\begin{equation}\label{MultiTraceSpace}
  \mbH(\Gamma) \coloneqq \mbH(\Gamma_0) \times \cdots \times \mbH(\Gamma_n),
\end{equation}
recalling that in \eqref{DefPairOfTraces} we have set $\mbH(\Gamma_j)\coloneqq\mH^{1/2}(\Gamma_j)\times \mH^{-1/2}(\Gamma_j)$ (note that no components on $\Sigma$ are involved in $\mbH(\Gamma)$).
The multi-trace space above is equipped with the Cartesian product norm defined by 
\begin{equation*}
  \Vert \frv\Vert_{\mbH(\Gamma)}^{2} \coloneqq 
  \sum_{j=0}^n \Vert \frv_j\Vert_{\mbH(\Gamma_j)}^{2}, \quad 
  \text{for }  \frv = (\frv_0,\dots,\frv_n)\in\mbH(\Gamma).
\end{equation*}
Throughout the paper we use gothic symbols $\fru, \frv, \frw$ to denote tuples of Dirichlet-Neumann traces, with a subscript indicating the pair of traces on a certain subdomain boundary. The trace operators $\gamma^{j}$ local to subdomains can be bundled to
form a \emph{global trace operator} on the skeleton $\Gamma$
\begin{equation}\label{GlobalTraceOperator}
    \gamma (U) \coloneqq (\gamma^{0}(U),\dots, \gamma^{n}(U)),
\end{equation}
which naturally maps continuously onto the multi-trace space $\gamma\colon\mH^{1}(\Delta,\Omega_0)\times
\dots\times\mH^{1}(\Delta,\Omega_n)\to  \mbH(\Gamma)$. Moreover, the multi-trace space \eqref{MultiTraceSpace} is naturally
equipped with the non-degenerate bilinear pairing $\lbr\cdot \, , \, \cdot\rbr_{\Gamma}\colon\mbH(\Gamma)\times\mbH(\Gamma)\to \CC$
defined by 
\begin{equation}\label{eq:pairingHGamma}
    \lbr \fru,\frv\rbr_{\Gamma} 
    \coloneqq \sum_{j=0}^n \lbr \fru_j,\frv_j\rbr_{\Gamma_j}, 
    \quad \text{for } \fru=(\fru_0,\dots,\fru_n), \frv=(\frv_0,\dots,\frv_n) \in\mbH(\Gamma).    
\end{equation}
We also need to introduce a subspace of \eqref{MultiTraceSpace} consisting of tuples of traces that comply with Dirichlet and Neumann transmission conditions through each interface $\Gamma_j\cap\Gamma_k$: the so-called \emph{single-trace space} $\mbX(\Gamma)\subset \mbH(\Gamma)$
is a closed subspace of $\mbH(\Gamma)$ defined as follows  
\begin{equation}\label{SingleTraceSpace}
  \begin{aligned}
    \mbX(\Gamma) \coloneqq \{ & (u_j,p_j)_{j=0,\dots,n}\in \mbH(\Gamma) \mid
    \;\exists V\in \mH^{1}(\RR^d),\;\bq\in \mH(\mrm{div},\RR^d)\;\\
    & \text{such that } u_j = V\vert_{\Gamma_j}\;\text{and}\;p_j = \bn_j\cdot\bq\vert_{\Gamma_j}\;\forall j=0,\dots,n\}.
  \end{aligned}
\end{equation}
In contrast to other articles about multi-trace formalism such as \cite{MR3202533,ClHi:MTF:2013}, Definition~\eqref{SingleTraceSpace} for $\mbX(\Gamma)$ stems from the decomposition $\RR^{d}\setminus\Omega_{\Sigma} = \cup_{j=0}^n\overline{\Omega}_j$, which is \textit{not} a partition of the full space $\RR^d$, i.e.~the subdomain $\Omega_\Sigma$ is assumed non-empty here. Because of this, the single-trace space $\mbX(\Gamma)$ obeys a \emph{modified} polarity identity involving a residual term localized on $\Sigma$, the boundary of the heterogeneous subdomain $\Omega_\Sigma$, see \eqref{eq:boundaries}. This property, stated in the following proposition, will play a crucial role in our analysis. 
\begin{proposition}[Modified polarity identity]\label{DefinitionTraceOnSigma}
  For any $\fru = (u_j,p_j)_{j=0,\dots,n}\in \mbX(\Gamma)$ stemming from the traces $u_j = V\vert_{\Gamma_j}$ and $p_j = \bn_j\cdot \bq\vert_{\Gamma_j}$ of some $V\in \mH^{1}(\RR^d)$, $\bq\in \mH(\mrm{div},\RR^d)$, define 
  \begin{equation}\label{eq:defT}
    \Tr(\fru) \coloneqq (V\vert_{\Sigma},\bn_{\Sigma}\cdot\bq\vert_{\Sigma}).
  \end{equation}
  Then $\Tr(\fru)$ does not depend on the particular liftings $V,\bq$, and the formula above defines a continuous and surjective operator $\Tr\colon\mbX(\Gamma)\to \mbH(\Sigma)$ satisfying the \emph{modified polarity identity}
  \begin{equation}\label{PolarityOperatorTr}
    \lbr \fru,\frv\rbr_{\Gamma} = -\lbr\Tr(\fru),\Tr(\frv)\rbr_{\Sigma}\quad \forall \, \fru,\frv\in\mbX(\Gamma).
  \end{equation}  
\end{proposition}
\noindent 
This proposition was established in \cite[Prop.~3.1 and Prop.~3.2]{ClHi:impenetrable:2015}, where $\Omega_\Sigma$ represented an impenetrable part of the propagation medium.  The operator $\Tr$ should be understood as a trace operator on $\Sigma$. Subsequently, we shall decompose the operator $\Tr$ into Dirichlet and Neumann components, setting $\Tr(\fru) =(\Tr_{\dir}(\fru), \Tr_{\neu}(\fru) )$, with $\Tr_{\dir}\colon\mbX(\Gamma)\to \mH^{1/2}(\Sigma)$ and $\Tr_{\neu}\colon\mbX(\Gamma)\to \mH^{-1/2}(\Sigma)$ continuous. The modified polarity identity leads to a variational characterization of $\mbX(\Gamma)$:
\begin{lemma}[Variational characterization of $\mbX(\Gamma)$]\label{LemmaModifiedPolarity}
  For any $\fru\in \mbH(\Gamma)$, we have $\fru\in \mbX(\Gamma)$ if and only if
  $\lbr \fru,\frv\rbr_{\Gamma} = 0$ for all $\frv\in \mbX(\Gamma)$ satisfying $\Tr(\frv) =0$. 
\end{lemma}
\begin{proof}
  First, as a direct application of  Proposition~\ref{DefinitionTraceOnSigma}, for any $\fru\in \mbX(\Gamma)$ and any $\frv\in \mbX(\Gamma)$ with $\Tr(\frv) = 0$, we have  $\lbr \fru,\frv\rbr_{\Gamma} = -\lbr \Tr(\fru),0\rbr_{\Sigma} = 0$.

  Reciprocally, take an arbitrary $\fru\in \mbH(\Gamma)$, and assume that $\lbr \fru,\frv\rbr_{\Gamma} =0$ for all $\frv\in \mbX(\Gamma)$ satisfying $\Tr(\frv) = 0$. Consider $U_j\in \mH^{1}(\Omega_j)$, $\bp_j\in \mH(\mrm{div},\Omega_j)$ such that $\fru = (U_j\vert_{\Gamma_j},\bn_j\cdot\bp_j \vert_{\Gamma_j})_{j=0,\dots,n}$, and define $U \in \mL^{2}(\RR^d\setminus\Omega_\Sigma)$ and $\bp\in \mL^{2}(\RR^d\setminus\Omega_\Sigma)^{d}$ by $U\vert_{\Omega_j} \coloneqq U_j$ and $\bp\vert_{\Omega_j} \coloneqq \bp_j$.

  We need to prove that $U\in \mH^{1}(\RR^d\setminus\overline{\Omega}_\Sigma)$ and $\bp\in\mH(\mrm{div},\RR^{d}\setminus\overline{\Omega}_\Sigma)$ to conclude. We prove the result only for $U$, since the proof proceeds in a completely analogous manner for $\bp$. It suffices to show the existence of $C>0$ such that
  \begin{equation*}
    \biggl \lvert \int_{\RR^{d}\setminus\Omega_{\Sigma}} U\,\mrm{div}(\bvphi) \, d\bx \biggr \rvert \leq C\Vert \bvphi\Vert_{\mL^{2}(\RR^{d})}
    \quad\forall \bvphi\in \mathscr{C}^{\infty}_{\comp}(\RR^d\setminus\overline{\Omega}_{\Sigma})^d, 
  \end{equation*}
  where
  $\mathscr{C}^{\infty}_{\comp}(\RR^d\setminus\overline{\Omega}_{\Sigma})\coloneqq\{
  V\in \mathscr{C}^{\infty}(\RR^d) \mid \,\mrm{supp}(V)\;
  \text{bounded},\;V = 0\;\text{in}\;\Omega_{\Sigma}\}$. Pick
  $\bvphi\in
  \mathscr{C}^{\infty}_{\comp}(\RR^d\setminus\overline{\Omega}_{\Sigma})^d$
  arbitrary and set $\frv =
  (0,\bn_{j}\cdot\bvphi\vert_{\Gamma_j})_{j=0,\dots, n}$. By
  construction we have $\frv\in\mbX(\Gamma)$ and $\Tr(\frv) = 0$,
  since $\bn_{\Sigma}\cdot\bvphi\vert_{\Sigma} = 0$. Next, decomposing
  the integral according to $\RR^d\setminus\Omega_{\Sigma} =
  \overline{\Omega}_0\cup\dots \cup \overline{\Omega}_{n}$, and using
  the identity $\lbr \fru,\frv\rbr_{\Gamma} = 0$, we have $
  \int_{\RR^{d}\setminus\Omega_{\Sigma}} U\,\mrm{div}(\bvphi)\,d\bx =
  \sum_{j=0}^{n}\int_{\Omega_j} U_j\,\mrm{div}(\bvphi) \,d\bx =
  -\sum_{j=0}^{n}\int_{\Omega_j} \bvphi\cdot\nabla U_j \,d\bx$, which
  leads to the conclusion.
\end{proof}

\noindent 
Let us point out that any tuple $( u_j,p_j)_{j=0,\dots,n}\in
\mbX(\Gamma)$ satisfies $u_j = u_k$ and $p_j = -p_k$ on
$\Gamma_j\cap\Gamma_k$.  This observation and
Lemma~\ref{LemmaModifiedPolarity} lead to alternative ways of writing
the transmission conditions:
\begin{lemma}[Characterizations of transmission conditions]\label{ReformulationTransmissionConditions}
  For any $U\in \mL^{2}_{\loc}(\RR^{d})$ such that
  $U\vert_{\Omega_\Sigma}\in \mH_{\loc}^1(\Delta,\overline{\Omega}_\Sigma)$ and
  $U\vert_{\Omega_j}\in \mH_{\loc}^1(\Delta,\overline{\Omega}_j), j=0\dots,n $,
  we have that $U$ satisfies the transmission conditions of Problem
  \eqref{eq:bvp}, that is, $U\in \mH_{\loc}^1(\Delta,\RR^d)$ if and
  only if
  \begin{equation}
    \label{eq:ReformulationTransmissionConditions}
    \gamma(U)\in \mbX(\Gamma)\quad \text{and}\quad \Tr(\gamma(U)) = \gamma^{\Sigma}(U), 
  \end{equation}  
  or equivalently
  \begin{equation}
    \label{eq:ReformulationTransmissionConditionsbis}
    \lbr \gamma(U), \frv\rbr_{\Gamma} + \lbr \gamma^\Sigma (U), \Tr(\frv) \rbr_\Sigma = 0 \quad 
    \text{for all }\frv\in \mbX(\Gamma). 
  \end{equation}  
\end{lemma}
\begin{proof}
  For characterization \eqref{eq:ReformulationTransmissionConditions}, it is enough to  combine the observation above with the definitions of $\Tr$ in \eqref{eq:defT} and of the global trace operator \eqref{GlobalTraceOperator}. 

  Now, we prove that \eqref{eq:ReformulationTransmissionConditionsbis} is a variational reformulation of \eqref{eq:ReformulationTransmissionConditions}. A direct application of the modified polarity identity \eqref{PolarityOperatorTr} shows that \eqref{eq:ReformulationTransmissionConditions} implies \eqref{eq:ReformulationTransmissionConditionsbis}. 
  Conversely, suppose that \eqref{eq:ReformulationTransmissionConditionsbis} holds true. In particular, if we take $\frv\in \mbX(\Gamma)$ with $\Tr(\frv)=0$, then $\lbr \gamma(U), \frv\rbr_{\Gamma} = 0$ for all $\frv\in \mbX(\Gamma)$ with $\Tr(\frv)=0$. According to Lemma \ref{LemmaModifiedPolarity}, this implies that $\gamma(U)\in\mbX(\Gamma)$. Moreover, considering now a generic $\frv\in \mbX(\Gamma)$ and applying the polarity identity \eqref{PolarityOperatorTr} to the first term of \eqref{eq:ReformulationTransmissionConditionsbis}, we get 
  \[ 
    - \lbr \Tr(\gamma(U)), \Tr(\frv)\rbr_{\Sigma} + \lbr \gamma^\Sigma (U), \Tr(\frv) \rbr_\Sigma = 0 \quad 
    \text{for all }\frv\in \mbX(\Gamma),
  \]
  that yields $\Tr(\gamma(U)) = \gamma^{\Sigma}(U)$ because $\Tr$ surjectively maps $\mbX(\Gamma)$ onto $\mbH(\Sigma)$. 
\end{proof}
This characterization of transmission conditions motivates the introduction of a variant of the single-trace space involving the additional subdomain boundary $\Sigma$: 
\begin{equation}
  \label{eq:defXtilde}
  \widetilde{\mbX}(\Gamma)\coloneqq\set{ (\fru,\Tr(\fru)) | \fru\in\mbX(\Gamma)}, 
\end{equation}
which stems from the decomposition of the full space $\RR^{d} = \cup_{j=0}^{n}  \, \overline{\Omega}_j \cup \overline{\Omega}_{\Sigma}$ as in \cite{ClHi:MTF:2013}. 
With this space we can rephrase once more: $U$ satisfies the transmission conditions of Problem \eqref{eq:bvp}
if and only if $(\gamma(U),\gamma^{\Sigma}(U))\in \widetilde{\mbX}(\Gamma)$.  
\begin{remark}
\label{rem:construct}
A crucial procedure to construct an element of $\widetilde{\mathbb{X}}(\Gamma)$ is the following. Given $j$ and a function $V \in \mathrm{H}^1_\textup{loc}(\Delta, \mathbb{R}^d \backslash \Omega_j)$, we set $\frv_k = \gamma^k(V)$ for $k \ne j$, $\frv_j = \gamma^j_c(V)$ and $\frv_\Sigma = \Tr(\frv) = \gamma^\Sigma(V)$. Then $(\frv_0, \dots, \frv_n, \frv_\Sigma) \in \widetilde{\mathbb{X}}(\Gamma)$.
\end{remark}

We conclude this section by introducing a variational space adapted to the presence of heterogeneities in $\Omega_{\Sigma}$, namely
\begin{equation}
\label{eq:XsigmaGamma}
  \XSigmaGamma\coloneqq \set{(U,\fru)\in \mH^{1}(\Omega_{\Sigma})\times\mbX(\Gamma) | \gamma^{\Sigma}_{\dir}(U)=\Tr_{\dir}(\fru)}, 
\end{equation}
i.e.~we impose that on $\Sigma$ the Dirichlet trace of a ``heterogeneous'' component $U$ defined in $\Omega_{\Sigma}$ matches the Dirichlet trace $\Tr_{\dir}(\fru)$ of a single-trace tuple $\fru$ defined on the skeleton $\Gamma$. This is clearly a closed subspace of $\mH^{1}(\Omega_{\Sigma})\times\mbX(\Gamma)$
for the inherited Cartesian product norm given by $(U,\fru)\mapsto (\Vert U\Vert_{\mH^{1}(\Omega_{\Sigma})}^2+\Vert \fru\Vert_{\mbH(\Gamma)}^2)^{1/2}$.

\section{Review of the classical Costabel coupling}
\label{sec:2costabel}

\noindent We revisit the classical Costabel symmetric coupling \cite[§7]{Costabel:lectures:1987}\cite{Costabel:symm:1987}, writing the formulation in the compact notation introduced in the previous sections. This will also allow the reader to get more acquainted with our notation.

\begin{figure}
\centering
\includegraphics[width=0.35\linewidth]{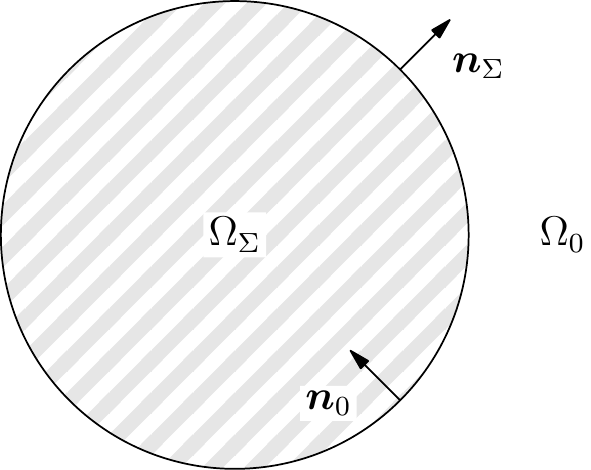}
\caption{Geometric setting for the classical Costabel coupling.}
\label{fig:costabel_geo}
\end{figure}
The classical Costabel coupling gives a symmetric variational formulation of the transmission problem~\eqref{eq:bvp} in the case $n=0$, i.e.~$\RR^d = \overline\Omega_0 \cup \overline\Omega_\Sigma$, $\Gamma = \Gamma_0 = \Sigma$ (see Figure~\ref{fig:costabel_geo}),
which combines direct boundary integral equations\footnote{A boundary integral equation is of direct type if its unknowns are Dirichlet/Neumann traces of the solution to the related boundary value problem.} for $\Omega_0$ with a volume variational formulation for $\Omega_\Sigma$.
Note that in our presentation, in contrast to what is usually done in the literature, for $\Omega_0$ we take its own outward-pointing normal vector $\bn_0$. This choice is more suitable in view of the extension to the multi-domain case. In the two-subdomain case of the present section we have $\mbX(\Gamma) = \mbH(\Gamma) = \mbH(\Gamma_0) = \mbH(\Sigma)$ and 
\begin{equation}\label{eq:XSigmaGamma2subs}
  \XSigmaGamma=\{ (V, (\gamma^{\Sigma}_{\dir}(V),q) ) \mid
V\in\mH^{1}(\Omega_{\Sigma}), \, q\in\mH^{-1/2}(\Sigma)\}  
\end{equation} 
so $\XSigmaGamma$ is naturally isomorphic to $\mH^{1}(\Omega_{\Sigma})\times \mH^{-1/2}(\Sigma)$, which is the space where the Costabel coupling is usually posed. 

Now consider $U\in \mH^{1}_{\loc}(\RR^d)$ solution to the transmission problem \eqref{eq:bvp}. We are going to reformulate this problem equivalently in terms of the pair
\begin{equation}\label{eq:VarSpaceandDefu}
  (U\vert_{\Omega_{\Sigma}},\fru)\in \XSigmaGamma, \quad
    \text{where}\;\;\fru = \gamma^{0}(U).
\end{equation}
Thus, the Dirichlet transmission condition $\gamma^{\Sigma}_{\dir}(U) = \gamma^{0}_{\dir}(U)$ shall be enforced strongly through the choice of $\XSigmaGamma$ as variational space (recall its definition in \eqref{eq:XsigmaGamma}). 
To reformulate \eqref{eq:bvp} variationally, we first deal with the Helmholtz equation satisfied by $U$ in $\Omega_\Sigma$. Pick an arbitrary test pair $(V,\frv)\in \XSigmaGamma$ and, after multiplying the equation by $V$, apply Green's formula in $\Omega_\Sigma$. This leads to a variational identity involving a boundary term:
\begin{equation}\label{eq:HelmOmegaSigma}
  \begin{array}{l}
    a_{\Sigma}(U,V) - \langle \gamma_{\neu}^{\Sigma}(U),
    \gamma_{\dir}^{\Sigma}(V)\rangle_{\Sigma} = F_{\Sigma}(V)\\[10pt]
    \text{where}\;\;a_{\Sigma}(U,V)\coloneqq\int_{\Omega_{\Sigma}}
    (\nabla U\cdot\nabla V-\kappa_{\Sigma}^2(\bx)UV) \,d\bx\\
    \hspace{1.45cm} F_{\Sigma}(V)\coloneqq\int_{\Omega_{\Sigma}}fV\,d\bx.
  \end{array}
\end{equation}
Next, to rewrite the boundary term, we observe that $\gamma_{\dir}^{\Sigma}(V) = \Tr_{\dir}(\frv)$ because $(V,\frv)\in \XSigmaGamma$, and $\gamma^{\Sigma}_{\neu}(U) = \Tr_{\neu}(\fru)$ by the Neumann transmission condition and \eqref{eq:VarSpaceandDefu}.  
Hence, recalling the operator $\theta(v,q) \coloneqq (-v,q)$, we apply identity~\eqref{eq:thetaIdentity}, together with the polarity property~\eqref{PolarityOperatorTr} using $\fru,\frv\in \mbX(\Gamma)$, so that we obtain 
\begin{equation}\label{rewriteBoundaryTermSigma}
  \begin{aligned}
    -\langle \gamma_{\neu}^{\Sigma}(U),
    \gamma_{\dir}^{\Sigma}(V)\rangle_{\Sigma}
    & = -\langle \Tr_{\neu}(\fru),\Tr_{\dir}(\frv)\rangle_{\Sigma}\\
    & = -\lbr \Tr(\fru),\theta(\Tr(\frv))\rbr_{\Sigma}/2 + \lbr \Tr(\fru),\Tr(\frv)\rbr_{\Sigma}/2 \\
    & = +\lbr \fru,\theta(\frv)\rbr_{\Gamma}/2 + \lbr \Tr(\fru),\Tr(\frv)\rbr_{\Sigma}/2. 
  \end{aligned}
\end{equation}
Therefore, Equation~\eqref{eq:HelmOmegaSigma} becomes
\begin{equation}\label{VarIdCostabel1}
  a_{\Sigma}(U,V) + \lbr \fru,\theta(\frv)\rbr_{\Gamma}/2
  +  \lbr \Tr(\fru),\Tr(\frv)\rbr_{\Sigma}/2 = F_{\Sigma}(V).
\end{equation}
Now, we wish to exploit boundary integral operators in $\Omega_0$. 
Since $U_{\inc}$ solves the homogeneous Helmholtz equation with wavenumber $\kappa_0$ in $\Omega_{\Sigma} = \RR^d \backslash \overline \Omega_0$ and $\gamma^{0}(U_{\inc}) = \gamma^{0}_c(U_{\inc})$, the ``exterior'' representation formula~\eqref{eq:extRep} is applicable to $U_\inc$ in $\Omega_0$ and yields 
$\gamma^0\mathsf{G}_{\kappa_0}^{0} (\gamma^{0}(U_\inc)) = \gamma^0\mathsf{G}_{\kappa_0}^{0}(\gamma^{0}_c(U_\inc))=0$. 
As $U-U_\inc$ solves the homogeneous Helmholtz equation in $\Omega_0$ and satisfies the associated $\kappa_0$-radiation condition, the representation formula~\eqref{eq:intRep} is applicable to $U-U_\inc$ in $\Omega_0$ and yields 
$\gamma^0(U-U_{\inc})= \gamma^0\mathsf{G}_{\kappa_0}^0 (\gamma^0(U-U_\inc)) =
\gamma^0\mathsf{G}_{\kappa_0}^0 (\gamma^0(U))$. 
Then, making use of \eqref{eq:Akj12} and $\fru = \gamma^{0}(U)$, we get
\begin{equation}\label{ExteriorCalderonSec6}
\fru/2 = \mathsf{A}_{\kappa_0}^{0}(\fru) + \gamma^{0}(U_{\inc}).
\end{equation}
This is a reformulation of the Helmholtz equation satisfied by $U$ in $\Omega_0$ based on both Dirichlet and Neumann traces of the representation formula. Note that, in contrast to the present Costabel coupling, the Johnson-N\'ed\'elec coupling would involve just the Dirichlet one. Plugging \eqref{ExteriorCalderonSec6} into \eqref{VarIdCostabel1}, we finally obtain the variational formulation of the \emph{Costabel symmetric coupling} posed in $\XSigmaGamma$:
\begin{empheq}[box=\fbox]{equation}\label{VarIdCostabel2}
  \begin{aligned}
    & \text{find } (U,\fru)\in \XSigmaGamma\;\; \text{such that}\\
    & a_{\Sigma}(U,V) +\lbr \mathsf{A}_{\kappa_0}^{0}(\fru),\theta(\frv)\rbr_{\Gamma}
    +  \lbr \Tr(\fru),\Tr(\frv)\rbr_{\Sigma}/2\\
    & = F_{\Sigma}(V) - \lbr \gamma^{0}(U_{\inc}),\theta(\frv)\rbr_{\Gamma}
    \quad \forall\, (V,\frv)\in \XSigmaGamma.
  \end{aligned}
\end{empheq}
Note that all the four classical boundary integral operators, which are the components of the block operator $\mathsf{A}_{\kappa_0}^{0}$ (see \eqref{DefPMCHWTOp}), are involved in the Costabel coupling.
  In this two-subdomain configuration, where $\Gamma = \Gamma_0 = \Sigma$ and $\bn_0 = - \bn_\Sigma$, we have $\Tr((u,p)) = (u,-p)$ (see definition~\eqref{eq:defT} of $\Tr$), so that the term $+\lbr \Tr(\fru),\Tr(\frv)\rbr_{\Sigma}/2$ can be simplified as $- \lbr \fru,\frv \rbr_{\Sigma}/2$. 
Moreover, by the observation in \eqref{eq:XSigmaGamma2subs} and recalling the definition of $\theta$, formulation~\eqref{VarIdCostabel2} can be written more explicitly as:
\begin{equation*}
  \begin{aligned}
    & \text{find } U\in\mH^{1}(\Omega_{\Sigma}), \, p\in\mH^{-1/2}(\Sigma)\;\; \text{such that}\\
    & \int_{\Omega_{\Sigma}}
    (\nabla U\cdot\nabla V-\kappa_{\Sigma}^2(\bx)UV) \,d\bx +\lbr \mathsf{A}_{\kappa_0}^{0}( (\gamma_{\dir}^{\Sigma}U, p) ), (-\gamma_{\dir}^{\Sigma}V, q)\rbr_{\Sigma}
    -  \lbr (\gamma_{\dir}^{\Sigma}U, p), (\gamma_{\dir}^{\Sigma}V, q) \rbr_{\Sigma}/2\\
    & = \int_{\Omega_{\Sigma}}fV\,d\bx \,- \lbr \gamma^{0}(U_{\inc}), (-\gamma_{\dir}^{\Sigma}V, q) \rbr_{\Sigma}
    \qquad \forall\, V\in\mH^{1}(\Omega_{\Sigma}), \, q\in\mH^{-1/2}(\Sigma).
  \end{aligned}
\end{equation*} 

Now, let $\mathsf{a}_\textup{C} \colon \XSigmaGamma\times\XSigmaGamma \to \CC$ designate the bilinear form on the left-hand side of \eqref{VarIdCostabel2}. The bilinear form $a_{\Sigma}(\cdot,\cdot)$ satisfies a G\r{a}rding inequality, as well as $\lbr \mathsf{A}_{\kappa_0}^{0}(\cdot),\theta(\cdot)\rbr_{\Gamma}$ (see Proposition~\ref{pr:gardingAj}). Hence, since $\Re \{\lbr \Tr(\frv),\Tr(\overline{\frv})\rbr_{\Sigma}\} = 0$, we conclude, as in \cite{HiMe:sfembem:2006}, that $\mathsf{a}_\textup{C}(\cdot,\cdot)$ satisfies a \emph{G\r{a}rding inequality}: there exist a compact bilinear form $\mathcal{K} \colon \XSigmaGamma\times \XSigmaGamma \to \CC$ and a constant $\beta>0$ such that
\begin{equation*}
\Re \{\,\mathsf{a}_\textup{C} \bigl(\,(V,\frv),\overline{(V,\frv)}\, \bigr)  + 
    \mathcal{K}\bigl(\,(V,\frv),\overline{(V,\frv)}\,\bigr)\,\}
    \,\ge\, \beta (\lVert V \rVert_{\mH^1(\Omega_\Sigma)}^2 + \lVert \frv \rVert_{\mbH(\Gamma)}^2)
\end{equation*}
for all $(V,\frv) \in \XSigmaGamma$.  As a consequence, the operator induced by $\mathsf{a}_\textup{C}$ is of Fredholm type with index $0$ (see \cite[Theorem 2.33]{McLean:book:2000}), i.e.~it is bijective if and only if it is injective.

The classical Costabel coupling may be affected by the \emph{spurious resonances} phenomenon, that is, the formulation fails to possess a unique solution for the wavenumbers $\kappa_0$ whose square is an interior Dirichlet eigenvalue of $-\Delta$ on $\Omega_\Sigma$,
i.e.~for $\kappa_0$ belonging to
\begin{equation*}
    \mathfrak{S}(\Delta,\Omega_\Sigma) \coloneqq 
    \set{\kappa \in\CC | \exists \, W \in \mH^1_{0}(\Omega_\Sigma)\backslash\{0\} \text{ such that } 
    -\Delta W = \kappa^2 W\;\text{in}\;\Omega_\Sigma}.
\end{equation*}
\begin{example}[Spurious resonances]
\label{ex:resCostabel}
Let $\kappa_0 \in \mathfrak{S}(\Delta,\Omega_\Sigma)$ and $W \in\mH^1(\Omega_\Sigma)\backslash\{0\}$ such that $-\Delta W =\kappa_0^2 W$ in $\Omega_\Sigma$ and $W=0$ on $\Sigma$. 
In particular $\gamma_{\dir}^{\Sigma}(W) = \gamma^{0}_{\dir,c}(W) = 0$. Then, setting $U=0$ and $\fru = \gamma^0_c(W)$, we have $(U,\fru)\in \XSigmaGamma$. 
Moreover, by the ``exterior'' representation formula \eqref{eq:extRep} we have $\mathsf{G}_{\kappa_0}^0(\gamma^0_c(W)) = 0$ in $\Omega_0$, and together with \eqref{eq:Akj12} we obtain $\mathsf{A}_{\kappa_0}^{0}(\gamma^0_c(W)) = 0-\gamma^0_c(W)/2$.   
Therefore, by the polarity property~\eqref{PolarityOperatorTr} and identity~\eqref{eq:thetaIdentitybis} 
\begin{equation*}
  \begin{aligned}
    & a_{\Sigma}(U,V) +\lbr \mathsf{A}_{\kappa_0}^{0}(\fru),\theta(\frv)\rbr_{\Gamma}
    + \lbr \Tr(\fru),\Tr(\frv)\rbr_{\Sigma}/2
    \\
    & \quad = 0 + \lbr \mathsf{A}_{\kappa_0}^{0}(\gamma^0_c(W)),\theta(\frv)\rbr_{\Gamma}
    - \lbr \gamma^0_c(W), \frv\rbr_{\Gamma}/2\\
    & \quad = -\lbr \gamma^0_c(W),\theta(\frv)\rbr_{\Gamma}/2
    - \lbr \gamma^0_c(W), \frv\rbr_{\Gamma}/2
    = -\langle \gamma^{0}_{\dir,c}(W),q\rangle_{\Gamma} =0
  \end{aligned}
\end{equation*}
for all $(V,\frv)\in \XSigmaGamma$, with $\frv = (v,q)$. This indicates that $(U,\fru)$ is a non-trivial solution to \eqref{VarIdCostabel2} with homogeneous right-hand side $F_{\Sigma}\equiv0$, $U_\inc=0$. 
\end{example}

It turns out that $\kappa_0 \in \mathfrak{S}(\Delta,\Omega_\Sigma)$ is also a necessary condition for the presence of spurious resonances. To prove this, we need the following equivalence result between the Costabel coupling formulation \eqref{VarIdCostabel2} and the transmission problem \eqref{eq:bvp} with $n=0$.
\begin{proposition}[Equivalence]\label{EquivalenceCouplageCostabeln=0}
If $\widetilde{U} \in \mH^1_\loc(\RR^d)$ solves \eqref{eq:bvp} with $n=0$, then the pair 
$(U,\fru) = (\widetilde{U}|_{\Omega_\Sigma},\gamma^{0}(\widetilde{U}))$ 
solves \eqref{VarIdCostabel2}. Reciprocally, if $(U,\fru)\in \XSigmaGamma$ solves \eqref{VarIdCostabel2}, then the solution to \eqref{eq:bvp} with $n=0$ is given by
\begin{equation}
  \label{eq:UtildeC}
  \begin{aligned}
    \widetilde{U} (\bx) & \coloneqq U(\bx) & \text{ for } \bx \in \Omega_\Sigma, \\
    \widetilde{U} (\bx) & \coloneqq (\mathsf{G}_{\kappa_0}^0(\fru) +U_\inc) (\bx) & \text{ for } \bx \in \Omega_0.
  \end{aligned}
\end{equation}
\end{proposition}
\begin{proof}
  The first implication stems from the derivation of
  \eqref{VarIdCostabel2}, so we only need to examine the other
  implication.  First of all, $(\widetilde{U}-U_\inc)|_{\Omega_0} =
  \mathsf{G}_{\kappa_0}^0 (\fru)$ is $\kappa_0$-outgoing radiating in
  $\Omega_0$, see e.g.~\cite[Theorem 3.2]{CoKr:bookIEM:1983}.  Second,
  $\widetilde{U}$ satisfies the Helmholtz equation in $\Omega_0$ since
  it is satisfied by $U_\inc$ by definition and also by the
  potentials, see e.g.~\cite[§2.4]{CoKr:bookIEM:1983}.  If we take
  $(V,0)\in \mH_0^{1}(\Omega_\Sigma)\times\{0\} \subset
  \XSigmaGamma$ as test function in \eqref{VarIdCostabel2}, we
  obtain $a_\Sigma(U,V) = a_\Sigma(\widetilde{U},V) = F_\Sigma(V)$, so
  $\widetilde{U}$ satisfies Helmholtz equation also in
  $\Omega_\Sigma$, and there only remains to prove that
  $\widetilde{U}$ complies with the transmission conditions of
  \eqref{eq:bvp} through $\Gamma \equiv \Sigma$.

  \smallskip 
  Now, considering a generic $(V,\frv)\in\XSigmaGamma$ where
  $V\in \mH^1(\Omega_{\Sigma})$ (not necessarily
  $V\in\mH_0^1(\Omega_\Sigma)$), and integrating by parts, we obtain
  \begin{equation}
    \label{eq:equiv1}
    a_\Sigma(\widetilde{U},V) 
    - \braket{\gamma_{\neu}^\Sigma \widetilde{U}, \gamma_{\dir}^\Sigma V}_\Gamma = 
    F_\Sigma(V) \quad
    \forall \, V \in \mH^1(\Omega_\Sigma).
  \end{equation}
  By \eqref{eq:UtildeC} and \eqref{eq:Akj12}, we have
  \begin{equation}\label{eq:gamma0utilde}
    \gamma^{0}(\widetilde{U}) = \mathsf{A}_{\kappa_0}^{0}(\fru) +
  \fru/2 + \gamma^{0}(U_{\inc}).
  \end{equation}
  Then, plugging \eqref{eq:equiv1} and \eqref{eq:gamma0utilde} into \eqref{VarIdCostabel2} leads to
  \begin{equation*}
    \begin{aligned}
      &  \braket{\gamma_{\neu}^\Sigma \widetilde{U}, \gamma_{\dir}^\Sigma V}_\Gamma
      + \lbr \mathsf{A}_{\kappa_0}^{0}(\fru),\theta(\frv)\rbr_{\Gamma}
      + \lbr \Tr(\fru),\Tr(\frv)\rbr_{\Sigma}/2
      = - \lbr \gamma^{0}(U_{\inc}),\theta(\frv)\rbr_{\Gamma}\\
      &  \braket{\gamma_{\neu}^\Sigma \widetilde{U}, \gamma_{\dir}^\Sigma V}_\Gamma
      + \lbr \gamma^{0}(\widetilde{U}) - \fru/2,\theta(\frv)\rbr_{\Gamma}
      + \lbr \Tr(\fru),\Tr(\frv)\rbr_{\Sigma}/2 = 0
    \end{aligned}
  \end{equation*}
  that is, by the polarity property~\eqref{PolarityOperatorTr} and
  identity~\eqref{eq:thetaIdentitybis} writing $\fru=(u,p)$,
  $\frv=(v,q)$,  
  \begin{equation*}
    \begin{aligned}
    & \braket{\gamma_{\neu}^\Sigma \widetilde{U}, \gamma_{\dir}^\Sigma V}_\Gamma
    + \lbr \gamma^{0}(\widetilde{U}),\theta(\frv)\rbr_{\Gamma}
    = \lbr \fru, \theta(\frv) \rbr_{\Gamma}/2
    + \lbr \fru, \frv \rbr_{\Gamma}/2\\
    & \braket{\gamma_{\neu}^\Sigma \widetilde{U}, \gamma_{\dir}^\Sigma V}_\Gamma
    - \langle u,q\rangle_{\Gamma}
    - \lbr \theta\circ\gamma^{0}(\widetilde{U}),\frv\rbr_{\Gamma}
    = 0. 
    \end{aligned}
  \end{equation*}
  Since $(U,\fru)\in \XSigmaGamma$ we have $u =
  \Tr_{\dir}(\fru) = \gamma^{\Sigma}_{\dir}(U) =
  \gamma^{\Sigma}_{\dir}(\widetilde{U})$.  Similarly, for the test
  pair we have $(V,\frv)\in \XSigmaGamma$, hence
  $\gamma^{\Sigma}_{\dir}(V) = \Tr_{\dir}(\frv) = v$.  As a
  consequence, $\braket{\gamma_{\neu}^\Sigma \widetilde{U},
    \gamma_{\dir}^\Sigma V}_\Gamma - \langle u,q\rangle_{\Gamma} =
  - \lbr \gamma^\Sigma(\widetilde{U}), \frv\rbr_\Gamma$ and, finally, we
  obtain $\lbr \gamma^\Sigma(\widetilde{U})+\theta\circ\gamma^{0}
  (\widetilde{U}),\frv\rbr_{\Gamma}= 0$ for all $\frv = (v,q)\in
  \mH^{1/2}(\Gamma)\times \mH^{-1/2}(\Gamma)$. This implies that
  $\gamma^\Sigma(\widetilde{U})=-\theta\circ\gamma^{0}(\widetilde{U})$,
  which also rewrites $\gamma^{0}_{\dir}(\widetilde{U}) =
  \gamma^{\Sigma}_{\dir}(\widetilde{U})$ and
  $\gamma^{0}_{\neu}(\widetilde{U})
  =-\gamma^{\Sigma}_{\neu}(\widetilde{U})$.
\end{proof}

\begin{corollary}[Injectivity condition]
\label{cor:injCondCostabel}
Let $(U,\fru) \in \XSigmaGamma$, solve \eqref{VarIdCostabel2} with $F_{\Sigma}\equiv0$ and $U_\inc=0$.
Then $U=0$. If $\kappa_0 \notin \mathfrak{S}(\Delta,\Omega_\Sigma)$ we also have $\fru=0$ necessarily.
\end{corollary}
\begin{proof}
  By the equivalence Proposition~\ref{EquivalenceCouplageCostabeln=0}, $\widetilde{U} \in \mH^1_\loc(\RR^d)$ defined by \eqref{eq:UtildeC} satisfies the transmission problem \eqref{eq:bvp} with $n=0$, which is well posed, so $\widetilde{U}=0$. 
  Since $\widetilde{U}|_{\Omega_\Sigma} = U$, we get $U=0$. 
  Denoting $\fru = (u,p)$, we then have $u = \Tr_{\dir}(\fru) = \gamma_{\dir}^{\Sigma}(U) = 0$ because $(U,\fru)\in \XSigmaGamma$. Moreover, since $\widetilde{U}|_{\Omega_0} = \mathsf{G}_{\kappa_0}^0(\fru)$, we obtain $\mathsf{G}_{\kappa_0}^0(\fru)(\bx) = 0$, that is $\mathsf{SL}_{\kappa_0}^0 (p) (\bx) = 0$ for $\bx \in \Omega_0$. 
  Therefore $\gamma_\dir^0 \mathsf{SL}_{\kappa_0}^0 (p) = 0$, which implies $p=0$ given $\kappa_0 \notin \mathfrak{S}(\Delta,\Omega_\Sigma)$ (see \cite[Theorem 3.9.1]{SaSw:book:2011}). 
\end{proof}

We refer to \cite{HiMe:sfembem:2006} for a combined field integral equation FEM-BEM formulation immune to spurious resonances.

\section{Single-trace FEM-BEM formulation}
\label{sec:STF}

\noindent 
In this section we shall revisit the analysis presented in the previous section, this time considering multi-domain configurations ($n\geq 1$, with potential cross-points) instead of a simple two-domain setting. This will lead to a first coupling variational formulation for the transmission problem~\eqref{eq:bvp} in the targeted multi-domain configuration. We combine a volume variational formulation in $\Omega_\Sigma$ with the boundary integral formulation on $\Gamma$ called Single-Trace Formulation (STF), first analyzed in \cite{Pet:STF:1989}. The Costabel coupling lends itself well to match the STF since it is based on the full set of Calder\'on identities, from which the STF arises.  In \cite[§4]{ClHi:impenetrable:2015} the STF was revisited and adapted to the case with an impenetrable part represented by the subdomain $\Omega_\Sigma$. The present analysis, where $\Omega_\Sigma$ is a heterogeneous part, bears several similarities to the analysis in \cite{ClHi:impenetrable:2015}.

As in the previous section, let us start with a function $U$ that is a unique solution to the transmission problem~\eqref{eq:bvp}.
We are going to reformulate this transmission problem in terms of the pair
\begin{equation}
  \begin{aligned}
    & (U\vert_{\Omega_{\Sigma}}, \fru)\in \XSigmaGamma\\
    & \text{where}\quad \fru = \gamma(U) = (\gamma^{0}(U),\dots,\gamma^{n}(U)).
  \end{aligned}
\end{equation}
Here, except for the Neumann condition through $\Sigma$ that writes $\gamma^{\Sigma}_{\neu}(U) = \Tr_\neu(\fru)$, the transmission conditions shall be enforced strongly by the choice of $\XSigmaGamma$ as variational space. As in Section~\ref{sec:2costabel}, pick an arbitrary test pair $(V,\frv)\in \XSigmaGamma$, and apply Green's formula in $\Omega_\Sigma$. Again, we obtain the following classical variational identity:
\begin{equation}\label{VolumeVariational1FormSec7}
  \begin{array}{l}
    a_{\Sigma}(U,V) - \langle \gamma_{\neu}^{\Sigma}(U),
    \gamma_{\dir}^{\Sigma}(V)\rangle_{\Sigma} = F_{\Sigma}(V)\\[10pt]
    \text{where}\;\;a_{\Sigma}(U,V)\coloneqq\int_{\Omega_{\Sigma}}
    (\nabla U\cdot\nabla V-\kappa_{\Sigma}^2(\bx)UV) \, d\bx\\
    \hspace{1.45cm} F_{\Sigma}(V)\coloneqq\int_{\Omega_{\Sigma}}fV\,d\bx.
  \end{array}
\end{equation}
Next, we rewrite the boundary term as in \eqref{rewriteBoundaryTermSigma}, except that, before applying the polarity property~\eqref{PolarityOperatorTr} to the term $-\lbr \Tr(\fru),\theta(\Tr(\frv))\rbr_{\Sigma}$, we need to introduce a multi-domain analogue of the operator $\theta$: $\Theta(\frv) \coloneqq (\theta(\frv_0),\dots,\theta(\frv_n))$ for $\frv = (\frv_0,\dots,\frv_n)\in \mbH(\Gamma)$. Noting that  $\theta(\Tr(\frv)) = \Tr(\Theta(\frv))$, we can write:
\begin{equation}\label{VolumeVariational1FormSec7-2}
  \begin{aligned}
    -\langle \gamma_{\neu}^{\Sigma}(U), \gamma_{\dir}^{\Sigma}(V)\rangle_{\Sigma}
    & = -\langle \Tr_{\neu}(\fru),\Tr_{\dir}(\frv)\rangle_{\Sigma}\\
    & = -\lbr \Tr(\fru),\theta(\Tr(\frv))\rbr_{\Sigma}/2 + \lbr \Tr(\fru),\Tr(\frv)\rbr_{\Sigma}/2 \\
    & = +\lbr \fru,\Theta(\frv)\rbr_{\Gamma}/2 + \lbr \Tr(\fru),\Tr(\frv)\rbr_{\Sigma}/2.
  \end{aligned}
\end{equation}
Plugging \eqref{VolumeVariational1FormSec7-2} into \eqref{VolumeVariational1FormSec7} we obtain
\begin{equation}\label{VolumeVariational2FormSec7}
    a_{\Sigma}(U,V)
    + \lbr \fru,\Theta(\frv)\rbr_{\Gamma}/2
    + \lbr \Tr(\fru),\Tr(\frv)\rbr_{\Sigma}/2
    = F_{\Sigma}(V).
\end{equation}
Following for $\Omega_0$ the same argumentation as in Section \ref{sec:2costabel}, we have that 
$\gamma^0\mathsf{G}_{\kappa_0}^{0} (\gamma^{0}(U_\inc)) = \gamma^0\mathsf{G}_{\kappa_0}^{0}
(\gamma^{0}_c(U_\inc))=0$, and 
$\gamma^{0}(U-U_\inc) = \gamma^0\mathsf{G}_{\kappa_0}^{0}(\gamma^{0}(U-U_\inc)) =
\gamma^0\mathsf{G}_{\kappa_0}^{0}\gamma^{0}(U)$.   
Hence $\gamma^{0}(U) = \gamma^0\mathsf{G}_{\kappa_0}^{0}\gamma^{0}(U) +\gamma^{0}(U_\inc)$, 
which by \eqref{eq:Akj12} also rewrites 
$\gamma^{0}(U)/2 = \mathsf{A}_{\kappa_0}^{0}\gamma^{0}(U) +\gamma^{0}(U_\inc)$.
Moreover, since $U$ verifies the Helmholtz equation with constant wavenumber $\kappa_j$ in $\Omega_j$, $j=1,\dots,n$, the representation formula \eqref{eq:intRep} yields 
$\gamma^{j}(U)=\gamma^{j}\mathsf{G}_{\kappa_j}^{j}(\gamma^{j}(U))$, 
that is, by \eqref{eq:Akj12}, $\gamma^{j}(U)/2 = \mathsf{A}_{\kappa_j}^{j}\gamma^{j}(U)$.
With the notation $\fru = \gamma(U) = (\gamma^{0}(U),\dots,\gamma^{n}(U))$, we have obtained
\begin{equation}\label{ExteriorCalderonSec7}
  \begin{aligned}
    & \fru/2 = \mathsf{A}(\fru) + \fru^\inc\\ 
    & \text{where}\;\;\mathsf{A} \coloneqq \mathrm{diag}(\mathsf{A}^{0}_{\kappa_0},\dots,\mathsf{A}^{n}_{\kappa_n})\\
    & \textcolor{white}{\text{where}}\;\;
    \fru^\inc \coloneqq (\gamma^0(U_\inc), 0, \dots,0).
  \end{aligned}
\end{equation}
We draw the attention of the reader to the strong analogy between \eqref{ExteriorCalderonSec7} and \eqref{ExteriorCalderonSec6}, the essential difference being that we are now dealing with multiple subdomains, i.e. $\Omega_0,\dots,\Omega_n$ instead of only $\Omega_0$. 
Now, plugging \eqref{ExteriorCalderonSec7} into the second term in the left-hand side of \eqref{VolumeVariational2FormSec7} leads to the \emph{single-trace FEM-BEM formulation}:
\begin{empheq}[box=\fbox]{equation}
  \label{eq:STF}
  \begin{aligned}
    & \text{Find } (U,\fru)\in \XSigmaGamma\text{ such that }\\
    & a_{\Sigma}(U,V)
    + \lbr \mathsf{A}(\fru),\Theta(\frv)\rbr_{\Gamma}
    + \lbr \Tr(\fru),\Tr(\frv)\rbr_{\Sigma}/2\\
    & = F_{\Sigma}(V) - \lbr \fru^\inc,\Theta(\frv)\rbr_{\Gamma}
    \quad \forall (V,\frv) \in \XSigmaGamma.
  \end{aligned}
\end{empheq}
Noticing the strong the similarities between \eqref{eq:STF} and \eqref{VarIdCostabel2}, we have just derived a generalization of the Costabel coupling~\eqref{VarIdCostabel2} to multi-domain settings.
The expanded expression for \eqref{eq:STF} reads: 
\begin{equation*}
  \begin{aligned}
    & \text{Find } (U,\fru)\in \XSigmaGamma\text{ such that }\\
    & \int_{\Omega_{\Sigma}}
    (\nabla U\cdot\nabla V-\kappa_{\Sigma}^2(\bx)UV) \,d\bx 
    + \sum_{j=0}^n \lbr \mathsf{A}^{j}_{\kappa_j}(\fru_j),\theta(\frv_j)\rbr_{\Gamma_j}
    + \lbr \Tr(\fru),\Tr(\frv)\rbr_{\Sigma}/2\\
    & = \int_{\Omega_{\Sigma}}fV\,d\bx \,- \lbr \gamma^0(U_\inc),\theta(\frv_0)\rbr_{\Gamma_0}
    \qquad \forall (V,\frv) \in \XSigmaGamma.
  \end{aligned}
\end{equation*}
Note that in this first multi-domain formulation the transmission conditions are imposed in strong form inside the function space $\XSigmaGamma$. Starting from \eqref{eq:STF}, a more flexible formulation will be designed in Section~\ref{sec:MTF}.

The link between the single-trace FEM-BEM formulation \eqref{eq:STF} and the transmission problem \eqref{eq:bvp} is examined in the following proposition. 
\begin{proposition}[Equivalence]
\label{pr:equivSTF}
If $\widetilde{U} \in \mH^1_\loc(\RR^d)$ solves \eqref{eq:bvp}, 
then the pair $(U,\fru) = ( \widetilde{U}|_{\Omega_\Sigma}, \gamma(\widetilde{U}) )$ solves \eqref{eq:STF}.
If $(U,\fru)\in \XSigmaGamma$ solves \eqref{eq:STF}, then the solution to \eqref{eq:bvp} is given by
\begin{equation}
\label{eq:UtildeSTF}
\begin{split}
\widetilde{U} (\bx) & \coloneqq U(\bx) \quad \text{ for } \bx \in \Omega_\Sigma, \\
\widetilde{U} (\bx) & \coloneqq \mathsf{G}_{\kappa_j}^j(\fru_j)(\bx) +
U_\inc(\bx)1_{\Omega_0}(\bx) \quad \text{ for } \bx \in \Omega_j, j=0,\dots,n.
\end{split}
\end{equation}
\end{proposition}
\begin{proof}
  We will follow closely the proof of Proposition \ref{EquivalenceCouplageCostabeln=0} established for the case $n=0$, except that we now have multiple subdomains $\Omega_j$. By similar arguments as in the beginning of that proof, it remains only to show that $\widetilde{U}$ given by \eqref{eq:UtildeSTF} complies with the transmission conditions of \eqref{eq:bvp}. For that we will use their characterization given by Lemma \ref{ReformulationTransmissionConditions}. 

  Considering an arbitrary test pair $(V,\frv)\in \XSigmaGamma$, and applying Green's formula in $\Omega_\Sigma$ leads to 
  \begin{equation}\label{eq:multidomainequiv1}
    a_\Sigma(U,V) 
    - \braket{\gamma_{\neu}^\Sigma \widetilde{U}, \gamma_{\dir}^\Sigma V}_\Sigma 
    = F_\Sigma(V) \quad
    \forall \, V \in \mH^1(\Omega_\Sigma).
  \end{equation}
  On the other hand, by applying the trace operator $\gamma^{j}$ on the second line of \eqref{eq:UtildeSTF} and using \eqref{eq:Akj12}, we get 
  $\gamma^{0}(\widetilde{U}) = \mathsf{A}^{0}_{\kappa_0}(\fru_0) + \fru_0/2 + \gamma^{0}(U_\inc)$ and 
  $\gamma^{j}(\widetilde{U}) = \mathsf{A}^{j}_{\kappa_j}(\fru_j) + \fru_j/2$ for $j=1\dots n$, that is, in compact notation, 
  $\gamma(\widetilde{U}) = \mathsf{A}(\fru) + \fru/2 + \fru^{\inc}$.
  Now we plug this and \eqref{eq:multidomainequiv1} into \eqref{eq:STF}, so we obtain
  \begin{equation}\label{eqmultistftransmission1}
    \begin{aligned}
      \braket{\gamma_{\neu}^\Sigma \widetilde{U}, \gamma_{\dir}^\Sigma V}_\Sigma +
      \lbr \gamma(\widetilde{U}), \Theta(\frv)\rbr_{\Gamma} 
      - \lbr \fru, \Theta(\frv)\rbr_{\Gamma}/2
      + \lbr \Tr(\fru),\Tr(\frv)\rbr_{\Sigma}/2 = 0\\
      \text{for all}\;(V,\frv)\in \XSigmaGamma.
    \end{aligned}
  \end{equation}
  By the polarity identity \eqref{PolarityOperatorTr} and \eqref{eq:thetaIdentitybis} we can write 
  \begin{equation*}
    \begin{split}
    - \lbr \fru, \Theta(\frv)\rbr_{\Gamma}/2 + \lbr \Tr(\fru),\Tr(\frv)\rbr_{\Sigma}/2 & = 
    \lbr \Tr(\fru),\Tr(\Theta(\frv))\rbr_{\Sigma}/2 + \lbr \Tr(\fru),\Tr(\frv)\rbr_{\Sigma}/2 \\
    & = \braket{\Tr_\dir(\fru),\Tr_\neu(\frv)}_\Sigma,
    \end{split}
  \end{equation*}
  so \eqref{eqmultistftransmission1} becomes 
  \begin{equation*}
    \braket{\gamma_{\neu}^\Sigma \widetilde{U}, \gamma_{\dir}^\Sigma V}_\Sigma 
      + \braket{\Tr_\dir(\fru),\Tr_\neu(\frv)}_\Sigma
      + \lbr \gamma(\widetilde{U}), \Theta(\frv)\rbr_{\Gamma}  
       = 0 \quad
      \text{for all}\;(V,\frv)\in \XSigmaGamma.
  \end{equation*}
  Moreover, since $(U,\fru) \in \XSigmaGamma$ and $\widetilde{U}\vert_{\Omega_{\Sigma}} = U $, we have $\Tr_{\dir}(\fru) = \gamma_{\dir}^{\Sigma}(U) = \gamma_{\dir}^{\Sigma}(\widetilde{U})$, and also $\gamma^{\Sigma}_{\dir}(V) = \Tr_{\dir}(\frv)$ because $(V,\frv)\in \XSigmaGamma$. 
  Therefore, by \eqref{eq:thetaIdentitybis3} and $\theta \circ \Tr = \Tr \circ \Theta$, we conclude that 
  \begin{equation*}
      \lbr \gamma^\Sigma (\widetilde{U}), \Tr(\Theta(\frv)) \rbr_\Sigma 
      + \lbr \gamma(\widetilde{U}), \Theta(\frv)\rbr_{\Gamma}  
       = 0 \quad 
      \text{for all}\;\frv\in \mbX(\Gamma).
  \end{equation*}
  Thanks to the variational characterization \eqref{eq:ReformulationTransmissionConditionsbis}, since $\Theta$ is an automorphism, we conclude that $\widetilde{U}$ satisfies the transmission conditions of Problem \eqref{eq:bvp}.
\end{proof}

The bilinear form $a_{\Sigma}(\cdot,\cdot)$ satisfies a G\r{a}rding inequality,
as well as $\lbr \mathsf{A}(\cdot),\Theta(\cdot)\rbr_{\Gamma}$, see \cite[§4.1]{Pet:STF:1989} and
\cite[Proposition 4.2]{ClHi:impenetrable:2015}. In addition
we have $\Re\{\lbr \Tr(\frv),\Tr(\overline{\frv})\rbr_{\Sigma}\} = 0$.
From these remarks we conclude that 
$\mathsf{a}_\textup{STF}\colon\XSigmaGamma \times \XSigmaGamma \to \CC$
defined as the bilinear form on the left-hand side of \eqref{eq:STF} satisfies a
\emph{G\r{a}rding inequality}.

\begin{proposition}[G\r{a}rding inequality] 
\label{pr:gardingSTF}
There exist a compact bilinear form $\mathcal{K} : \XSigmaGamma\times \XSigmaGamma \to \CC$
and a constant $\beta>0$ such that 
\begin{equation*}
\Re \left \{ \mathsf{a}_\textup{STF} \bigl( (V,\frv),(\overline{V},\overline{\frv}) \bigr)  + 
\mathcal{K}\bigl( (V,\frv),(\overline{V},\overline{\frv}) \bigr)  \right\}
\ge \,\beta\, (\lVert V \rVert_{\mH^1(\Omega_\Sigma)}^2 + \lVert \frv \rVert_{\mbH(\Gamma)}^2)
\end{equation*}
for all $(V,\frv) \in \XSigmaGamma$. 
\end{proposition}

As a consequence, the operator induced by $\mathsf{a}_\textup{STF}$ is of Fredholm type with index $0$ (see \cite[Theorem 2.33]{McLean:book:2000}), that is, formulation \eqref{eq:STF} has a unique solution for all $f \in \mL^2(\Omega_\Sigma)$, $U_\inc \in \mH^1_\loc(\RR^d)$ if and only if for $F_{\Sigma} \equiv 0$, $\fru^\inc=0$ it only has the trivial solution.  
Other important consequences of the G\r{a}rding inequality are, again \emph{in the case of injectivity} (see \cite[Theorems 4.2.9, 4.2.8]{SaSw:book:2011}): stability of the variational formulation \eqref{eq:STF} in the sense of an inf-sup condition; and, for Galerkin equations discretizing \eqref{eq:STF}, the validity of a discrete inf-sup condition, which implies well-posedness for the Galerkin equations and a quasi-optimal convergence of the Galerkin solutions to the exact solution. 

\subsection{Spurious resonances}
Unfortunately, like the classical Costabel coupling, the single-trace FEM-BEM formulation \eqref{eq:STF} may be affected by the {spurious resonances} phenomenon, 
that is, the associated operator may be not injective, whereas the transmission problem \eqref{eq:bvp} is always well-posed. 
Here we examine in which situations the spurious resonances phenomenon occurs. 
The following proposition identifies the injectivity condition, which depends on the wavenumbers \emph{and} on the geometric configuration. This condition turns out to be the same as in \cite[Theorem 4.8]{ClHi:impenetrable:2015}, which dealt with a partially impenetrable composite medium.

\begin{proposition}[Injectivity condition]
\label{pr:injSTF}
Let $(U,\fru) \in \XSigmaGamma$  solve formulation \eqref{eq:STF} with $F_{\Sigma} \equiv 0$, $\fru^\inc=0$. Then  $U=0$.  
We also have $\fru=0$ if the following additional condition is satisfied:
\begin{equation}
  \label{eq:injCondSTF}
  \Sigma \not \subset \Gamma_j \; \text{ or } \;
    \kappa_j \notin \mathfrak{S}(\Delta,\Omega_\Sigma) \quad 
    \text{for all } j=0,\dots,n.
\end{equation}
In the case where Condition \eqref{eq:injCondSTF} does not hold, there exists $\fru\in\mbX(\Gamma)\setminus\{0\}$ such that
$(0,\fru)\in \XSigmaGamma$ solves \eqref{eq:STF} with $F_{\Sigma} \equiv 0$, $\fru^\inc=0$.
\end{proposition}
\begin{proof}
By the equivalence Proposition~\ref{pr:equivSTF}, the function $\widetilde U$ defined in \eqref{eq:UtildeSTF} solves the homogeneous transmission problem \eqref{eq:bvp}, which is well-posed, so $\widetilde U = 0$. In particular, $U=\widetilde{U}|_{\Omega_\Sigma} =0$, and $\Tr_\dir(\fru) =\gamma_\dir^\Sigma U = 0$.  
Employing test functions $(V,\frv) \in \XSigmaGamma$ with $\gamma_\dir^\Sigma V = \Tr_\dir(\frv) = 0$ in formulation \eqref{eq:STF} with $F_{\Sigma} \equiv 0$, $\fru^\inc=0$, we obtain that $\fru$ satisfies 
\[
[\mathsf{A}(\fru), \Theta(\frv)] = 0, \quad \forall \, \frv \in \mbX(\Gamma) \text{ with } \Tr_{\dir}(\frv) = 0, 
\]
i.e.~$\fru \in \mbX(\Gamma)$ satisfies $\Tr_{\dir}(\fru) = 0$ and $[\mathsf{A}(\fru), \frv] = 0, \forall \, \frv \in \mbX(\Gamma)$ with $\Tr_{\dir}(\frv) = 0$, which is exactly the setting of \cite[Lemma 4.5, Lemma 4.6]{ClHi:impenetrable:2015}. As a consequence also \cite[Corollary 4.7]{ClHi:impenetrable:2015} holds true: if $\Sigma \not\subset{\Gamma_j}$ for all $j=0,\dots,n$, then for any choice of $\kappa_j > 0$ we have $\fru=0$. 
We also obtain that, if $\Sigma \subset \Gamma_j$ for a $j \in \{0,\dots,n\}$, then $\kappa_j \notin \mathfrak{S}(\Delta,\Omega_\Sigma)$ implies $\fru=0$, thanks to the reasoning in the third bullet in the proof of \cite[Theorem 4.8]{ClHi:impenetrable:2015}, which relies on \cite[Lemma 4.5, Lemma 4.6]{ClHi:impenetrable:2015}. 

Next, assuming that Condition \eqref{eq:injCondSTF} does not hold i.e.~
$\Sigma \subset \Gamma_i$ and $\kappa_i \in \mathfrak{S}(\Delta,\Omega_\Sigma)$
for a certain $i \in \{0,\dots,n\}$,
we construct $\fru\neq 0$ such that $(0,\fru)$ solves 
\eqref{eq:STF} with $F_{\Sigma} \equiv 0$, $\fru^\inc=0$.
Since $\Sigma \subset \Gamma_i$, by the geometric considerations in the first bullet in the proof of \cite[Theorem 4.8]{ClHi:impenetrable:2015}, we get that $\Omega_\Sigma$ is exactly one bounded connected component of $\RR^d\backslash\overline\Omega_i$, and in particular $\Omega_\Sigma$ is completely separated from the other subdomains $\Omega_j$, $j\ne i$:
\begin{equation}
\label{eq:4.10impPaper}
\overline\Omega_\Sigma \; \cap \bigcup_{j=0, j\ne i}^n \overline \Omega_j = \emptyset.
\end{equation}
Since $\kappa_i \in \mathfrak{S}(\Delta,\Omega_\Sigma)$, there exists $W \in \mH^1(\Omega_\Sigma)\backslash\{0\}$ such that  $-\Delta W - \kappa_i^2 W=0$ in $\Omega_\Sigma$ and $W=0$ on $\Sigma$. 
We consider $U^*=0 \in  \mH^1(\Omega_\Sigma)$, $u_i =0 \in \mH^{1/2}(\Gamma_i)$, and $p_i \in \mH^{-1/2}(\Gamma_i)$ with $p_i=0$ on $\Gamma_i \backslash \Sigma$ and $p_i = -\gamma_\neu^\Sigma W$ on $\Sigma$. 
We set $\fru_i^*= (u_i,p_i)$ and $\fru_j^* = (0,0)$ for $j\ne i, j=0,\dots,n$, thus by construction and \eqref{eq:4.10impPaper}, we have $\fru^* \in \mbX(\Gamma)$ and $\Tr_\dir(\fru^*) = u_i = 0 = \gamma_\dir^\Sigma U^*$, that is $(U^*,\fru^*) \in \XSigmaGamma$. 
If we evaluate the left-hand side of formulation \eqref{eq:STF} in $(U^*,\fru^*)$ we get: given any $(V,\frv) \in \XSigmaGamma$
\[
\begin{split}
&\left[\mathsf{A}_{\kappa_i}^i (\fru_i^*), \theta(\frv_i) \right]_{\Gamma_i} -  \frac{1}{2}\Braket{\Tr_\neu(\fru^*),\Tr_\dir(\frv)}_\Sigma = \\
&\left[\gamma^i \mathsf{G}_{\kappa_i}^i (\fru_i^*), \theta(\frv_i) \right]_{\Gamma_i} - \frac{1}{2} \left[\fru_i^*, \theta(\frv_i) \right]_{\Gamma_i} -  \frac{1}{2}\Braket{\Tr_\neu(\fru^*),\Tr_\dir(\frv)}_\Sigma = \\
  & \left[\gamma^i \mathsf{SL}_{\kappa_i}^i (p_i), \theta(\frv_i) \right]_{\Gamma_i} -
  \frac{1}{2} \braket{p_i, v_i}_{\Gamma_i} + \frac{1}{2} \Braket{p_i,v_i}_\Sigma,
\end{split}
\]
where we have used \eqref{eq:Akj12}, and  \eqref{eq:4.10impPaper} to write $\Tr_\dir(\frv) = v_i$, $\Tr_\neu(\fru^*) = -p_i$.
Now, the last two terms cancel each other out since by construction $p_i=0$ on $\Gamma_i \backslash \Sigma$. 
For the same reason in the first term $\mathsf{SL}_{\kappa_i}^i (p_i) = \mathsf{SL}_{\kappa_i}^\Sigma (p_i)$. 
Moreover, by the representation formula~\eqref{eq:intRep} on $\Omega_\Sigma$, for $\bx \in \RR^d\backslash\overline\Omega_\Sigma$ we have
\[
0 = \mathsf{G}^\Sigma_{\kappa_i}(\gamma^\Sigma W)(\bx) =  \mathsf{SL}_{\kappa_i}^\Sigma(\gamma_\neu^\Sigma W)(\bx) =  - \mathsf{SL}_{\kappa_i}^\Sigma (p_i)(\bx),
\]
 therefore $\gamma^i \mathsf{SL}_{\kappa_i}^i (p_i) = 0$ and $(U^*,\fru^*)$ is a non-trivial solution to formulation \eqref{eq:STF}. 
\end{proof}

Note that Corollary~\ref{cor:injCondCostabel} for the classical Costabel coupling is a particular case of the previous proposition, where $\Sigma \subset \Gamma_0$ (actually $\Sigma = \Gamma_0$). 
In the multi-domain configuration, surprising situations can arise, as shown in the next example.
\begin{figure}
\centering
\includegraphics[width=0.45\linewidth]{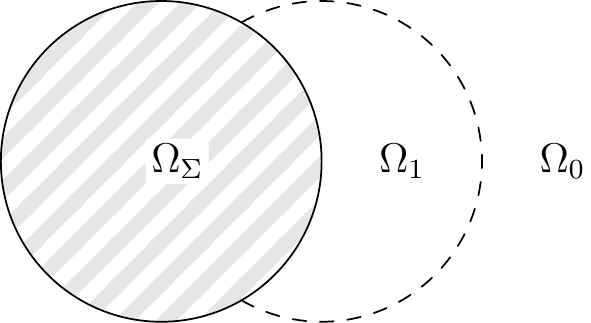}
\caption{Situation without spurious resonances.}
\label{fig:surprising}
\end{figure}
\begin{example}
Consider the transmission problem \eqref{eq:bvp} with $n=1$, i.e. $\RR^d = \overline{\Omega}_0 \cup \overline{\Omega}_1 \cup \overline{\Omega}_\Sigma$, but suppose that $\kappa_0 = \kappa_1$ so that the interface $\Gamma_0 \cap \Gamma_1$ is ``artificial''.  
In fact, the material configuration is the same as in the classical Costabel coupling, which is affected by spurious resonances if $\kappa_0 \in \mathfrak{S}(\Delta,\Omega_\Sigma)$. On the contrary, if we assume that the $(d-1)$-dimensional Hausdorff measure of $\Sigma \cap \Gamma_0$ and $\Sigma \cap \Gamma_1$ is strictly positive as in Figure~\ref{fig:surprising}, so that $\Sigma \not \subset \Gamma_1$ and $\Sigma \not \subset \Gamma_0$, then, no matter which is the value of $\kappa_0$, the corresponding single-trace FEM-BEM formulation \eqref{eq:STF} does not have spurious resonances! 
\end{example}

\section{Single-trace combined field FEM-BEM formulation}
\label{sec:STFcomb}

\noindent 
We have shown that the single-trace FEM-BEM formulation \eqref{eq:STF} is affected by spurious resonances when $\Sigma \subset \Gamma_i$ and $\kappa_i \in \mathfrak{S}(\Delta,\Omega_\Sigma)$ for a certain $i \in \{0,\dots,n\}$. 
As a remedy, we modify the boundary integral formulation on $\Gamma$ by adapting the approach of Combined Field Integral Equations (CFIE), first introduced in \cite{BuMi:CFIE:1971} for direct integral equations. The basic idea behind the CFIE approach is that Helmholtz boundary value problems with Robin (also called impedance) boundary conditions are always uniquely solvable, in contrast to interior pure Dirichlet (or pure Neumann) problems.  
The classical CFIEs thus rely on complex combinations of Dirichlet and Neumann traces, but neglecting the fact that they belong to different function spaces. Here, we adopt {regularized} CFIEs (see e.g.~\cite{BuHi:rCFIE:2005}), in which suitable compact operators map between Dirichlet and Neumann traces. 

For the transmission problem \eqref{eq:bvp} with $n=0$, a variational formulation based on regularized CFIEs and the Costabel coupling was proposed in \cite{HiMe:sfembem:2006}. Here, to extend to the multi-domain case this coupling formulation immune to spurious resonances, we adopt a procedure inspired by \cite[§5]{ClHi:impenetrable:2015} (where $\Omega_\Sigma$ represented an impenetrable part of the medium). 

\subsection{Regularizing operator and trace transformation operator}

The main step to obtain a combined field formulation that fixes \eqref{eq:STF} is to pick test functions satisfying generalized Robin conditions on $\Sigma$. These  conditions are based on a linear \emph{regularizing operator} $\mathsf{M} \colon \mH^{-1/2}(\Sigma) \to \mH^{+1/2}(\Sigma)$ that satisfies 
\begin{subequations}
\begin{align}
& \mathsf{M} \text{ is compact}, \label{eq:regopA}\\
& \Im\{\braket{\mathsf{M}\varphi, \overline{\varphi}}_\Sigma\} > 0 \quad \forall \, \varphi \in  \mH^{-1/2}(\Sigma) \backslash \{0\}. \label{eq:regopB}
\end{align}
\end{subequations}
For instance, if $\widetilde{\mathsf{M}}$ is any second order strongly coercive real symmetric surface differential operator on $\Sigma$, then $\mathsf{M} = \imagi \widetilde{\mathsf{M}}$ matches the two conditions above.
\begin{example}
  A concrete choice for such an operator was proposed in \cite[§4]{BuHi:rCFIE:2005}: $\mathsf{M} =  \imagi (-\Delta_\Sigma + \Id_\Sigma)^{-1} \colon \mH^{-1}(\Sigma) \to \mH^{1}(\Sigma)$, where $\Delta_\Sigma$ denotes the Laplace-Beltrami operator on $\Sigma$. 
  In this case, compactness of $\mathsf{M} : \mH^{-1/2}(\Sigma) \to \mH^{1/2}(\Sigma)$ follows from the continuity of $\mathsf{M} \colon \mH^{-1}(\Sigma) \to \mH^{1}(\Sigma)$ and the compact embeddings $\mH^{-1/2}(\Sigma) \subset \mH^{-1}(\Sigma)$ and $\mH^{1}(\Sigma) \subset \mH^{1/2}(\Sigma)$. 
  Note that to avoid evaluations of $\mathsf{M}$ in the resulting combined field formulation~\eqref{eq:CSTF}, one can reformulate \eqref{eq:CSTF} as a mixed variational formulation with auxiliary variables like in \cite[§5.4]{ClHi:impenetrable:2015}.
\end{example}
\noindent Invoking the duality of the spaces $\mH^{1/2}(\Sigma)$ and $\mH^{-1/2}(\Sigma)$, we can also define the adjoint regularizing operator $\mathsf{M}^* \colon \mH^{-1/2}(\Sigma) \to \mH^{+1/2}(\Sigma)$ by 
\[
\braket{\mathsf{M}^*p,q}_\Sigma \coloneqq \braket{\mathsf{M}q,p}_\Sigma \quad \text{for all } p,q \in \mH^{-1/2}(\Sigma).
\]
Note that $\mathsf{M}^*$ satisfies properties \eqref{eq:regopA}-\eqref{eq:regopB} if and only if $\mathsf{M}$ does. 
Now, given a regularizing operator $\mathsf{M}$, we define the subspace of $\mH^{1}(\Omega_{\Sigma})\times \mbX(\Gamma)$ satisfying generalized Robin conditions on $\Sigma$: 
\begin{equation}\label{DefEspaceRobin}
  \XSigmaGammaM \coloneqq
    \set{ (V,\frv) \in \mH^{1}(\Omega_{\Sigma})\times \mbX(\Gamma) | 
     \Tr_\dir(\frv) =
    \mathsf{M} \Tr_\neu(\frv) + \gamma^{\Sigma}_{\dir}(V) }.
\end{equation}
Please note the relationship between the space above and $\XSigmaGamma$ defined in \eqref{eq:XsigmaGamma}, whose elements satisfy instead Dirichlet conditions on $\Sigma$. In fact, as shown in the lemma below, the space $\XSigmaGammaM$ can be obtained as the image of the space $\XSigmaGamma$ through a trace transformation operator. Its definition involves the regularizing operator $\mathsf{M}$, and a bounded extension operator $\mathsf{E}_\Sigma \colon \mH^{1/2}(\Sigma) \to \mH^1(\RR^d)$ that provides a right inverse of the trace operator $\gamma_\dir^\Sigma$ (see e.g.~\cite[Lemma 3.36]{McLean:book:2000}). 
Then, we define the \emph{trace transformation operator} 
\begin{equation}
  \label{eq:R-C}
  \begin{aligned}
    & \mathsf{R} \colon \mH^{1}(\Omega_{\Sigma})\times \mbX(\Gamma)\to \mH^{1}(\Omega_{\Sigma})\times \mbX(\Gamma)\\
    & \mathsf{R} (V,\frv):= (V,\frv + \mathsf{C}(\frv))\\
    & \text{with } \mathsf{C}(\frv) \coloneqq (\gamma_\dir^j
    \circ \mathsf{E}_\Sigma \circ \mathsf{M} \circ \Tr_\neu(\frv), \, 0)_{j=0}^n,
  \end{aligned}
\end{equation}
where $\mathsf{C} \colon \mbX(\Gamma) \to \mbX(\Gamma)$ inherits compactness from $\mathsf{M}$. 
Since $\mathsf{C}^2 = 0$, we have $\mathsf{R}^{-1} (V,\frv) = (V,\frv - \mathsf{C}(\frv))$, and $\mathsf{R}$ is an isomorphism. 
We can prove the following lemma, which is a variant of \cite[Lemma 5.2]{ClHi:impenetrable:2015}. 
\begin{lemma}[Trace transformation]
\label{lem:traceTrasf}
$\mathsf{R}(\XSigmaGamma) = \XSigmaGammaM$. 
\end{lemma}
\begin{proof}
Recalling the definitions of $\Tr$ in \eqref{eq:defT}, of $\gamma$ in \eqref{GlobalTraceOperator}, and of $\mathsf{E}_\Sigma$ above, we have that the operator $\Tr_\dir \circ \gamma \circ \mathsf{E}_\Sigma$ is the identity on $\mH^{1/2}(\Sigma)$, hence $\Tr_\dir \mathsf{C} = \mathsf{M} \Tr_\neu$. Note also that $\Tr_\neu \mathsf{C} = 0$. 
Therefore, if $(V,\frv) \in \XSigmaGamma$ we have 
$\Tr_\dir(\frv + \mathsf{C}(\frv)) 
= \gamma^{\Sigma}_{\dir}(V) + \mathsf{M} \Tr_\neu(\frv)
= \gamma^{\Sigma}_{\dir}(V) + \mathsf{M} \Tr_\neu(\frv + \mathsf{C}(\frv))$, 
which shows that $\mathsf{R}(\XSigmaGamma) \subset \XSigmaGammaM$. 
Now let  $(V,\frv) \in \XSigmaGammaM$, then 
$\Tr_\dir(\frv - \mathsf{C}(\frv)) =
\mathsf{M} \Tr_\neu(\frv) + \gamma^{\Sigma}_{\dir}(V) -
\mathsf{M} \Tr_\neu(\frv) = \gamma^{\Sigma}_{\dir}(V)$. 
Hence $\mathsf{R}^{-1}(\XSigmaGammaM) \subset \XSigmaGamma$.
\end{proof}

\subsection{The formulation}
In order to obtain the new combined field formulation, we proceed in a manner similar to Section~\ref{sec:STF}, this time choosing test pairs $(V',\frv')$ in $\XSigmaGammaM$ instead of $\XSigmaGamma$.
Again the transmission problem \eqref{eq:bvp} with solution $U\in \mH^{1}_{\loc}(\RR^d)$ will be reformulated as a coupled problem with solution
\begin{equation}
  \begin{aligned}
    & (U\vert_{\Omega_{\Sigma}}, \fru)\in \XSigmaGamma\\
    & \text{where}\quad \fru = \gamma(U) = (\gamma^{0}(U),\dots,\gamma^{n}(U)).
  \end{aligned}
\end{equation}
For $(V',\frv')\in \XSigmaGammaM$, applying Green's formula in $\Omega_{\Sigma}$ leads to $a_{\Sigma}(U,V')- \langle \gamma_{\neu}^{\Sigma}(U),\gamma_{\dir}^{\Sigma}(V') \rangle_{\Sigma} = F_{\Sigma}(V')$, as in Section~\ref{sec:STF}. 
Next, we transform the boundary term following steps similar to \eqref{VolumeVariational1FormSec7-2}, but with an extra term since here $\gamma_{\dir}^{\Sigma}(V') = \Tr_{\dir}(\frv')-\mathsf{M} \Tr_\neu(\frv')$:   
\begin{equation*}
  \begin{aligned}
     - \langle \gamma_{\neu}^{\Sigma}(U),\gamma_{\dir}^{\Sigma}(V') \rangle_{\Sigma}
    & = -\langle \Tr_{\neu}(\fru),\Tr_{\dir}(\frv')\rangle_{\Sigma} +
    \langle \Tr_{\neu}(\fru),\mathsf{M}\Tr_{\neu}(\frv') \rangle_{\Sigma} \\
    & = \lbr \fru,\Theta(\frv')\rbr_{\Gamma}/2 + \lbr \Tr(\fru),\Tr(\frv')\rbr_{\Sigma}/2 +
    \langle \Tr_{\neu}(\fru),\mathsf{M}\Tr_{\neu}(\frv') \rangle_{\Sigma},
  \end{aligned}
\end{equation*}
that is, by the boundary integral representations in the subdomains $\Omega_0,\dots,\Omega_n$ summarized by \eqref{ExteriorCalderonSec7}, 
\begin{equation}\label{eq:boundarytermCFIE}
  \begin{aligned}
    & - \langle \gamma_{\neu}^{\Sigma}(U),\gamma_{\dir}^{\Sigma}(V') \rangle_{\Sigma}\\
    & = \lbr \mathsf{A} (\fru),\Theta(\frv')\rbr_{\Gamma} + \lbr\fru^\inc,\Theta(\frv')\rbr_{\Gamma}
    + \lbr \Tr(\fru),\Tr(\frv')\rbr_{\Sigma}/2 +
    \langle \Tr_{\neu}(\fru),\mathsf{M}\Tr_{\neu}(\frv') \rangle_{\Sigma}.
  \end{aligned}
\end{equation}
Now, according to the parametrization of $\XSigmaGammaM$ in Lemma \ref{lem:traceTrasf}, we have 
$(V',\frv') = \mathsf{R}(V,\frv) = (V,(\Id+\mathsf{C})\frv)$ for $(V,\frv)\in \XSigmaGamma$,  
and this representation can be injected into \eqref{eq:boundarytermCFIE}:
\begin{equation*}
  \begin{aligned}
    - \langle \gamma_{\neu}^{\Sigma}(U),\gamma_{\dir}^{\Sigma}(V) \rangle_{\Sigma}
    & = \lbr \mathsf{A} (\fru),\Theta(\frv)\rbr_{\Gamma} + \lbr \mathsf{A} (\fru),\Theta\mathsf{C}(\frv)\rbr_{\Gamma}
    + \lbr\fru^\inc,\Theta(\Id+\mathsf{C})\frv)\rbr_{\Gamma} \\
    &+ \lbr \Tr(\fru),\Tr(\frv)\rbr_{\Sigma}/2 + \lbr \Tr(\fru),\Tr\mathsf{C}(\frv)\rbr_{\Sigma}/2 
    + \langle \Tr_{\neu}(\fru),\Tr_{\dir}\mathsf{C}(\frv) \rangle_{\Sigma}, 
  \end{aligned}
\end{equation*}
where for the last term we have used $\mathsf{M} \Tr_\neu = \Tr_\dir \mathsf{C}$ and $\Tr_\neu \mathsf{C} = 0$. 
Moreover, by \eqref{eq:thetaIdentity} and \eqref{PolarityOperatorTr} we can rewrite the sum of the last two terms in the equation above as
\[
  \lbr \Tr(\fru), \Tr\mathsf{C}(\frv)\rbr_{\Sigma}/2 
  + \langle \Tr_{\neu}(\fru),\Tr_{\dir}\mathsf{C}(\frv) \rangle_{\Sigma}  
  = \lbr \Tr(\fru), \theta \Tr\mathsf{C}(\frv)\rbr_{\Sigma}/2 
  = - \lbr \fru, \Theta\mathsf{C}(\frv)\rbr_{\Gamma}/2.
\]
In conclusion, summing up and defining the source term $\widetilde{\fru}^\inc \in \mbH(\Gamma)$ 
and the bilinear form $\mathsf{c} \colon \mbX(\Gamma) \times  \mbX(\Gamma) \to \CC$ 
\begin{equation}\label{eq:cBilForm}
  \begin{aligned}
    & [\widetilde{\fru}^\inc, \frv]_{\Gamma}
    \coloneqq [\fru^\inc, \Theta(\Id+\mathsf{C})\frv]_{\Gamma}, \\
    & \mathsf{c}(\frw, \frv)
    \coloneqq [(\mathsf{A}-\Id/2)\frw, \Theta\mathsf{C}\frv]_{\Gamma},
  \end{aligned}
\end{equation}
we obtain the formulation 
\begin{empheq}[box=\fbox]{equation}
  \label{eq:CSTF}
  \begin{split}
    &\text{Find } (U,\fru)\in \XSigmaGamma \text{ such that} \\
    & a_\Sigma(U,V) + [\mathsf{A}(\fru), \Theta(\frv)]_{\Gamma}
    + \mathsf{c}(\fru, \frv) +\left[\Tr(\fru),\Tr(\frv)\right]_{\Sigma}/2 \\
    & = F_\Sigma(V)  - [\widetilde{\fru}^\inc, \frv]_{\Gamma}
    \quad \forall \, (V,\frv)\in\XSigmaGamma,
  \end{split}
\end{empheq}
which we dub \emph{single-trace combined field FEM-BEM formulation}
because the test pairs that we have considered for its derivation
comply with an impedance condition on $\Sigma$, see
\eqref{DefEspaceRobin}.

Formulation \eqref{eq:CSTF} differs from \eqref{eq:STF} by terms
involving the operator $\mathsf{C}$ only, which is compact. Hence, a
\emph{G\r{a}rding inequality} analogue to
Proposition~\ref{pr:gardingSTF} also holds for \eqref{eq:CSTF}. The
additional benefit of using \eqref{DefEspaceRobin} is to eliminate the
spurious resonance phenomenon and to yield systematic unique
solvability. To prove this, we start by establishing an intermediate
lemma.

\begin{lemma}\label{IntermediateLemma}
  Let $ (U,\fru)\in \XSigmaGamma$ solve formulation \eqref{eq:CSTF}
  with $F_{\Sigma} \equiv 0$, $\tilde{\fru}^\inc=0$. Then we have
  \begin{equation}\label{IntermediateLemmaIdentityToEstablish}
    \langle \gamma_{\neu}^{\Sigma}(U)-\Tr_{\neu}(\fru), \gamma_{\dir}^{\Sigma}(V)\rangle_{\Sigma} +
      \lbr(\mathsf{A}-\Id/2)\fru,\Theta(\frv)\rbr_{\Gamma} = 0 \quad
      \forall \, (V,\frv)\in \XSigmaGammaM.
  \end{equation}
\end{lemma}
\begin{proof}
  The proof essentially consists in rewinding the derivation of \eqref{eq:CSTF}
  in reverse order. First of all, observe that for $V\in \mH^{1}_{0}(\Omega_{\Sigma})$ we have $(V,0)\in \XSigmaGamma$. With this choice of test pairs
  we obtain $a_{\Sigma}(U,V) = 0$ for all $V\in\mH^1_0(\Omega_\Sigma)$, 
  which leads to $\Delta U + \kappa_{\Sigma}^{2}U = 0$ in $\Omega_{\Sigma}$.
  As a consequence we have $a_{\Sigma}(U,V) = \langle \gamma_{\neu}^{\Sigma}(U),
  \gamma_{\dir}^{\Sigma}(V)\rangle_{\Sigma}$ for any $V\in \mH^{1}(\Omega_{\Sigma})$.
  Coming back to \eqref{eq:CSTF} with homogeneous right-hand side, we obtain 
  \begin{equation*}
    0 = \langle \gamma_{\neu}^{\Sigma}(U),\gamma_{\dir}^{\Sigma}(V)\rangle_{\Sigma}
      + [\mathsf{A}(\fru), \Theta(\frv)]_{\Gamma} + \mathsf{c}(\fru, \frv)
      + \left[\Tr(\fru),\Tr(\frv)\right]_{\Sigma}/2 \quad
      \forall \, (V,\frv)\in \XSigmaGamma.
  \end{equation*}
  Next plugging the definition of $\mathsf{c}$ provided by \eqref{eq:cBilForm} into the expression
  above, for a given $(V,\frv)\in \XSigmaGamma$ we obtain 
  \begin{equation}\label{LemmaIntermediateResultInjectivityCFIE}
    0 = \langle \gamma_{\neu}^{\Sigma}(U),\gamma_{\dir}^{\Sigma}(V)\rangle_{\Sigma}
       + \lbr(\mathsf{A}-\Id/2)\fru,\Theta(\Id+\mathsf{C})\frv\rbr_{\Gamma} 
       + \left[\fru,\Theta(\frv)\right]_{\Gamma}/2 + \left[\Tr(\fru),\Tr(\frv)\right]_{\Sigma}/2.
  \end{equation}
  Next, since $\fru,\frv\in \mbX(\Gamma)$, we have $\left[\fru,\Theta(\frv)\right]_{\Gamma}
  = -\left[\Tr(\fru),\theta\Tr(\frv)\right]_{\Sigma}$. By \eqref{eq:thetaIdentity},
  we conclude that $\left[\fru,\Theta(\frv)\right]_{\Gamma} + \left[\Tr(\fru),\Tr(\frv)\right]_{\Sigma}
  = - 2\langle \Tr_{\neu}(\fru),\Tr_{\dir}(\frv)\rangle_{\Sigma}$. In addition, we have $\Tr_{\dir}(\frv) =
  \gamma_{\dir}^{\Sigma}(V)$ since  $(V,\frv)\in \XSigmaGamma$. Plugging this into
  \eqref{LemmaIntermediateResultInjectivityCFIE} leads to the identity 
  \begin{equation*}
    0 = \langle \gamma_{\neu}^{\Sigma}(U)-\Tr_{\neu}(\fru),\gamma_{\dir}^{\Sigma}(V)\rangle_{\Sigma}
      + \lbr(\mathsf{A}-\Id/2)\fru,\Theta(\Id+\mathsf{C})\frv\rbr_{\Gamma} \quad 
      \forall (V,\frv)\in \XSigmaGamma.
  \end{equation*}
  To finish the proof there only remains to apply Lemma \ref{lem:traceTrasf}
\end{proof}

\begin{proposition}[Injectivity]
  Let $ (U,\fru)\in \XSigmaGamma$ solve formulation \eqref{eq:CSTF}
  with $F_{\Sigma}\equiv 0$, $\tilde{\fru}^\inc=0$. Then $U=0$, $\fru=0$.
\end{proposition}
\begin{proof}
  Consider the space $\mbX_{\mathsf{M}}(\Gamma):=\Set{\frv\in \mbX(\Gamma) | 
  \Tr_{\dir}(\frv) = \mathsf{M}\Tr_{\neu}(\frv)}$ and observe that
  we have $(0,\frv)\in \XSigmaGammaM$ for any $\frv
  \in \mbX_{\mathsf{M}}(\Gamma)$. As a consequence we can apply Lemma
  \ref{IntermediateLemma} above choosing $(V,\frv) = (0,\frv)$
  with $\frv\in \mbX_{\mathsf{M}}(\Gamma)$, so we obtain that
  \begin{equation*}
    \lbr(\mathsf{A}-\Id/2)\fru,\Theta(\frv)\rbr_{\Gamma} = 0\quad
    \forall \, \frv\in \mbX_{\mathsf{M}}(\Gamma).
  \end{equation*}
  Let us denote $\frw := (\mathsf{A}-\Id/2)\fru$. Considering
  any $\frv\in \mbX(\Gamma)$ such that $\Tr(\frv) = 0$,
  we have $\Theta(\frv) \in \mbX(\Gamma)$ with $\Tr(\Theta(\frv)) = 0$
  so that $\Theta(\frv)\in \mbX_{\mathsf{M}}(\Gamma)$ and
  $\lbr \frw,\frv\rbr_{\Gamma} = \lbr(\mathsf{A}-\Id/2)\fru,
  \Theta\circ\Theta(\frv)\rbr_{\Gamma} = 0$. Applying Lemma \ref{LemmaModifiedPolarity},
  we conclude that  $\frw\in \mbX(\Gamma)$. So, by \eqref{PolarityOperatorTr}, 
  $\lbr \Tr(\frw), \theta\Tr(\frv)\rbr_{\Sigma} = 0\; \forall
  \frv\in \mbX_{\mathsf{M}}(\Gamma)$, that is $\langle
  \Tr_{\dir}(\frw) + \mathsf{M}^* \Tr_{\neu}(\frw), \Tr_{\neu}(\frv)
  \rangle_\Sigma = 0$ $\forall \frv\in \mbX_{\mathsf{M}}(\Gamma)$,
  which implies $\Tr_{\dir}(\frw) = -\mathsf{M}^*\Tr_{\neu}(\frw)$, as
  $\Tr_{\neu}$ is surjective.  From this and by \eqref{eq:regopB} we
  conclude that
  \begin{equation*}
    \begin{aligned}
      0\leq
      & \;2\Im\{\langle \mathsf{M}^*\Tr_{\neu}(\frw),\Tr_{\neu}(\overline{\frw})\rangle_{\Sigma}\}\\
      & = -2\Im\{\langle\Tr_{\dir}(\frw),\Tr_{\neu}(\overline{\frw})\rangle_{\Sigma}\} \\
      & = -\Im\{\lbr\Tr(\frw),\Tr(\overline{\frw})\rbr_{\Sigma}\}  =
      \Im\{\lbr \frw,\overline{\frw}\rbr_{\Gamma}\}.
    \end{aligned}
  \end{equation*}
  Moreover, by construction, since $\mathsf{A}^2 = \Id/4$, we have $(\mathsf{A}+\Id/2)\frw =
  (\mathsf{A}+\Id/2)(\mathsf{A}-\Id/2)\fru = 0$, so we can write  
  $\lbr \frw,\overline{\frw}\rbr_{\Gamma}/2 = -\lbr \mathsf{A}(\frw),\overline{\frw}\rbr_{\Gamma}$.
  Therefore, we deduce that 
  $0\leq \Im\{\langle \mathsf{M}^*\Tr_{\neu}(\frw),\Tr_{\neu}(\overline{\frw})\rangle_{\Sigma}\}
  = - \Im \lbr \mathsf{A}(\frw),\overline{\frw}\rbr_{\Gamma}\leq 0$ by applying
  Proposition \ref{RadiationConditionConsequence} for the last inequality. Hence, we obtain 
  $\Im\{\langle \mathsf{M}^*\Tr_{\neu}(\frw),\Tr_{\neu}(\overline{\frw})\rangle_{\Sigma}\} = 0$.
  Next \eqref{eq:regopB} yields $\Tr_{\neu}(\frw) = 0$ and, since $\Tr_{\dir}(\frw) = -\mathsf{M}^*\Tr_{\neu}(\frw)$, 
  we finally obtain $\Tr(\frw) = 0$. This implies that  $\lbr\frw,\Theta(\frv)\rbr_{\Gamma} =
  -\lbr\Tr(\frw),\theta\Tr(\frv)\rbr_{\Sigma} = 0\;\forall \frv\in\mbX(\Gamma)$, which rewrites
  \begin{equation*}
    \lbr(\mathsf{A}-\Id/2)\fru,\frv\rbr_{\Gamma} = 0\quad \forall \, \frv\in\mbX(\Gamma).
  \end{equation*}
  Therefore, the second term in \eqref{IntermediateLemmaIdentityToEstablish} vanishes for all $(V,\frv)\in \XSigmaGammaM$,
  and by Lemma~\ref{IntermediateLemma} we conclude that $\gamma_{\neu}^{\Sigma}(U) = \Tr_{\neu}(\fru)$, 
  which implies $\gamma^{\Sigma}(U) = \Tr(\fru)$ since $(U,\fru)\in \XSigmaGamma$
  by assumption. On the other hand we have
  \begin{equation*}
    \begin{aligned}
      0
      & = \lbr(\mathsf{A}-\Id/2)\fru,\frv\rbr_{\Gamma}
      = \lbr(\mathsf{A}+\Id/2)\fru,\frv\rbr_{\Gamma} - \lbr \fru,\frv\rbr_{\Gamma}\\
      & = \lbr(\mathsf{A}+\Id/2)\fru,\frv\rbr_{\Gamma} + \lbr \Tr(\fru),\Tr(\frv)\rbr_{\Sigma}
       = \lbr(\mathsf{A}+\Id/2)\fru,\frv\rbr_{\Gamma}\\
      & \hspace{0.5cm} \text{for all $\frv\in\mbX(\Gamma)$ such that $\Tr(\frv) = 0$.}
    \end{aligned}
  \end{equation*}
  From this last equality, applying Lemma \ref{LemmaModifiedPolarity},
  we obtain that $(\mathsf{A}+\Id/2)\fru\in\mbX(\Gamma)$.
  Moreover, since we established that $\Tr(\fru) = \gamma^{\Sigma}(U)$
  and $\Tr(\frw) = \Tr((\mathsf{A}-\Id /2)\fru) = 0$, we obtain
  \begin{equation*}
    \Tr((\mathsf{A}+\Id/2)\fru) = \Tr(\frw) + \Tr(\fru) = \gamma^{\Sigma}(U).
  \end{equation*}
  Finally, let us define $\widetilde{U}\in \mL^{2}_{\loc}(\RR^d)$ by
  $\widetilde{U}(\bx) = U(\bx)$ for $\bx\in\Omega_{\Sigma}$, and
  $\widetilde{U}(\bx) = \mathsf{G}^j_{\kappa_j}(\fru_j)(\bx)$ for
  $\bx\in \Omega_j, j = 0,\dots,n$. By construction we have
  \begin{equation*}
    \begin{aligned}
      & \Delta \widetilde{U} + \kappa_{\Sigma}^{2}\widetilde{U} = 0\quad \text{in}\;\Omega_{\Sigma}\\
      & \Delta \widetilde{U} + \kappa_{j}^{2}\widetilde{U} = 0\quad \text{in}\;\Omega_{j}\;\forall j=0\dots n\\
      & \text{$\widetilde{U}$ is $\kappa_0$-outgoing radiating.}
    \end{aligned}
  \end{equation*}
  Let us prove that $\widetilde{U}$ satisfies the Neumann and Dirichlet transmission conditions
  through the skeleton of the subdomain partition. Using equation \eqref{eq:Akj12}, we have established that
  $\gamma(\widetilde{U}) =
  (\gamma^{j}\mathsf{G}^j_{\kappa_j}(\fru_j))_{j=0\dots n} =
  (\mathsf{A}+\Id/2)\fru \in\mbX(\Gamma)$ on the one hand, and
  \begin{equation*}
    \Tr(\gamma(\widetilde{U})) = \Tr((\mathsf{A}+\Id/2)\fru) =
    \gamma^{\Sigma}(U) = \gamma^{\Sigma}(\widetilde{U}).
  \end{equation*}
  Hence, by Lemma \ref{ReformulationTransmissionConditions}
  we see that $\widetilde{U}$ is solution to the transmission problem
  \eqref{eq:bvp} with zero right-hand side.  Since this boundary value
  problem admits a unique solution $\widetilde{U} \equiv 0$, we get
  $U = 0$, $\Tr(\fru) = \gamma^{\Sigma}(U) = 0$ and $(\mathsf{A}+\Id/2)\fru = 0$,
  which implies in particular
  \begin{equation*}
    \begin{aligned}
      & \fru\in \mbX(\Gamma)\;\text{with}\;\Tr(\fru) = 0\\
      & \text{and}\;\;\lbr \mathsf{A}(\fru),\frv\rbr_\Gamma = 0
      \;\forall \frv\in \mbX(\Gamma).
    \end{aligned}
  \end{equation*}
  According to \cite[Thm.~4.1]{Pet:STF:1989} or \cite[Prop.~A.1]{ClHi:MTF:2013},
  the homogeneous formulation above has a unique solution, hence finally $\fru = 0$. 
\end{proof}

\section{Multi-trace FEM-BEM formulation}
\label{sec:MTF}

\begin{figure}
\centering
\includegraphics[width=0.9\linewidth]{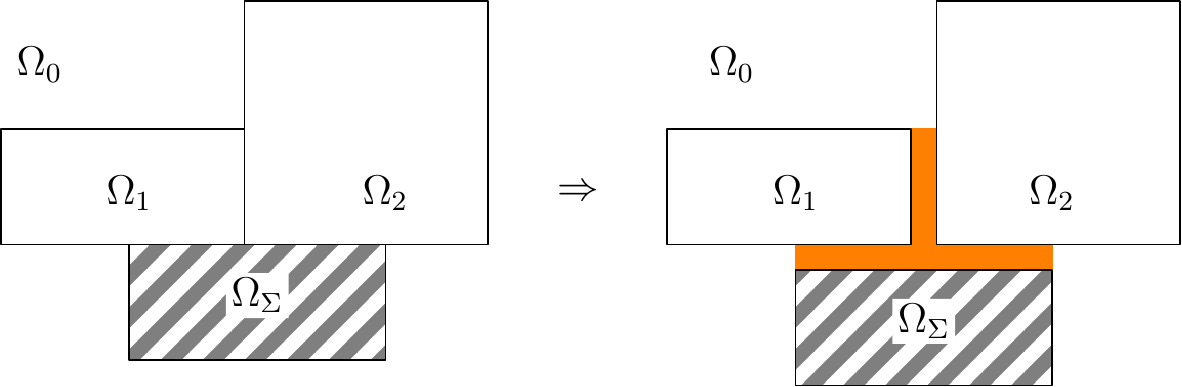}
\caption{Illustration of the gap idea (the gap is highlighted in orange)}.
\label{fig:gap}
\end{figure}
\noindent 
Single-trace formulations are not very flexible because the spaces $\XSigmaGamma$ and $\mbX(\Gamma)$ contain the transmission conditions in strong form, which constitutes an obstacle to operator preconditioning \cite{MR3202533}. Multi-trace formulations are designed to tackle this issue. 

As in \cite[§5]{ClHi:MTF:2013} and \cite[§6]{ClHi:impenetrable:2015}, the heuristic idea is to act as if the single-trace FEM-BEM formulation~\eqref{eq:STF} were applied to gap configurations with vanishing gap, see Figure~\ref{fig:gap}: the subdomains $\Omega_\Sigma$, $\Omega_j$, $j=1,\dots,n$, are torn apart and an (infinitely) thin gap, filled with the same propagation medium as $\Omega_0$, is introduced, so that all bounded subdomains are isolated from each other. Although the geometric limit process can not be rigorously described, the gap idea is useful to get a first insight about the properties satisfied by the multi-trace formulation based on those of the single-trace formulation (like Propositions~\ref{pr:gardingMTF} and \ref{pr:injMTF}).

In the gap setting (Figure~\ref{fig:gap}, right), the boundary of $\Omega_0$ can be partitioned as $\Gamma_0 = \cup_{j=1}^n \Gamma_j \,\cup \, \Sigma$, and the sigle-trace space $\XSigmaGamma$ is isomorphic to the space $\mH^1(\Omega_\Sigma) \times \widehat{\mbH}(\Gamma)$, where the \emph{multi-trace space} $\widehat{\mbH}(\Gamma)$, introduced in \cite[§6.1]{ClHi:impenetrable:2015}, is defined as 
\begin{equation}\label{eq:MTspaceNeu}
\widehat{\mbH}(\Gamma) \coloneqq \mbH(\Gamma_1) \times \cdots \times \mbH(\Gamma_n) \times \mH^{-1/2}(\Sigma).
\end{equation}
The isomorphism is given by the map $s \colon (U,\fru) \mapsto \left ( U, (\fru_1, \dots, \fru_n, \Tr_\neu(\fru) ) \right)$, whose inverse, in the gap setting, is the map $t \colon \left (U, (\hat{\fru}_1,\dots,\hat{\fru}_n, p_\Sigma) \right) \mapsto \left (U, (\tilde{\fru}_0, \hat{\fru}_1,\dots,\hat{\fru}_n) \right )$, where 
\[
\tilde{\fru}_0 (\bx) \coloneqq  
\begin{cases}
\phi(\hat{\fru}_j)(\bx), & \bx \in \Gamma_j, \,j=1,\dots,n, \\
\phi({\gamma_\dir^\Sigma U},{p_\Sigma})(\bx), & \bx \in \Sigma.
\end{cases}
\qquad \text{ with } \phi(u,p) \coloneqq (u,-p).
\]
The multi-trace space $\widehat{\mbH}(\Gamma)$ differs from the multi-trace space ${\mbH}(\Gamma)$ defined in \eqref{MultiTraceSpace} since it does not contain any contribution on $\Gamma_0$. Instead, it includes Neumann traces on $\Sigma$. 
It will enter the functional framework for the global multi-trace formulation, also for general geometrical settings, such as in Figure~\ref{fig:gap}, left. 
Note that the unknown traces are doubled on each interface that separates two (bounded) subdomains, hence the attribute \emph{multi-trace}. 
We equip the space $\widehat{\mbH}(\Gamma)$ with the standard norm of the Cartesian product:  
\begin{equation*}
  \lVert \hat{\frv} \rVert_{\widehat{\mbH}(\Gamma)}^2 \coloneqq 
  \sum_{j=1}^n \lVert \hat{\frv}_j\rVert_{\mbH(\Gamma_j)}^{2} 
    + \lVert q_\Sigma \rVert_{\mH^{-1/2}(\Sigma)}^2, \quad 
  \text{for } \hat{\frv} = (\hat{\frv}_1,\dots,\hat{\frv}_n, q_\Sigma) \in \widehat{\mbH}(\Gamma).
\end{equation*}
The dual space of $\widehat{\mbH}(\Gamma)$ is the space $\widecheck{\mbH}(\Gamma) \coloneqq  \mbH(\Gamma_1) \times \cdots \times \mbH(\Gamma_n) \times \mH^{1/2}(\Sigma)$, with respect to the duality pairing 
\begin{equation}\label{eq:pairingHatCheck}
\llbracket \check{\fru}, \hat{\frv} \rrbracket \coloneqq  
\sum_{j=1}^n[\check{\fru}_j,\hat{\frv}_j]_{\Gamma_j} 
+ \braket{u_\Sigma,q_\Sigma}_\Sigma, 
\end{equation}
for $\check{\fru}= (\check{\fru}_1,\dots,\check{\fru}_n, u_\Sigma) \in \widecheck{\mbH}(\Gamma)$, $\hat{\frv}= (\hat{\frv}_1,\dots,\hat{\frv}_n,q_\Sigma) \in \widehat{\mbH}(\Gamma)$. 

For notational convenience it is useful to introduce also a multi-trace space that includes both Dirichlet and Neumann traces on $\Sigma$, but no components on $\Gamma_0$: 
\begin{equation}\label{eq:MTspaceFullSigma}
  \wdoublehat{\mbH}(\Gamma) \coloneqq  \mbH(\Gamma_1) \times \cdots \times \mbH(\Gamma_n) \times \mbH(\Sigma),
\end{equation} 
with the skew-symmetric duality pairing 
\begin{equation}\label{eq:pairingFullSigma}
  \lBrace \doublehat{\fru}, \doublehat{\frv} \rBrace \coloneqq 
  \sum_{j=1}^n [\doublehat{\fru}_j,\doublehat{\frv}_j]_{\Gamma_j} 
  + [\doublehat{\fru}_\Sigma,\doublehat{\frv}_\Sigma]_\Sigma, 
\end{equation}
for $\doublehat{\fru}= (\doublehat{\fru}_1,\dots,\doublehat{\fru}_n,\doublehat{\fru}_\Sigma)$, $\doublehat{\frv}= (\doublehat{\frv}_1,\dots,\doublehat{\frv}_n,\doublehat{\frv}_\Sigma) \in \wdoublehat{\mbH}(\Gamma)$.

\subsection{Derivation of the formulation}

The multi-trace FEM-BEM formulation can be seen as the single-trace FEM-BEM formulation \eqref{eq:STF} applied to gap configurations with vanishing gap. However, it is difficult to study the vanishing gap limit with a rigorous mathematical argument. Following the idea in \cite[§8]{ClHi:MTF:2013}, the multi-trace formulation is rather obtained by trying to eliminate from the single-trace formulation \eqref{eq:STF} all the contributions on $\Gamma_0$. Essentially this is achieved by exploiting repeatedly the modified polarity identity \eqref{PolarityOperatorTr} and the variational characterization of transmission conditions \eqref{eq:ReformulationTransmissionConditionsbis}. 

We first reshape the right-hand side of formulation \eqref{eq:STF}, more precisely the term $-[\fru^\inc,\Theta(\frv)]_\Gamma$, where $\fru^\inc = (\gamma^0U_\inc, 0, \dots,0)$.
Since $U_\inc \in \mH_{\loc}(\Delta,\RR^d)$, we can apply \eqref{eq:ReformulationTransmissionConditionsbis} 
to write, for $\frv \in \mbX(\Gamma)$ (from which $\Theta(\frv) \in \mbX(\Gamma)$), 
\begin{equation}
\label{eq:rhsMTF0}
\begin{split}
-[\fru^\inc,\Theta(\frv)]_\Gamma &= 
-[\gamma^0 U_\inc,\theta(\frv_0)]_{\Gamma_0} + [\gamma(U_\inc), \Theta(\frv)]_\Gamma + [\gamma^\Sigma U_\inc, \Tr(\Theta(\frv))]_\Sigma \\
&= \sum_{j=1}^n [\gamma^j U_\inc, \theta(\frv_j)]_{\Gamma_j} + \left [\gamma^\Sigma U_\inc, \theta \Tr(\frv)\right]_\Sigma.
\end{split}
\end{equation}

Next, we focus on the left-hand side of formulation \eqref{eq:STF}. By \eqref{PolarityOperatorTr} we write 
\begin{equation}
\label{eq:MTF0}
\begin{split}
& [\mathsf{A}(\fru), \Theta(\frv)]_\Gamma = 
[\mathsf{A}(\fru), \Theta(\frv)]_\Gamma + \bigl( [\fru, \Theta(\frv)]_\Gamma +  [\Tr(\fru), \Tr(\Theta(\frv))]_\Sigma \bigr)/2\\
& = \sum_{j=0}^n [(\mathsf{A}_{\kappa_j}^j + {\Id}/{2})\fru_j, \theta(\frv_j)]_{\Gamma_j} +   [\Tr(\fru), \Tr(\Theta(\frv))]_\Sigma/2 \\
& = [\gamma^0 \mathsf{G}_{\kappa_0}^0 (\fru_0), \theta(\frv_0)]_{\Gamma_0} + \sum_{j=1}^n [(\mathsf{A}_{\kappa_j}^j + {\Id}/{2})\fru_j, \theta(\frv_j)]_{\Gamma_j} +   [\Tr(\fru), \Tr(\Theta(\frv))]_\Sigma/2,
\end{split}
\end{equation}
where we have brought out the term with contributions on $\Gamma_0$ that needs to be rewritten, and applied \eqref{eq:Akj12}. 
Now, since $(\fru, \Tr(\fru)) \in \widetilde{\mbX}(\Gamma)$ (see definition \eqref{eq:defXtilde}), \cite[Lemma 8.1]{ClHi:MTF:2013} yields 
\[
\sum_{j=0}^n \mathsf{G}_{\kappa_0}^j(\fru_j)(\bx) + \mathsf{G}_{\kappa_0}^\Sigma(\Tr(\fru))(\bx)=0  \quad \text{ for } \bx  \in \Omega_0, 
\]
thus, taking interior traces on $\Gamma_0$ and testing against $ \theta(\frv_0)$, we get 
\begin{equation}
\label{eq:MTF1}
[\gamma^0 \mathsf{G}_{\kappa_0}^0 (\fru_0), \theta(\frv_0)]_{\Gamma_0} = 
- \sum_{j=1}^n [\gamma^0 \mathsf{G}_{\kappa_0}^j (\fru_j), \theta(\frv_0)]_{\Gamma_0} 
- [\gamma^0 \mathsf{G}_{\kappa_0}^\Sigma (\Tr(\fru)), \theta(\frv_0)]_{\Gamma_0}. 
\end{equation}
We wish to examine each term on the right-hand side of \eqref{eq:MTF1}. To this purpose, take an arbitrary $j=1,\dots,n$ and follow the procedure described in Remark~\ref{rem:construct} to construct the element $\widetilde{\frw} = (\frw, \frw_\Sigma) = (\frw_0, \dots, \frw_n, \frw_\Sigma) \in \widetilde{\mbX}(\Gamma)$ defined as
\[
\frw_q \coloneqq \gamma^q \mathsf{G}_{\kappa_0}^j(\fru_j) \; \text{ if } q \ne j, \qquad 
\frw_j \coloneqq \gamma_c^j \mathsf{G}_{\kappa_0}^j(\fru_j), \qquad
\frw_\Sigma \coloneqq \Tr(\frw) = \gamma^\Sigma \mathsf{G}_{\kappa_0}^j(\fru_j).
\]
So, by \eqref{PolarityOperatorTr} we have, for $\frv \in \mbX(\Gamma)$, $[\frw, \Theta(\frv)]_\Gamma +  [\frw_\Sigma, \Tr(\Theta(\frv))]_\Sigma = 0$, that is, after splitting, 
\[
\begin{split}
[\gamma^0 \mathsf{G}_{\kappa_0}^j (\fru_j), \theta(\frv_0)]_{\Gamma_0} = 
&- \sum_{q=1, q \ne j}^n [\gamma^q \mathsf{G}_{\kappa_0}^j (\fru_j), \theta(\frv_q)]_{\Gamma_q} \\
& - [\gamma_c^j \mathsf{G}_{\kappa_0}^j (\fru_j), \theta(\frv_j)]_{\Gamma_j} 
- [\gamma^\Sigma \mathsf{G}_{\kappa_0}^j (\fru_j), \Tr(\Theta(\frv))]_\Sigma,
\end{split}
\]
and in a similar way, using again the construction in Remark~\ref{rem:construct}, we obtain 
\[
[\gamma^0 \mathsf{G}_{\kappa_0}^\Sigma (\Tr(\fru)), \theta(\frv_0)]_{\Gamma_0} = 
- \sum_{q=1}^n [\gamma^q \mathsf{G}_{\kappa_0}^\Sigma (\Tr(\fru)), \theta(\frv_q)]_{\Gamma_q}
- [\gamma_c^\Sigma \mathsf{G}_{\kappa_0}^\Sigma (\Tr(\fru)),  \Tr(\Theta(\frv))]_\Sigma.
\]
Then, substituting the last two expressions in \eqref{eq:MTF1} we get 
\[
\begin{split}
[\gamma^0 \mathsf{G}_{\kappa_0}^0 (\fru_0), \theta(\frv_0)]_{\Gamma_0}  = 
\sum_{j=1}^n \biggl (\sum_{q=1, q \ne j}^n [\gamma^q \mathsf{G}_{\kappa_0}^j (\fru_j), \theta(\frv_q)]_{\Gamma_q} 
+ [\gamma^\Sigma \mathsf{G}_{\kappa_0}^j (\fru_j), \Tr(\Theta(\frv))]_\Sigma \biggr ) \\
 + \sum_{q=1}^n [\gamma^q \mathsf{G}_{\kappa_0}^\Sigma (\Tr(\fru)), \theta(\frv_q)]_{\Gamma_q} 
 +  \sum_{j=1}^n [\gamma_c^j \mathsf{G}_{\kappa_0}^j (\fru_j), \theta(\frv_j)]_{\Gamma_j} + [\gamma_c^\Sigma \mathsf{G}_{\kappa_0}^\Sigma (\Tr(\fru)),  \Tr(\Theta(\frv))]_\Sigma. 
\end{split}
\]
Finally, we plug the equation above into the initial rewriting \eqref{eq:MTF0} and, recalling that by \eqref{eq:Akj12bis} $\gamma_c^j \mathsf{G}_{\kappa_0}^j = \mathsf{A}_{\kappa_0}^j - \Id/2$, we obtain
\[
\begin{split}
[\mathsf{A}(\fru), \Theta(\frv)]_\Gamma & = 
\sum_{j=1}^n [(\mathsf{A}_{\kappa_j}^j + {\Id}/{2})\fru_j, \theta(\frv_j)]_{\Gamma_j} +  [\Tr(\fru), \Tr(\Theta(\frv))]_\Sigma/2 \\
& + \sum_{j=1}^n [(\mathsf{A}_{\kappa_0}^j - {\Id}/{2})\fru_j, \theta(\frv_j)]_{\Gamma_j} 
+ [(\mathsf{A}_{\kappa_0}^\Sigma - {\Id}/{2})\Tr(\fru), \Tr(\Theta(\frv))]_\Sigma \\
& + \sum_{j=1}^n \biggl (\sum_{q=1, q \ne j}^n [\gamma^q \mathsf{G}_{\kappa_0}^j (\fru_j), \theta(\frv_q)]_{\Gamma_q} 
+ [\gamma^\Sigma \mathsf{G}_{\kappa_0}^j (\fru_j), \Tr(\Theta(\frv))]_\Sigma \biggr ) \\
& + \sum_{q=1}^n [\gamma^q \mathsf{G}_{\kappa_0}^\Sigma (\Tr(\fru)), \theta(\frv_q)]_{\Gamma_q},
\end{split}
\]
that is, simplifying,
\begin{equation}
\label{eq:rewriteA}
\begin{split}
[\mathsf{A}(\fru), \Theta(\frv)]_\Gamma & = 
\sum_{j=1}^n [(\mathsf{A}_{\kappa_j}^j + \mathsf{A}_{\kappa_0}^j)\fru_j, \theta(\frv_j)]_{\Gamma_j}  
+ [\mathsf{A}_{\kappa_0}^\Sigma \Tr(\fru), \theta\Tr(\frv) ]_\Sigma \\
& + \sum_{j=1}^n \biggl (\sum_{q=1, q \ne j}^n [\gamma^q \mathsf{G}_{\kappa_0}^j (\fru_j), \theta(\frv_q)]_{\Gamma_q} 
+ [\gamma^\Sigma \mathsf{G}_{\kappa_0}^j (\fru_j),\theta\Tr(\frv)]_\Sigma \biggr ) \\
& + \sum_{q=1}^n [\gamma^q \mathsf{G}_{\kappa_0}^\Sigma (\Tr(\fru)), \theta(\frv_q) ]_{\Gamma_q}.\\
\end{split}
\end{equation}

To sum up, if we define the continuous linear operator $\wdoublehat{\mathsf{A}} \colon \wdoublehat{\mbH}(\Gamma) \to \wdoublehat{\mbH}(\Gamma)$ as
\begin{equation}
\label{eq:defAhathat}
\wdoublehat{\mathsf{A}} \coloneqq
\begin{bmatrix}
\mathsf{A}_{\kappa_1}^1 + \mathsf{A}_{\kappa_0}^1 & \gamma^1\mathsf{G}_{\kappa_0}^2 & \dots &  \gamma^1\mathsf{G}_{\kappa_0}^n &  \gamma^1\mathsf{G}_{\kappa_0}^\Sigma \\
\gamma^2\mathsf{G}_{\kappa_0}^1 & \mathsf{A}_{\kappa_2}^2 + \mathsf{A}_{\kappa_0}^2 &   &  \gamma^2\mathsf{G}_{\kappa_0}^n &  \gamma^2\mathsf{G}_{\kappa_0}^\Sigma \\
\vdots & & \ddots & & \vdots \\
\gamma^n \mathsf{G}_{\kappa_0}^1 & \gamma^n\mathsf{G}_{\kappa_0}^2 & & \mathsf{A}_{\kappa_n}^n + \mathsf{A}_{\kappa_0}^n  &  \gamma^n\mathsf{G}_{\kappa_0}^\Sigma \\
\gamma^\Sigma \mathsf{G}_{\kappa_0}^1 & \gamma^\Sigma\mathsf{G}_{\kappa_0}^2 & \dots & \gamma^\Sigma\mathsf{G}_{\kappa_0}^n & \mathsf{A}_{\kappa_0}^\Sigma
\end{bmatrix}
\end{equation}
and for compact notation we set 
\begin{gather*}
\doublehat{\fru} \coloneqq (\fru_1, \dots, \fru_n, (\gamma_\dir^\Sigma U, \Tr_\neu(\fru))), \quad
\doublehat{\frv} \coloneqq (\frv_1, \dots, \frv_n, (\gamma_\dir^\Sigma V, \Tr_\neu(\frv))), \\
\doublehat{\mathfrak{f}} \coloneqq (\gamma^1 U_\inc, \dots, \gamma^n U_\inc, \gamma^\Sigma U_\inc), 
\end{gather*}
using the transformed expressions \eqref{eq:rewriteA} and \eqref{eq:rhsMTF0}, where we additionally replace  $\Tr_\dir(\fru) = \gamma_\dir^\Sigma U$, $\Tr_\dir(\frv) = \gamma_\dir^\Sigma V$, we have found that the single-trace FEM-BEM formulation \eqref{eq:STF} is equivalent to 
\begin{equation*}
\begin{split}
& \text{find } (U,\fru)\in \XSigmaGamma \text{ such that} \\
&  a_\Sigma(U,V) + \lBrace \wdoublehat{\mathsf{A}}(\doublehat{\fru}), \Theta(\doublehat{\frv}) \rBrace
+\frac{1}{2} \left[ \left(\gamma_\dir^\Sigma U, \Tr_\neu(\fru)\right), (\gamma_\dir^\Sigma V, \Tr_\neu(\frv)) \right]_\Sigma \\
&  = F_\Sigma(V)  + \lBrace \doublehat{\mathfrak{f}},\Theta(\doublehat{\frv}) \rBrace \quad \forall \, (V,\frv) \in \XSigmaGamma.
\end{split}
\end{equation*}
This new expression does not have any contributions on $\Gamma_0$, except for $\Tr_\neu(\fru)$, $\Tr_\neu(\frv) \in \mH^{-1/2}(\Sigma)$ that in particular depend on $p_0, q_0$. 
In the spirit of \cite[§9]{ClHi:MTF:2013} and the discussion at the beginning of this section, we now replace the function space $\XSigmaGamma$ by the space with decoupled traces $\mH^1(\Omega_\Sigma) \times \widehat{\mbH}(\Gamma)$, which is a more flexible functional setting.   
In particular, we replace $\Tr_\neu(\fru), \Tr_\neu(\frv)$ by some $p_\Sigma, q_\Sigma \in  \mH^{-1/2}(\Sigma)$. 
Then, we define the \emph{global multi-trace FEM-BEM formulation} 
\begin{empheq}[box=\fbox]{equation}
\label{eq:MTF}
\begin{aligned}
&
\begin{split}
& \text{find } (U,\hat{\fru}) \in \mH^1(\Omega_\Sigma) \times \widehat{\mbH}(\Gamma), \, \hat{\fru} = (\hat{\fru}_1,\dots,\hat{\fru}_n, p_\Sigma), \text{ such that} \\
& \quad a_\Sigma(U,V) + \lBrace \wdoublehat{\mathsf{A}}(\doublehat{\fru}), \Theta(\doublehat{\frv}) \rBrace
+\frac{1}{2} \left[ \left(\gamma_\dir^\Sigma U, p_\Sigma\right), \left(\gamma_\dir^\Sigma V, q_\Sigma\right) \right]_\Sigma \\
& \quad = F_\Sigma(V)  + \lBrace \doublehat{\mathfrak{f}},\Theta(\doublehat{\frv}) \rBrace \quad \forall \, (V,\hat{\frv}) \in \mH^1(\Omega_\Sigma) \times \widehat{\mbH}(\Gamma), \, \hat{\frv} = (\hat{\frv}_1,\dots,\hat{\frv}_n, q_\Sigma)
\end{split}
\\
&
\begin{split}
\text{where } 
& \doublehat{\fru} \coloneqq (\hat{\fru}_1, \dots, \hat{\fru}_n, (\gamma_\dir^\Sigma U, p_\Sigma)), \quad
\doublehat{\frv} \coloneqq (\hat{\frv}_1, \dots, \hat{\frv}_n, (\gamma_\dir^\Sigma V, q_\Sigma)), \\ 
&  \doublehat{\mathfrak{f}} \coloneqq (\gamma^1 U_\inc, \dots, \gamma^n U_\inc, \gamma^\Sigma U_\inc).
\end{split}
\end{aligned}
\end{empheq}
Note that $\wdoublehat{\mathsf{A}}$, defined in \eqref{eq:defAhathat}, is a full-matrix operator with off-diagonal terms $\gamma^q  \mathsf{G}_{\kappa_0}^j$, $\gamma^\Sigma  \mathsf{G}_{\kappa_0}^j$, $\gamma^q  \mathsf{G}_{\kappa_0}^\Sigma$ that couple all subdomains with all other subdomains, hence the attribute \emph{global}. 
The attribute \emph{multi-trace} comes from the fact that the unknown traces are doubled on each interface that separates two (bounded) subdomains. 

The expanded expression for the multi-trace FEM-BEM formulation \eqref{eq:MTF} reads: 
find $(U,\hat{\fru}) \in \mH^1(\Omega_\Sigma) \times \widehat{\mbH}(\Gamma), \, \hat{\fru} = (\hat{\fru}_1,\dots,\hat{\fru}_n, p_\Sigma)$, such that 
\begin{equation}
\label{eq:MTFexp}
\begin{split}
& a_\Sigma(U,V) + 
\sum_{j=1}^n [(\mathsf{A}_{\kappa_j}^j + \mathsf{A}_{\kappa_0}^j)\hat{\fru}_j, \theta(\hat{\frv}_j)]_{\Gamma_j}  
+ \left [\mathsf{A}_{\kappa_0}^\Sigma \left(\gamma_\dir^\Sigma U, p_\Sigma \right), \theta\left(\gamma_\dir^\Sigma V, q_\Sigma\right)\right ]_\Sigma \\
& + \sum_{j=1}^n \biggl(\sum_{q=1, q \ne j}^n [\gamma^q \mathsf{G}_{\kappa_0}^j (\hat{\fru}_j), \theta(\hat{\frv}_q)]_{\Gamma_q} 
+ \left[\gamma^\Sigma \mathsf{G}_{\kappa_0}^j (\hat{\fru}_j),\theta\left(\gamma_\dir^\Sigma V,q_\Sigma\right) \right]_\Sigma \biggr) \\
& + \sum_{q=1}^n \left [\gamma^q \mathsf{G}_{\kappa_0}^\Sigma \left(\gamma_\dir^\Sigma U, p_\Sigma\right), \theta(\hat{\frv}_q) \right ]_{\Gamma_q}
+\frac{1}{2} \left[ \left(\gamma_\dir^\Sigma U, p_\Sigma \right), \left(\gamma_\dir^\Sigma V, q_\Sigma\right) \right]_\Sigma \\
&  = F_\Sigma(V)  +
\sum_{j=1}^n [\gamma^j U_\inc, \theta(\hat{\frv}_j)]_{\Gamma_j} + \left[\gamma^\Sigma U_\inc, \theta\left(\gamma_\dir^\Sigma V, q_\Sigma\right) \right]_\Sigma
\end{split}
\end{equation}
for all $(V,\hat{\frv}) \in \mH^1(\Omega_\Sigma) \times \widehat{\mbH}(\Gamma), \,\hat{\frv} = (\hat{\frv}_1,\dots,\hat{\frv}_n, q_\Sigma)$.

\begin{remark} 
Note that in the case $n=0$ of a single (heterogeneous) scatterer the multi-trace formulation \eqref{eq:MTF} reduces to the Costabel coupling \eqref{VarIdCostabel2}, just as the single-trace formulation \eqref{eq:STF}. 
Indeed, in this case $\RR^d = \overline\Omega_0 \cup \overline\Omega_\Sigma$ and 
$\widehat{\mbH}(\Gamma) = \mH^{-1/2}(\Sigma)$.
Then the expanded expression \eqref{eq:MTFexp} simply becomes: 
find $(U, p_\Sigma) \in \mH^{1}(\Omega_{\Sigma})\times \mH^{-1/2}(\Sigma)$ such that
\begin{multline*}
a_\Sigma(U,V) +  \left [\mathsf{A}_{\kappa_0}^\Sigma \left(\gamma_\dir^\Sigma U, p_\Sigma\right),\left(-\gamma_\dir^\Sigma V, q_\Sigma\right) \right ]_\Sigma 
+\frac{1}{2} \left[ \left(\gamma_\dir^\Sigma U, p_\Sigma\right), \left(\gamma_\dir^\Sigma V, q_\Sigma\right) \right]_\Sigma =\\
F_\Sigma(V)  + \left[\gamma^\Sigma U_\inc, \left(-\gamma_\dir^\Sigma V, q_\Sigma\right) \right]_\Sigma 
 \quad \forall \, (V, q_\Sigma) \in  \mH^{1}(\Omega_{\Sigma})\times \mH^{-1/2}(\Sigma). 
\end{multline*}
Moreover, set $\fru = (u_0, p_0) = (\gamma_\dir^\Sigma U, -p_\Sigma)$ and $\frv = (v_0, q_0) = (\gamma_\dir^\Sigma V, -q_\Sigma)$, so that $(U,\fru)\in \XSigmaGamma$ and $(V,\frv)\in \XSigmaGamma$. Note that $\Tr(\fru) = (\gamma_\dir^\Sigma U, p_\Sigma)$, $\Tr(\frv) = (\gamma_\dir^\Sigma V, q_\Sigma)$, and $\gamma^\Sigma U_\inc = (\gamma_\dir^0 U_\inc, - \gamma_\neu^0 U_\inc)$. We can also write 
\[
\mathsf{A}_{\kappa_0}^\Sigma \left(\gamma_\dir^\Sigma U, p_\Sigma\right) = 
\mathsf{A}_{\kappa_0}^\Sigma \left(u_0, -p_0\right) = 
\left( \{\gamma_\dir^0\} , -\{\gamma_\neu^0 \}\right) \circ \left(-\mathsf{G}_{\kappa_0}^0 \left(u_0, p_0\right) \right) =
\left( -\{\gamma_\dir^0\}, \{\gamma_\neu^0 \}\right)  \circ \mathsf{G}_{\kappa_0}^0 (\fru), 
\]
so we get: find $(U,\fru)\in \XSigmaGamma$ such that
\begin{multline*}
a_\Sigma(U,V) +  \left [\left( -\{\gamma_\dir^0\} , \{\gamma_\neu^0 \}\right)  \circ \mathsf{G}_{\kappa_0}^0 (\fru),\left(-v_0, -q_0\right) \right ]_{\Gamma} 
+\frac{1}{2} \lbr \Tr(\fru),\Tr(\frv)\rbr_{\Sigma} =\\
F_\Sigma(V)  + \left[\left(\gamma_\dir^0 U_\inc, - \gamma_\neu^0 U_\inc\right), \left(-v_0, -q_0\right) \right]_{\Gamma} 
 \quad \forall \, (V,\frv)\in \XSigmaGamma, 
\end{multline*}
and also the signs turn out to agree with those in formulation \eqref{VarIdCostabel2}. 
\end{remark}

\subsection{Properties of the multi-trace FEM-BEM formulation}

The relationship between the multi-trace FEM-BEM formulation \eqref{eq:MTF} and the transmission problem \eqref{eq:bvp} is examined in the following proposition. 
\begin{proposition}[Link with the transmission problem]
\label{pr:linkMTF}
If $(U,\hat{\fru}) \in \mH^1(\Omega_\Sigma) \times \widehat{\mbH}(\Gamma)$, $\hat{\fru} = (\hat{\fru}_1,\dots,\hat{\fru}_n, p_\Sigma) \in \widehat{\mbH}(\Gamma)$, solve \eqref{eq:MTF}, then the solution to \eqref{eq:bvp} is given by
\begin{equation}
\label{eq:UtildeMTF}
\begin{split}
\widetilde{U} (\bx) & \coloneqq U(\bx) \quad \text{ for } \bx \in \Omega_\Sigma, \\
\widetilde{U} (\bx) & \coloneqq \biggl(U_\inc - \mathsf{G}_{\kappa_0}^\Sigma\left(\gamma_\dir^\Sigma U, p_\Sigma\right)  - \sum_{j=1}^n \mathsf{G}_{\kappa_0}^j(\hat{\fru}_j) \biggr) (\bx) \quad \text{ for } \bx \in \Omega_0, \\
\widetilde{U} (\bx) & \coloneqq \mathsf{G}_{\kappa_j}^j(\hat{\fru}_j)(\bx) \quad \text{ for } \bx \in \Omega_j, j=1,\dots,n.
\end{split}
\end{equation}
\end{proposition}
\begin{proof}
First of all, $(\widetilde{U}-U_\inc)|_{\Omega_0} = - \mathsf{G}_{\kappa_0}^\Sigma({\gamma_\dir^\Sigma U},{p_\Sigma})  - \sum_{j=1}^n \mathsf{G}_{\kappa_0}^j(\hat{\fru}_j)$ is $\kappa_0$-outgoing radiating in $\Omega_0$,  see e.g.~\cite[Theorem 3.2]{CoKr:bookIEM:1983}.
It is also clear that $\widetilde{U}$ satisfies the Helmholtz equation in $\Omega_j$, $j=1,\dots,n$ since it is satisfied by the potentials (see  e.g.~\cite[§2.4]{CoKr:bookIEM:1983}), and in $\Omega_0$ since it is also satisfied by $U_\inc$ by definition. 
By testing \eqref{eq:MTF} with $V \in \mH^1_0(\Omega_\Sigma)$, $\hat{\frv}=0$, we get $a_\Sigma(U,V) = a_\Sigma(\widetilde{U},V) = F_\Sigma(V)$, so $\widetilde{U}$ satisfies the Helmholtz equation also in $\Omega_\Sigma$. 
The property that remains to be verified is the transmission conditions:  
by characterization~\eqref{eq:ReformulationTransmissionConditionsbis} it is sufficient to show that for all $\frv \in \mbX(\Gamma)$ we have $\lbr \gamma(\widetilde{U}), \frv\rbr_{\Gamma} + \lbr \gamma^\Sigma (\widetilde{U}), \Tr(\frv) \rbr_\Sigma = 0$, 
i.e., by definition of $\widetilde{U}$, 
\begin{equation}
\label{eq:aim}
\biggl[ \gamma^0 U_\inc - \gamma^0 \mathsf{G}_{\kappa_0}^\Sigma\left(\gamma_\dir^\Sigma U, p_\Sigma \right)  - \sum_{j=1}^n \gamma^0 \mathsf{G}_{\kappa_0}^j(\hat{\fru}_j), \frv_0 \biggr]_{\Gamma_0} 
+ \sum_{j=1}^n [\gamma^j  \mathsf{G}_{\kappa_j}^j(\hat{\fru}_j), \frv_j]_{\Gamma_j} + [\gamma^\Sigma U, \Tr(\frv)]_\Sigma = 0.
\end{equation}

We fix an arbitrary $\frv \in \mbX(\Gamma)$ and denote $\frv_* \coloneqq (\frv_1, \dots, \frv_n, \Tr_\neu(\frv)) \in \widehat{\mbH}(\Gamma)$. 
Since $\widetilde{U}|_{\Omega_\Sigma} = U$ satisfies the Helmholtz equation in $\Omega_\Sigma$, integrating by parts we get
\begin{equation}
\label{eq:FEMaux}
a_\Sigma({U},V) - F_\Sigma(V) = \braket{\gamma_\neu^\Sigma {U},\gamma_\dir^\Sigma V}_\Sigma \quad \forall \, V \in \mH^1(\Omega_\Sigma).
\end{equation}
Moreover, by \eqref{eq:Akj12}-\eqref{eq:Akj12bis}, we have
\begin{equation}
\label{eq:BEMaux}
\mathsf{A}_{\kappa_j}^j + \mathsf{A}_{\kappa_0}^j  = \gamma^j  \mathsf{G}_{\kappa_j}^j + \gamma_c^j  \mathsf{G}_{\kappa_0}^j, \; j=1,\dots,n, \qquad
\mathsf{A}_{\kappa_0}^\Sigma =  \gamma_c^\Sigma  \mathsf{G}_{\kappa_0}^\Sigma + \Id/2.
\end{equation}
Thus, if we test formulation~\eqref{eq:MTF} (or its expanded form \eqref{eq:MTFexp}) with $V$ satisfying $\gamma_\dir^\Sigma V = \Tr_\dir(\frv)$ and with $\hat{\frv}= \frv_*$, using \eqref{eq:FEMaux}-\eqref{eq:BEMaux}, we obtain 
\begin{equation}
\label{eq:tested}
\begin{split}
0 &= \braket{\gamma_\neu^\Sigma {U}, \Tr_\dir(\frv)}_\Sigma 
+ \sum_{j=1}^n \bigl ([ \gamma^j  \mathsf{G}_{\kappa_j}^j(\hat{\fru}_j), \theta(\frv_j)]_{\Gamma_j}  + [ \gamma_c^j  \mathsf{G}_{\kappa_0}^j(\hat{\fru}_j), \theta(\frv_j)]_{\Gamma_j} \bigr ) \\
& + \left [\gamma_c^\Sigma  \mathsf{G}_{\kappa_0}^\Sigma \left(\gamma_\dir^\Sigma U, p_\Sigma\right), \Tr(\Theta(\frv)) \right ]_\Sigma 
+ \frac{1}{2} \left [ \left(\gamma_\dir^\Sigma U, p_\Sigma\right), \Tr(\Theta(\frv)) \right ]_\Sigma\\
& + \sum_{j=1}^n \biggl (\sum_{q=1, q \ne j}^n [\gamma^q \mathsf{G}_{\kappa_0}^j (\hat{\fru}_j), \theta(\frv_q)]_{\Gamma_q} 
+  [\gamma^\Sigma \mathsf{G}_{\kappa_0}^j (\hat{\fru}_j), \Tr(\Theta(\frv)) ]_\Sigma \biggr ) \\
& + \sum_{q=1}^n \left [\gamma^q \mathsf{G}_{\kappa_0}^\Sigma \left(\gamma_\dir^\Sigma U, p_\Sigma\right), \theta(\frv_q) \right ]_{\Gamma_q} 
+\frac{1}{2} \left[ \left(\gamma_\dir^\Sigma U, p_\Sigma\right), \Tr(\frv) \right]_\Sigma  \\
&- \sum_{j=1}^n [\gamma^j U_\inc, \theta(\frv_j)]_{\Gamma_j} - \left[\gamma^\Sigma U_\inc, \Tr(\Theta(\frv)) \right]_\Sigma,
\end{split}
\end{equation}
that is, gathering terms conveniently, $0 = t_1 + t_2 + t_3 + t_4 + t_5$, where
\[
\begin{split}
& t_1 = \sum_{j=1}^n [ \gamma^j  \mathsf{G}_{\kappa_j}^j(\hat{\fru}_j), \theta(\frv_j)]_{\Gamma_j}\\
& t_2 = \braket{\gamma_\neu^\Sigma {U},\Tr_\dir(\frv)}_\Sigma
+ \frac{1}{2} \left [ \left(\gamma_\dir^\Sigma U, p_\Sigma\right), \Tr(\Theta(\frv)) \right ]_\Sigma 
+\frac{1}{2} \left[ \left(\gamma_\dir^\Sigma U, p_\Sigma\right), \Tr(\frv) \right]_\Sigma\\
& t_3 = \sum_{j=1}^n \biggl ( [ \gamma_c^j  \mathsf{G}_{\kappa_0}^j(\hat{\fru}_j), \theta(\frv_j)]_{\Gamma_j} 
+ \! \! \! \sum_{q=1, q \ne j}^n [\gamma^q \mathsf{G}_{\kappa_0}^j (\hat{\fru}_j), \theta(\frv_q)]_{\Gamma_q} 
+  [\gamma^\Sigma \mathsf{G}_{\kappa_0}^j (\hat{\fru}_j), \Tr(\Theta(\frv)) ]_\Sigma \biggr )\\
&  t_4 = \left [\gamma_c^\Sigma  \mathsf{G}_{\kappa_0}^\Sigma \left(\gamma_\dir^\Sigma U, p_\Sigma \right), \Tr(\Theta(\frv)) \right ]_\Sigma
+ \sum_{q=1}^n \left [\gamma^q \mathsf{G}_{\kappa_0}^\Sigma \left(\gamma_\dir^\Sigma U, p_\Sigma\right), \theta(\frv_q) \right ]_{\Gamma_q}\\
& t_5 = - \sum_{j=1}^n [\gamma^j U_\inc, \theta(\frv_j)]_{\Gamma_j} - \left[\gamma^\Sigma U_\inc, \Tr(\Theta(\frv)) \right]_\Sigma.
\end{split}
\]
First of all, note that the term $t_2$ simplifies into $[\gamma^\Sigma U, \Tr(\Theta(\frv))]_\Sigma$, which is exactly the last term of the sought equation \eqref{eq:aim}, but with $\frv$ replaced by $\Theta(\frv)$. 
In order to treat the other terms we will employ the polarity identity \eqref{PolarityOperatorTr} and the procedure described in Remark~\ref{rem:construct} three times. 
First, for a given $j=1,\dots,n$, we have $(\gamma^0 \mathsf{G}_{\kappa_0}^j(\hat{\fru}_j), \dots, \gamma_c^j \mathsf{G}_{\kappa_0}^j(\hat{\fru}_j), \dots, \gamma^n \mathsf{G}_{\kappa_0}^j(\hat{\fru}_j),\gamma^\Sigma \mathsf{G}_{\kappa_0}^j (\hat{\fru}_j)) \in \widetilde{\mbX}(\Gamma)$, thus 
\begin{equation}
\label{eq:first}
\begin{split}
&[\gamma_c^j \mathsf{G}_{\kappa_0}^j(\hat{\fru}_j), \theta(\frv_j)]_{\Gamma_j} + \! \sum_{q=1, q \ne j}^n [\gamma^q \mathsf{G}_{\kappa_0}^j (\hat{\fru}_j), \theta(\frv_q)]_{\Gamma_q}  + [\gamma^\Sigma \mathsf{G}_{\kappa_0}^j (\hat{\fru}_j), \Tr(\Theta(\frv)) ]_\Sigma \\
& \quad = - [\gamma^0 \mathsf{G}_{\kappa_0}^j(\hat{\fru}_j), \theta(\frv_0)]_{\Gamma_0}.
\end{split}
\end{equation}
Second, we have $(\gamma^0 \mathsf{G}_{\kappa_0}^\Sigma ({\gamma_\dir^\Sigma U},{p_\Sigma}), \dots, \gamma^n \mathsf{G}_{\kappa_0}^\Sigma ({\gamma_\dir^\Sigma U},{p_\Sigma}), \gamma_c^\Sigma \mathsf{G}_{\kappa_0}^\Sigma ({\gamma_\dir^\Sigma U},{p_\Sigma})) \in \widetilde{\mbX}(\Gamma)$, thus 
\begin{equation}
\label{eq:second}
\begin{split}
& \sum_{q=1}^n \left [\gamma^q \mathsf{G}_{\kappa_0}^\Sigma \left(\gamma_\dir^\Sigma U, p_\Sigma\right), \theta(\frv_q) \right ]_{\Gamma_q} + 
\left [\gamma_c^\Sigma  \mathsf{G}_{\kappa_0}^\Sigma \left(\gamma_\dir^\Sigma U, p_\Sigma \right), \Tr(\Theta(\frv)) \right ]_\Sigma \\
& \quad = - \left[\gamma^0  \mathsf{G}_{\kappa_0}^\Sigma \left(\gamma_\dir^\Sigma U, p_\Sigma \right),  \theta(\frv_0) \right]_{\Gamma_0}.
\end{split}
\end{equation}
Third, we have $(\gamma^0 U_\inc, \dots, \gamma^n U_\inc, \gamma^\Sigma U_\inc) \in \widetilde{\mbX}(\Gamma)$, thus
\begin{equation}
\label{eq:third}
\sum_{j=1}^n [\gamma^j U_\inc, \theta(\frv_j)]_{\Gamma_j} + \left[\gamma^\Sigma U_\inc, \Tr(\Theta(\frv)) \right]_\Sigma = - [\gamma^0 U_\inc, \theta(\frv_0)]_{\Gamma_0}.
\end{equation}
Now, use \eqref{eq:first} summed over $j=1,\dots,n$, \eqref{eq:second}, \eqref{eq:third} to replace respectively $t_3, t_4, t_5$, therefore \eqref{eq:tested} becomes
\[
\begin{split}
0 = &\sum_{j=1}^n [ \gamma^j  \mathsf{G}_{\kappa_j}^j(\hat{\fru}_j), \theta(\frv_j)]_{\Gamma_j} 
- \sum_{j=1}^n [\gamma^0 \mathsf{G}_{\kappa_0}^j(\hat{\fru}_j), \theta(\frv_0)]_{\Gamma_0}
 - \left[\gamma^0  \mathsf{G}_{\kappa_0}^\Sigma \left(\gamma_\dir^\Sigma U, p_\Sigma \right),  \theta(\frv_0) \right]_{\Gamma_0} \\
&+ [\gamma^0 U_\inc, \theta(\frv_0)]_{\Gamma_0} +
[\gamma^\Sigma U, \Tr(\Theta(\frv))]_\Sigma,
\end{split}
\]
that is exactly the sought equation \eqref{eq:aim}, but with $\frv$ replaced by $\Theta(\frv)$, which is not a problem since $\Theta$ is an automorphism. 
\end{proof}

As suggested by the gap idea, also the multi-trace FEM-BEM formulation \eqref{eq:MTF} satisfies a \emph{G\r{a}rding inequality}:
\begin{proposition}[G\r{a}rding inequality] 
\label{pr:gardingMTF}
Let $\mathsf{a}_\textup{MTF} : (\mH^1(\Omega_\Sigma) \times \widehat{\mbH}(\Gamma)) \times (\mH^1(\Omega_\Sigma) \times \widehat{\mbH}(\Gamma)) \to \CC$ designate the bilinear form on the left-hand side of \eqref{eq:MTF}.
There exist a compact bilinear form $\mathcal{K} \colon (\mH^1(\Omega_\Sigma) \times \widehat{\mbH}(\Gamma)) \times (\mH^1(\Omega_\Sigma) \times \widehat{\mbH}(\Gamma)) \to \CC$ and a constant $\beta>0$ such that 
\begin{equation*}
\Re \bigl \{ \mathsf{a}_\textup{MTF} \bigl( (V,\hat{\frv}),(\overline{V},\overline{\hat{\frv}}) \bigr)  + 
\mathcal{K}\bigl( (V,\hat{\frv}),(\overline{V},\overline{\hat{\frv}}) \bigr)  \bigr \}
\ge \beta (\lVert V \rVert_{\mH^1(\Omega_\Sigma)}^2 + \lVert \hat{\frv} \rVert_{\widehat{\mbH}(\Gamma)}^2)
\end{equation*}
for all $(V,\hat{\frv}) \in \mH^1(\Omega_\Sigma) \times \widehat{\mbH}(\Gamma)$. 
\end{proposition}
\begin{proof}
We need to examine
\[
\mathsf{a}_\textup{MTF} \bigl( (V,\hat{\frv}),(\overline{V},\overline{\hat{\frv}})  \bigr)  =
a_\Sigma(V,\overline{V}) + \lBrace \wdoublehat{\mathsf{A}}(\doublehat{\frv}), \Theta(\overline{\doublehat{\frv}}) \rBrace
+\frac{1}{2} \left[ \left(\gamma_\dir^\Sigma V, q_\Sigma\right), \left(\gamma_\dir^\Sigma \overline{V}, \overline{q}_\Sigma\right) \right]_\Sigma, 
\]
where $\doublehat{\frv} \coloneqq (\hat{\frv}_1, \dots, \hat{\frv}_n, (\gamma_\dir^\Sigma V, q_\Sigma))$.
As already mentioned, $a_\Sigma$ satisfies a G\r{a}rding inequality as in \cite[Lemma 3.2]{Meury:thesis:2007}, and
\[
\Re \left \{ \left[ \left(\gamma_\dir^\Sigma V, q_\Sigma\right), \left(\gamma_\dir^\Sigma \overline{V}, \overline{q}_\Sigma\right) \right]_\Sigma \right \} = 0.
\]
For the remaining term $\lBrace \wdoublehat{\mathsf{A}}(\doublehat{\frv}), \Theta(\overline{\doublehat{\frv}}) \rBrace$, we proceed exactly as in the proof of \cite[Proposition 6.3]{ClHi:impenetrable:2015}, except that $\hat{\frv}_{n+1} \coloneqq ({\gamma_\dir^\Sigma V},{q_\Sigma})$ in the present case. Indeed, note that the first concise equality at the beginning of the proof of \cite[Proposition 6.3]{ClHi:impenetrable:2015} fits exactly the expression of $\wdoublehat{\mathsf{A}}$. Thus, we obtain that, for the case $\kappa_0 = \dots = \kappa_n = \imagi$, there exists $\tilde{\beta} > 0$ such that
\[
\Re \lBrace \wdoublehat{\mathsf{A}}(\doublehat{\frv}), \Theta(\overline{\doublehat{\frv}}) \rBrace \ge
\tilde{\beta} \sum_{j=1}^{n+1} \lVert \hat{\frv}_j \rVert_{\mbH(\Gamma_j)}^2 =
\tilde{\beta} \Bigl (\lVert \hat{\frv} \rVert_{\widehat{\mbH}(\Gamma)}^2) + \lVert \gamma_\dir^\Sigma V \rVert_{\mH^{1/2}(\Sigma)}^2 \Bigr),
\]
which leads to the desired conclusion since a change of the wavenumbers $\kappa_0, \dots, \kappa_n$ only induces a compact perturbation of the integral operators appearing in  $\wdoublehat{\mathsf{A}}$ (see e.g. \cite[Lemma 3.9.8]{SaSw:book:2011}).
\end{proof}

Again, in the case of injectivity all the nice consequences recalled below Proposition~\ref{pr:gardingSTF} would follow from the G\r{a}rding inequality. Hence, in the following proposition we examine the injectivity condition for the multi-trace FEM-BEM formulation \eqref{eq:MTF}. Note that the gap configuration falls exactly within the case $\Sigma \subset \Gamma_0$, in which spurious resonances affect the single-trace FEM-BEM formulation \eqref{eq:STF} if $\kappa_0 \in \mathfrak{S}(\Delta,\Omega_\Sigma)$ (recall Proposition~\ref{pr:injSTF}), so the following result is not surprising. 
\begin{proposition}[Injectivity condition]
\label{pr:injMTF}
Let $(U,\hat{\fru}) \in \mH^1(\Omega_\Sigma) \times \widehat{\mbH}(\Gamma)$ solves formulation \eqref{eq:MTF} with $F_\Sigma \equiv 0$, $\doublehat{\mathfrak{f}}=0$. 
Then $U=0$. 
We also have $\hat{\fru}=0$ if $\kappa_0 \notin \mathfrak{S}(\Delta,\Omega_\Sigma)$. 
If $\kappa_0 \in \mathfrak{S}(\Delta,\Omega_\Sigma)$, there exists $\hat{\fru}\in\widehat{\mbH}(\Gamma)\setminus\{0\}$ such that $(0,\hat{\fru})\in \mH^1(\Omega_\Sigma) \times \widehat{\mbH}(\Gamma)$ solves \eqref{eq:MTF} with $F_\Sigma \equiv 0$, $\doublehat{\mathfrak{f}}=0$.
\end{proposition}
\begin{proof}
By Proposition~\ref{pr:linkMTF}, the function $\widetilde U$ defined in \eqref{eq:UtildeMTF} solves the homogeneous transmission problem \eqref{eq:bvp}, which is well-posed, so $\widetilde U = 0$. In particular, $U=\widetilde{U}|_{\Omega_\Sigma} =0$ and $\gamma_\dir^\Sigma U = 0$. 
Therefore, if we test formulation \eqref{eq:MTF} with $F_\Sigma \equiv 0$, $\doublehat{\mathfrak{f}}=0$ using test functions $V \in \mH_0^1(\Omega_\Sigma)$ (and $\hat{\frv} = (\hat{\frv}_1,\dots,\hat{\frv}_n, q_\Sigma)\in \widehat{\mbH}(\Gamma)$), we obtain 
\[
\lBrace \wdoublehat{\mathsf{A}}(\doublehat{\fru}), \Theta(\doublehat{\frv}) \rBrace = 0, \text{ with }
\doublehat{\fru} = (\hat{\fru}_1, \dots, \hat{\fru}_n, (0, p_\Sigma)), \quad
\doublehat{\frv} = (\hat{\frv}_1, \dots, \hat{\frv}_n, (0, q_\Sigma)), 
\]
which reduces to $\llbracket \widehat{\mathsf{A}}(\hat{\fru}), \Theta(\hat{\frv})\rrbracket = 0$, where $\widehat{\mathsf{A}}$ is the operator defined in \cite[Equation~(6.3)]{ClHi:impenetrable:2015}. Then $\hat{\fru} \in \ker(\widehat{\mathsf{A}})$ and, by  \cite[Proposition 6.4]{ClHi:impenetrable:2015}, if $\kappa_0 \notin \mathfrak{S}(\Delta,\Omega_\Sigma)$ we get $\hat{\fru} =  0$.

Now we show that $\kappa_0 \notin \mathfrak{S}(\Delta,\Omega_\Sigma)$ is also a necessary condition. If $\kappa_0 \in \mathfrak{S}(\Delta,\Omega_\Sigma)$, by  \cite[Theorem 3.9.1]{SaSw:book:2011} we know that $\ker(\gamma_\dir^\Sigma \mathsf{SL}_{\kappa_0}^\Sigma) \ne \{0\}$, and we consider $p \in \ker(\gamma_\dir^\Sigma \mathsf{SL}_{\kappa_0}^\Sigma) \backslash \{0\}$. As $\gamma_\dir^\Sigma \mathsf{SL}_{\kappa_0}^\Sigma(p) = 0$, by jump relations \eqref{eq:jumprels} we have $\gamma_{\dir,c}^\Sigma \mathsf{SL}_{\kappa_0}^\Sigma(p) = 0$, and, since the exterior Helmholtz boundary value problem is well-posed, we get $\mathsf{SL}_{\kappa_0}^\Sigma(p)(\bx) = 0$ for $\bx \in \RR^d \backslash \Omega_\Sigma$. 
Therefore, $\gamma^q \mathsf{SL}_{\kappa_0}^\Sigma(p) = 0$ for all $q = 1,\dots, n$ and $\gamma_{\neu,c}^\Sigma \mathsf{SL}_{\kappa_0}^\Sigma(p) = 0$.
In particular, using \eqref{eq:Akj12bis},
\[
\mathsf{A}_{\kappa_0}^\Sigma \left(0, p\right) = \gamma_c^\Sigma \mathsf{SL}_{\kappa_0}^\Sigma(p) +  \left(0, p\right) /2  = \left(0, p/2\right).
\] 
Then, if we evaluate the left-hand side of formulation \eqref{eq:MTF} in $U^*=0$, $\hat{\fru}^* = (0, \dots, 0, p)$ we have 
\[
\sum_{q=1}^n  [\gamma^q \mathsf{SL}_{\kappa_0}^\Sigma (p), \theta(\hat{\frv}_q)  ]_{\Gamma_q} +
\left [ \left(0, p/2\right), \left(-\gamma_\dir^\Sigma V, q_\Sigma \right) \right ]_\Sigma + \frac{1}{2} \left [ \left(0, p\right), \left(\gamma_\dir^\Sigma V, q_\Sigma\right) \right ]_\Sigma = 0,
\]
for all $(V,\hat{\frv}) \in \mH^1(\Omega_\Sigma) \times \widehat{\mbH}(\Gamma)$, and we have found a non-trivial solution.
\end{proof}
Comparing the injectivity conditions in Propositon~\ref{pr:injSTF} and Proposition~\ref{pr:injMTF}, we see that in the case $\Sigma \subset \Gamma_0$, if the single-trace formulation \eqref{eq:STF} suffer from spurious resonances then so does the multi-trace formulation \eqref{eq:MTF}. On the other hand, in the case $\Sigma \not \subset \Gamma_0$, there are wavenumbers $\kappa_0$ for which the multi-trace formulation \eqref{eq:MTF} breaks down, while the single-trace formulation \eqref{eq:STF} remains injective. 
If the single-trace formulation \eqref{eq:STF} fails to be injective because $\Sigma \subset \Gamma_1$ and $\kappa_1 \in \mathfrak{S}(\Delta,\Omega_\Sigma)$, but $\kappa_0 \notin \mathfrak{S}(\Delta,\Omega_\Sigma)$, the multi-trace formulation \eqref{eq:MTF} is instead well-posed. 
Note that we could write a multi-trace formulation based on another subdomain than $\Omega_0$, say $\Omega_i$, loosely speaking by filling the gap with the same medium as $\Omega_i$.

\section{Multi-trace combined field FEM-BEM formulation}
\label{sec:MTFcomb}

\noindent 
We have shown that the multi-trace FEM-BEM formulation \eqref{eq:MTF} is affected by spurious resonances when $\kappa_0 \in \mathfrak{S}(\Delta,\Omega_\Sigma)$. 
Again, as a remedy, we adapt the approach of combined field integral equations. More precisely, as the standard multi-trace FEM-BEM formulation \eqref{eq:MTF} was obtained by manipulating the standard single-trace FEM-BEM formulation \eqref{eq:STF}, similarly we will obtain a combined field multi-trace FEM-BEM formulation by manipulating the combined field single-trace FEM-BEM formulation \eqref{eq:CSTF}. 
Since the difference between \eqref{eq:STF} and \eqref{eq:CSTF} lies only in the compact bilinear form $\mathsf{c}$ and in the right-hand side with $\widetilde{\fru}^\inc$ defined in \eqref{eq:cBilForm}, we just need to elaborate these terms. 

As in \cite[§6.4]{ClHi:impenetrable:2015} we first derive the formulation \emph{in the gap setting}, and we look for 
\[
\wdoublehat{\mathsf{c}} \colon (\mH^1(\Omega_\Sigma) \times \widehat{\mbH}(\Gamma))^2 \to \CC 
\quad \text{such that}  \quad \wdoublehat{\mathsf{c}} ((U,\hat{\fru}),(V,\hat{\frv})) = \mathsf{c}(\fru, \frv)
\]
where $(U,\hat{\fru})$, $(V,\hat{\frv}) \in \mH^1(\Omega_\Sigma) \times \widehat{\mbH}(\Gamma)$ correspond respectively to 
$(U,\fru)$, $(V,\frv) \in \XSigmaGamma$ under the isomorphism defined at the beginning of §\ref{sec:MTF}. 
Observe that in the gap setting, where $\Sigma \subset \partial\Omega_0$, the extension operator $\mathsf{E}_\Sigma$ can be picked to map into functions whose support is inside $\Omega_0 \cup \overline{\Omega}_\Sigma$, so that $\gamma_\dir^j \circ \mathsf{E}_\Sigma = 0$ for $j \ne 0$ and the operator $\mathsf{C}$ in \eqref{eq:R-C}, essentially, maps into $\mH^{1/2}(\Sigma)$.  
Then, applying also \eqref{eq:Akj12bis}, 
\begin{equation*}
\mathsf{c}(\fru, \frv) = \sum_{j=0}^n \left[\gamma_c^j \mathsf{G}_{\kappa_j}^j (\fru_j), \theta(\mathsf{C}\frv)_j\right]_{\Gamma_j} 
= \left[\gamma_c^0 \mathsf{G}_{\kappa_0}^0 (\fru_0), \theta(\mathsf{C}\frv)_0\right]_{\Gamma_0} 
= - \Braket{\gamma_\neu^\Sigma \mathsf{G}_{\kappa_0}^0(\fru_0), \mathsf{M}\Tr_\neu(\frv)}_\Sigma,
\end{equation*}
where we have used the definition of $\mathsf{C}$, $\gamma_{\neu,c}^0 = - \gamma_\neu^\Sigma$, $\gamma_{\dir}^0 = \gamma_\dir^\Sigma$, $\gamma_\dir^\Sigma \circ \mathsf{E}_\Sigma = \Id$. 
Moreover, since $(U,\fru) \in \XSigmaGamma$, in the gap setting $\fru_0$ equals $\phi(\fru_j)$ on each $\Gamma_j$, $j=1,\dots,n$, and equals $\phi(\gamma_\dir^\Sigma U, \Tr_\neu(\fru))$ on $\Sigma$, that reflects exactly the isomorphism defined at the beginning of §\ref{sec:MTF}. 
This implies 
$\mathsf{G}_{\kappa_0}^0(\fru_0) = - \mathsf{G}_{\kappa_0}^\Sigma(\gamma_\dir^\Sigma U, \Tr_\neu(\fru)) - \sum_{j=1}^n \mathsf{G}_{\kappa_0}^j (\fru_j)$. 
Therefore, for $\hat{\fru} = (\hat{\fru}_1,\dots,\hat{\fru}_n, p_\Sigma)  \in \widehat{\mbH}(\Gamma)$, $\hat{\frv} = (\hat{\frv}_1,\dots,\hat{\frv}_n, q_\Sigma)  \in \widehat{\mbH}(\Gamma)$, we get
\[
\wdoublehat{\mathsf{c}} ((U,\hat{\fru}),(V,\hat{\frv})) \coloneqq  
 \Braket{\mathsf{M}^*\gamma_\neu^\Sigma \mathsf{G}_{\kappa_0}^\Sigma \left(\gamma_\dir^\Sigma U, p_\Sigma \right), q_\Sigma}_\Sigma 
+ \sum_{j=1}^n \Braket{\mathsf{M}^*\gamma_\neu^\Sigma \mathsf{G}_{\kappa_0}^j (\hat{\fru}_j), q_\Sigma}_\Sigma.
\]
Now, summing the term in \eqref{eq:MTF} that derives from $[\mathsf{A}(\fru), \Theta(\frv)]_\Gamma$ we write
\[
\lBrace \wdoublehat{\mathsf{A}}(\doublehat{\fru}), \Theta(\doublehat{\frv}) \rBrace + \wdoublehat{\mathsf{c}} ((U,\hat{\fru}),(V,\hat{\frv})) 
= \lBrace \wdoublehat{\mathsf{A}}_\mathsf{M}(\doublehat{\fru}), \Theta(\doublehat{\frv}) \rBrace
\]
where
\[
\doublehat{\fru} \coloneqq (\hat{\fru}_1, \dots, \hat{\fru}_n, (\gamma_\dir^\Sigma U, p_\Sigma)), \quad
\doublehat{\frv} \coloneqq (\hat{\frv}_1, \dots, \hat{\frv}_n, (\gamma_\dir^\Sigma V, q_\Sigma)), \qquad \text{and}
\]
\begin{equation*}
\wdoublehat{\mathsf{A}}_\mathsf{M} \coloneqq
\begin{bmatrix}
\mathsf{A}_{\kappa_1}^1 + \mathsf{A}_{\kappa_0}^1 & \gamma^1\mathsf{G}_{\kappa_0}^2 & \dots &  \gamma^1\mathsf{G}_{\kappa_0}^n &  \gamma^1\mathsf{G}_{\kappa_0}^\Sigma \\
\gamma^2\mathsf{G}_{\kappa_0}^1 & \mathsf{A}_{\kappa_2}^2 + \mathsf{A}_{\kappa_0}^2 &   &  \gamma^2\mathsf{G}_{\kappa_0}^n &  \gamma^2\mathsf{G}_{\kappa_0}^\Sigma \\
\vdots & & \ddots & & \vdots \\
\gamma^n \mathsf{G}_{\kappa_0}^1 & \gamma^n\mathsf{G}_{\kappa_0}^2 & & \mathsf{A}_{\kappa_n}^n + \mathsf{A}_{\kappa_0}^n  &  \gamma^n\mathsf{G}_{\kappa_0}^\Sigma \\
\binom{\gamma_\dir^\Sigma + \mathsf{M}^*\gamma_\neu^\Sigma}{\gamma_\neu^\Sigma} \mathsf{G}_{\kappa_0}^1 & \binom{\gamma_\dir^\Sigma + \mathsf{M}^*\gamma_\neu^\Sigma}{\gamma_\neu^\Sigma}\mathsf{G}_{\kappa_0}^2 & \dots & \binom{\gamma_\dir^\Sigma + \mathsf{M}^*\gamma_\neu^\Sigma}{\gamma_\neu^\Sigma}\mathsf{G}_{\kappa_0}^n & 
\mathsf{A}_{\kappa_0}^\Sigma + \binom{\mathsf{M}^*\gamma_\neu^\Sigma}{0}\mathsf{G}_{\kappa_0}^\Sigma
\end{bmatrix}.
\end{equation*}
Note that $\wdoublehat{\mathsf{A}}_\mathsf{M}$ differs from $\wdoublehat{\mathsf{A}}$ only in the Dirichlet traces on $\Sigma$ in the last line. 

In a similar way, for the right-hand side $-[\widetilde{\fru}^\inc, \frv]_\Gamma = -[\fru^\inc, \Theta(\Id+\mathsf{C})(\frv)]_\Gamma$, we get
\[
[\fru^\inc, \Theta\mathsf{C}(\frv)]_\Gamma = 
 [\gamma^0 U_\inc, \theta(\mathsf{C}\frv)_0]_{\Gamma_0} = 
 - \braket{\gamma_\neu^\Sigma U_\inc, \mathsf{M}\Tr_\neu(\frv)}_\Sigma =
 - \braket{\mathsf{M}^* \gamma_\neu^\Sigma U_\inc, \Tr_\neu(\frv)}_\Sigma
\]
and combining with the term in \eqref{eq:MTF} that derives from $-[\fru^\inc, \Theta(\frv)]_\Gamma$ we write
\[
\doublehat{\mathfrak{f}}_\mathsf{M} \coloneqq \begin{pmatrix}\gamma^1 U_\inc, \dots, \gamma^n U_\inc, \binom{\gamma_\dir^\Sigma + \mathsf{M}^*\gamma_\neu^\Sigma}{\gamma_\neu^\Sigma} U_\inc \end{pmatrix}.
\]
In conclusion, we define the \emph{global multi-trace combined field FEM-BEM formulation} 
\begin{empheq}[box=\fbox]{equation}
\label{eq:CMTF}
\begin{aligned}
&
\begin{split}
& \text{find } (U,\hat{\fru}) \in \mH^1(\Omega_\Sigma) \times \widehat{\mbH}(\Gamma), \, \hat{\fru} = (\hat{\fru}_1,\dots,\hat{\fru}_n, p_\Sigma), \text{ such that} \\
& a_\Sigma(U,V) + \lBrace \wdoublehat{\mathsf{A}}_\mathsf{M}(\doublehat{\fru}), \Theta(\doublehat{\frv}) \rBrace
+\frac{1}{2} \left[ \left(\gamma_\dir^\Sigma U, p_\Sigma\right), \left(\gamma_\dir^\Sigma V, q_\Sigma \right) \right]_\Sigma \\
&  = F_\Sigma(V)  + \lBrace \doublehat{\mathfrak{f}}_\mathsf{M},\Theta(\doublehat{\frv}) \rBrace \quad \forall (V,\hat{\frv}) \in \mH^1(\Omega_\Sigma) \times \widehat{\mbH}(\Gamma),\hat{\frv} = (\hat{\frv}_1,\dots,\hat{\frv}_n, q_\Sigma)
\end{split}
\\
&
\text{where } 
\doublehat{\fru} \coloneqq (\hat{\fru}_1, \dots, \hat{\fru}_n, (\gamma_\dir^\Sigma U, p_\Sigma)), \quad
\doublehat{\frv} \coloneqq (\hat{\frv}_1, \dots, \hat{\frv}_n, (\gamma_\dir^\Sigma V, q_\Sigma)).
\end{aligned}
\end{empheq}
Even if we have derived this formulation in the gap setting, it is
still valid in a general geometric configuration such as
Figure~\ref{fig:gap}, left. This will be justified in what follows.
We first show which is the relationship of its solutions with the
solutions to the standard multi-trace FEM-BEM formulation
\eqref{eq:MTF}.

\begin{proposition}
\label{prop:MTF-CMTF}
A solution to the combined field multi-trace FEM-BEM formulation
\eqref{eq:CMTF} is also a solution to the standard multi-trace FEM-BEM
formulation \eqref{eq:MTF}.
\end{proposition}
\begin{proof}
Let $(U, \hat{\fru})$ be a solution to formulation
\eqref{eq:CMTF}. Then, if we take test functions $V=0$, $\hat{\frv} =
(0,\dots,0,q_\Sigma)$ with some $q_\Sigma \in \mH^{-1/2}(\Sigma)$
(thus $\gamma_\dir^\Sigma V = 0$, $\doublehat{\frv} =
(0,\dots,0,(0,q_\Sigma))$), it yields
\begin{multline*}
\Braket{\sum_{j=1}^n (\gamma_\dir^\Sigma
  +\mathsf{M}^*\gamma_\neu^\Sigma) \mathsf{G}_{\kappa_0}^j
  (\hat{\fru}_j) + (\{\gamma_\dir^\Sigma\}
  +\mathsf{M}^*\gamma_\neu^\Sigma)\mathsf{G}_{\kappa_0}^\Sigma
  \left(\gamma_\dir^\Sigma U, p_\Sigma\right), q_\Sigma}_\Sigma +
\frac{1}{2} \Braket{\gamma_\dir^\Sigma U, q_\Sigma}_\Sigma \\ =
\Braket{(\gamma_\dir^\Sigma +\mathsf{M}^*\gamma_\neu^\Sigma) U_\inc,
  q_\Sigma}_\Sigma \qquad \forall \, q_\Sigma \in \mH^{-1/2}(\Sigma),
\end{multline*}
and, since 
\[
\Braket{\{\gamma_\dir^\Sigma\}\mathsf{G}_{\kappa_0}^\Sigma
  \left(\gamma_\dir^\Sigma U, p_\Sigma\right), q_\Sigma}_\Sigma = \Braket{
  \gamma_\dir^\Sigma \mathsf{G}_{\kappa_0}^\Sigma
  \left(\gamma_\dir^\Sigma U, p_\Sigma\right), q_\Sigma}_\Sigma -
\frac{1}{2} \Braket{\gamma_\dir^\Sigma U, q_\Sigma}_\Sigma,
\]
we obtain
\[
\Braket{(\gamma_\dir^\Sigma +\mathsf{M}^*\gamma_\neu^\Sigma) \biggl( U_\inc -\sum_{j=1}^n \mathsf{G}_{\kappa_0}^j (\hat{\fru}_j) - \mathsf{G}_{\kappa_0}^\Sigma \left(\gamma_\dir^\Sigma U, p_\Sigma\right) \biggr), q_\Sigma}_\Sigma = 0
 \quad \forall \, q_\Sigma \in \mH^{-1/2}(\Sigma).
\]
Therefore, if we introduce
\begin{equation}\label{FctIntermediaireW}
  W \coloneqq U_\inc -\sum_{j=1}^n \mathsf{G}_{\kappa_0}^j (\hat{\fru}_j)
  - \mathsf{G}_{\kappa_0}^\Sigma \left(\gamma_\dir^\Sigma U, p_\Sigma\right),
\end{equation}
this means $\gamma_\dir^\Sigma W = -\mathsf{M}^* \gamma_\neu^\Sigma W$. 
Moreover, $W$ solves $-\Delta W - \kappa_0^2 W = 0$ in $\Omega_\Sigma$, so by Green's formula 
\[
\int_{\Omega_\Sigma} (|\nabla W|^2 - \kappa_0^2 |W|^2) d\bx =
-\Braket{\gamma_\neu^\Sigma \overline{W},
  \mathsf{M}^*\gamma_\neu^\Sigma W}_\Sigma,
\]
and taking the imaginary part, since $\kappa_0 \in \RR$, we obtain $0
= - \Im\{ \braket{\mathsf{M}\gamma_\neu^\Sigma \overline{W},
  \gamma_\neu^\Sigma W}_\Sigma \}$, that implies $\gamma_\neu^\Sigma W
= 0$ by property \eqref{eq:regopB} of $\mathsf{M}$.
As a consequence $\gamma_\dir^\Sigma W =
-\mathsf{M}^*\gamma_\neu^\Sigma W = 0$.  The conclusion
$\gamma_\dir^\Sigma W = 0$ finishes the proof because, looking at the
definition of $W$, this corresponds exactly to the equation in
formulation \eqref{eq:MTF} associated with the Dirichlet component of
the last line of $\wdoublehat{\mathsf{A}}$ and
$\doublehat{\mathfrak{f}}$, which represents the only difference
between formulations \eqref{eq:MTF} and \eqref{eq:CMTF}.
\end{proof}
A corollary of this proposition is that if $(U, \hat{\fru})$ satisfies
formulation \eqref{eq:CMTF}, then the unique solution to the
transmission problem \eqref{eq:bvp} is given by $\widetilde{U}$ in
\eqref{eq:UtildeMTF}. This justifies considering formulation
\eqref{eq:CMTF} for \emph{general geometric settings}.

Moreover, by the compactness of $\mathsf{M}$, the block operator
$\wdoublehat{\mathsf{A}}_\mathsf{M}$ is a compact perturbation of
$\wdoublehat{\mathsf{A}}$, so a \emph{G\r{a}rding inequality} analogue
to Proposition~\ref{pr:gardingMTF} still holds, and the induced
operator is of Fredholm type with index $0$. Therefore, in the case of
injectivity, all the good properties recalled below
Proposition~\ref{pr:gardingSTF} follow.  As desired, the combined
field formulation \eqref{eq:CMTF} is immune to spurious resonances for
any choice of the positive wavenumbers $\kappa_j$:
\begin{proposition}[Injectivity]
Let $(U,\hat{\fru}) \in \mH^1(\Omega_\Sigma) \times
\widehat{\mbH}(\Gamma)$ solve formulation \eqref{eq:CMTF} with
$F_\Sigma \equiv 0$, $\doublehat{\mathfrak{f}}_\mathsf{M}=0$. Then
$U=0$, $\hat{\fru} = 0$.
\end{proposition}
\begin{proof}
  Since $\doublehat{\mathfrak{f}}_\mathsf{M}=0$ we have
  $U_\inc = 0$. As a consequence proceeding as in the
  beginning of the proof of Proposition \ref{prop:MTF-CMTF} leads to
  considering 
  $W \coloneqq -\sum_{j=1}^n \mathsf{G}_{\kappa_0}^j (\hat{\fru}_j) -
  \mathsf{G}_{\kappa_0}^\Sigma(\gamma_\dir^\Sigma U,p_\Sigma)$
  and, following the same argumentation as above, this function satisfies
  $\gamma_{\neu}^\Sigma(W) = 0$. According to the definition of
  $\wdoublehat{\mathsf{A}}_{\mathsf{M}}$, this implies
  \begin{equation*}
    \wdoublehat{\mathsf{A}}_{+\mathsf{M}}(\doublehat{\fru}) =
    \wdoublehat{\mathsf{A}}(\doublehat{\fru}) =
    \wdoublehat{\mathsf{A}}_{-\mathsf{M}}(\doublehat{\fru})
  \end{equation*}
  since the terms involving $\mathsf{M}^*$ in the last row of the
  definition of $\wdoublehat{\mathsf{A}}_{\mathsf{M}}$ vanish. 
  Next, by Proposition \ref{prop:MTF-CMTF}, $(U, \hat{\fru})$ solves also
  formulation \eqref{eq:MTF}, so by Proposition \ref{pr:injMTF} we get
  $U=0$ and $\gamma_\dir^\Sigma U = 0$. Now, if we test formulation
  \eqref{eq:CMTF} (with $F_\Sigma \equiv
  0$, $\doublehat{\mathfrak{f}}_\mathsf{M}=0$) using test functions $V
  \in \mH_0^1(\Omega_\Sigma)$ (and $\hat{\frv} =
  (\hat{\frv}_1,\dots,\hat{\frv}_n, q_\Sigma)\in
  \widehat{\mbH}(\Gamma)$), we obtain $\lBrace \wdoublehat{\mathsf{A}}_\mathsf{M}(\doublehat{\fru}),
    \Theta(\doublehat{\frv}) \rBrace = 0$ hence 
  \begin{equation*}
    \lBrace \wdoublehat{\mathsf{A}}_{-\mathsf{M}}(\doublehat{\fru}),
    \Theta(\doublehat{\frv}) \rBrace = 0
  \end{equation*}
  with $\doublehat{\fru} = (\hat{\fru}_1, \dots, \hat{\fru}_n, (0, p_\Sigma))$,
  $\doublehat{\frv} = (\hat{\frv}_1, \dots, \hat{\frv}_n, (0, q_\Sigma))$.
  Note that this reduces to $\llbracket \widehat{\mathsf{A}}_\mathsf{M}(\hat{\fru}), \Theta(\hat{\frv})\rrbracket = 0$ for all $\hat{\frv} \in
  \widehat{\mbH}(\Gamma)$, where $\widehat{\mathsf{A}}_\mathsf{M}$ is
  defined in \cite[Equation~(6.21)]{ClHi:impenetrable:2015} and is
  injective by \cite[Proposition 6.7]{ClHi:impenetrable:2015}. Then
  $\hat{\fru} = 0$.
\end{proof}

\smallskip
\begin{appendix}

\section{Properties of the block boundary integral operator $\mathsf{A}_\kappa^\Omega$}
\label{app:operatorAomega}

\noindent We prove here two useful properties of the boundary integral operator $\mathsf{A}_\kappa^\Omega$ in \eqref{DefPMCHWTOp} since we could not find detailed proofs in the literature.  
\begin{proposition}[Generalized G\r{a}rding inequality]
\label{pr:gardingAjAppendix}
Set $\theta(v, q) \coloneqq (-v, q)$.
Let $\Omega$ be a generic Lipschitz domain that is either bounded or such that $\RR^d\setminus\overline{\Omega}$ is bounded. Then, there exist a compact operator $\mathcal{K}\colon \mbH(\partial \Omega) \to \mbH(\partial \Omega)$ and a constant $\alpha>0$ such that for all $ \fru \in \mbH(\partial \Omega)$ we have
\[
\Re \left \{ [(\mathsf{A}_\kappa^\Omega + \mathcal{K})\fru, \theta(\overline{\fru})]_{\partial\Omega} \right \}
\ge \alpha \lVert \fru \rVert_{\mbH(\partial \Omega)}^2.
\] 
\end{proposition}
\begin{proof}
Since a change of the wavenumber $\kappa$ only induces a compact perturbation of $\mathsf{A}_\kappa^\Omega$ \cite[Lemma 3.9.8]{SaSw:book:2011}, it suffices to prove the result for the case $\kappa = \imagi$, where $\imagi = \sqrt{-1}$. 
Set $\psi \coloneqq \mathsf{G}_\kappa^\Omega (\fru)$, then we write $\mathsf{A}_\kappa^\Omega (\fru) = \{\gamma^\Omega\} \, \psi$ and by the jump relations \eqref{eq:jumprels} we have $\fru = [\gamma^\Omega] \, \psi$. Therefore 
\[
\left[\mathsf{A}_\kappa^\Omega (\fru), \theta(\overline{\fru})\right]_{\partial\Omega} = \left[ \{\gamma^\Omega\} \, \psi, \theta [\gamma^\Omega] \, \overline{\psi} \right]_{\partial\Omega} =
\frac{1}{2} \left[ (\gamma^\Omega+\gamma_c^\Omega) \psi, \theta (\gamma^\Omega-\gamma_c^\Omega) \overline{\psi} \right]_{\partial\Omega} = m_1 + m_2
\]
where
\begin{align*}
&m_1 = \frac{1}{2} \left[ \gamma^\Omega \psi, \theta \gamma^\Omega \overline{\psi} \right]_{\partial\Omega} -  \frac{1}{2} \left[ \gamma_c^\Omega \psi, \theta \gamma_c^\Omega \overline{\psi} \right]_{\partial\Omega}, \\
&m_2 = \frac{1}{2} \left[ \gamma_c^\Omega \psi, \theta \gamma^\Omega \overline{\psi} \right]_{\partial\Omega} -  \frac{1}{2} \left[ \gamma^\Omega \psi, \theta \gamma_c^\Omega \overline{\psi} \right]_{\partial\Omega}. 
\end{align*}
We have $\Re (m_2)=0$, indeed
\begin{equation*}
\Re \left \{ [ \gamma_c^\Omega \psi, \theta \gamma^\Omega \overline{\psi} ]_{\partial\Omega} \right \} = \Re \left \{ \overline{[ \gamma_c^\Omega \psi, \theta \gamma^\Omega \overline{\psi} ]}_{\partial\Omega}  \right \} 
= \Re \left \{ [ \gamma_c^\Omega \overline{\psi}, \theta \gamma^\Omega {\psi} ]_{\partial\Omega} \right \} = \Re \left \{ [ \gamma^\Omega {\psi}, \theta \gamma_c^\Omega \overline{\psi} ]_{\partial\Omega} \right \}, 
\end{equation*}
where the last equality is an application of the property $[\fru, \theta(\frv)]_{\partial\Omega} = [\frv, \theta(\fru)]_{\partial\Omega}$ for $\fru, \frv \in \mbH(\partial \Omega)$.
To deal with $\Re (m_1)$, observe that we have $\Re \left \{ [\frv, \theta(\overline{\frv})]_{\partial\Omega}  \right \} = \Re  \{ \braket{v,\overline{q}}_{\partial\Omega} + \braket{\overline{v},q}_{\partial\Omega}  \} = 2 \Re  \{ \braket{v,\overline{q}}_{\partial\Omega} \}$ for $\frv = (v,q) \in \mbH(\partial \Omega)$. Thus
\begin{equation*}
\frac{1}{2} \Re \left \{ \left[ \gamma^\Omega \psi, \theta \gamma^\Omega \overline{\psi} \right]_{\partial\Omega} \right \} =  \Re \left \{ \braket{\gamma_\dir^\Omega \psi,\gamma_\neu^\Omega  \overline{\psi}}_{\partial\Omega} \right \} 
 = \Re \biggl \{ \int_{\Omega} (\lvert \nabla \psi \rvert^2 + \psi \Delta \overline{\psi}) d\bx \biggr \} = \lVert \psi \rVert _{\mH^1(\Omega)}^2,
\end{equation*}
where we integrated by parts and lastly used the fact that $\psi$ is a solution to the Helmholtz equation with $\kappa= \imagi$, so that $\Delta \psi = \psi$. 
Similarly, we get 
\[
 -\frac{1}{2} \Re \left \{ \left[ \gamma_c^\Omega \psi, \theta \gamma_c^\Omega \overline{\psi} \right]_{\partial\Omega} \right \} = \lVert \psi \rVert _{\mH^1(\RR^d\backslash\Omega)}^2,
\]
therefore
\[
\Re \left \{ \left[\mathsf{A}_\kappa^\Omega (\fru), \theta(\overline{\fru})\right]_{\partial\Omega}  \right \} = \lVert \psi \rVert _{\mH^1(\Omega)}^2 + \lVert \psi \rVert _{\mH^1(\RR^d\backslash\Omega)}^2.
\]
Now, note that 
\[
\lVert \psi \rVert _{\mH^1(\Delta, \Omega)}^2 = \lVert \psi \rVert _{\mH^1(\Omega)}^2  + \lVert \Delta \psi \rVert _{\mL^2(\Omega)}^2 = \lVert \psi \rVert _{\mH^1(\Omega)}^2  + \lVert \psi \rVert _{\mL^2(\Omega)}^2 \le 2 \lVert \psi \rVert _{\mH^1(\Omega)}^2,
\]
and by the continuity of the trace operators, there exists $C>0$ such that 
\begin{align*}
&  \lVert \gamma_\dir^\Omega V \rVert_{\mH^{1/2}(\partial \Omega)}^2 +  \lVert \gamma_\neu^\Omega V \rVert_{\mH^{-1/2}(\partial \Omega)}^2 
\le C \lVert V \rVert _{\mH^1(\Delta, \Omega)}^2  \quad \forall \, V \in \mH^1(\Delta, \Omega), \\
&  \lVert \gamma_{\dir,c}^\Omega V \rVert_{\mH^{1/2}(\partial \Omega)}^2 +  \lVert \gamma_{\neu,c}^\Omega V \rVert_{\mH^{-1/2}(\partial \Omega)}^2 
\le C \lVert V \rVert _{\mH^1(\Delta, \RR^d\backslash\Omega)}^2  \quad \forall \, V \in \mH^1(\Delta, \RR^d\backslash\Omega).
\end{align*}
Therefore
\begin{align*}
& \Re \left \{ \left[\mathsf{A}_\kappa^\Omega (\fru), \theta(\overline{\fru})\right]_{\partial\Omega}  \right \} = \lVert \psi \rVert _{\mH^1(\Omega)}^2 + \lVert \psi \rVert _{\mH^1(\RR^d\backslash\Omega)}^2 \\
& \ge \frac{1}{2C} \left(   \lVert \gamma_\dir^\Omega \psi \rVert_{\mH^{1/2}(\partial \Omega)}^2 + \lVert \gamma_{\dir,c}^\Omega \psi \rVert_{\mH^{1/2}(\partial \Omega)}^2 +  \lVert \gamma_\neu^\Omega \psi \rVert_{\mH^{-1/2}(\partial \Omega)}^2   +  \lVert \gamma_{\neu,c}^\Omega \psi \rVert_{\mH^{-1/2}(\partial \Omega)}^2  \right) \\
& \ge \frac{1}{4C} \left(  \lVert (\gamma_{\dir}^\Omega - \gamma_{\dir,c}^\Omega) \psi \rVert_{\mH^{1/2}(\partial \Omega)}^2  +  \lVert (\gamma_{\neu}^\Omega - \gamma_{\neu,c}^\Omega) \psi \rVert_{\mH^{-1/2}(\partial \Omega)}^2  \right) 
= \frac{1}{4C} \lVert \fru \rVert_{\mbH(\partial \Omega)}^2,
\end{align*}
where we used the triangular inequality and the jump relations \eqref{eq:jumprels}. 
\end{proof}

\begin{proposition}\label{RadiationConditionConsequence2}
  Assume that either $\Omega\subset \RR^d$ is bounded or $\RR^d\setminus \Omega$ is bounded.
  Then for all $\fru\in \mbH(\partial\Omega)$, we have
  $\Im \{[\mathsf{A}_\kappa^{\Omega}(\fru),\overline{\fru}]_{\partial\Omega}\}\geq 0$.
\end{proposition}
\begin{proof}
  
  Assume first that $\RR^d\setminus \Omega$ is bounded, pick an arbitrary
  $\fru\in \mbH(\partial\Omega)$ and set $\psi(\bx):=\mathsf{G}_\kappa^{\Omega} (\fru)(\bx)$.
  We have $\lbr \gamma^{\Omega}(\psi)\rbr = \fru$ according to the jump formula \eqref{eq:jumprels}
  and, on the other hand, $\{ \gamma^{\Omega}(\psi)\} = \mathsf{A}_\kappa^{\Omega}(\fru)$ according to definition
  \eqref{DefPMCHWTOp}. As a consequence, developing the expression 
  $2\lbr \mathsf{A}_\kappa^{\Omega}(\fru),\overline{\fru}\rbr_{\partial\Omega} =
  \lbr \gamma^{\Omega}(\psi)+\gamma^{\Omega}_{c}(\psi),
  \gamma^{\Omega}(\overline{\psi})-\gamma^{\Omega}_{c}(\overline{\psi}) \rbr_{\partial\Omega}$,
  yields
  \begin{equation}
    \begin{aligned}
      2\lbr \mathsf{A}_\kappa^{\Omega}(\fru),\overline{\fru}\rbr_{\partial\Omega}
      & =  \lbr \gamma^{\Omega}(\psi),\gamma^{\Omega}(\overline{\psi})\rbr_{\partial\Omega}
      - \lbr \gamma_c^{\Omega}(\psi),\gamma_c^{\Omega}(\overline{\psi})\rbr_{\partial\Omega}\\
      & \,+  \lbr \gamma_c^{\Omega}(\psi),\gamma^{\Omega}(\overline{\psi})\rbr_{\partial\Omega}
      - \lbr \gamma^{\Omega}(\psi),\gamma_c^{\Omega}(\overline{\psi})\rbr_{\partial\Omega}\\
      & =  \lbr \gamma^{\Omega}(\psi),\gamma^{\Omega}(\overline{\psi})\rbr_{\partial\Omega}
      - \lbr \gamma_c^{\Omega}(\psi),\gamma_c^{\Omega}(\overline{\psi})\rbr_{\partial\Omega}\\
      & \,+ 2 \Re\{\lbr \gamma_c^{\Omega}(\psi),\gamma^{\Omega}(\overline{\psi})\rbr_{\partial\Omega}\}.
    \end{aligned}
  \end{equation}
  Next observe that each of the first two terms in the right-hand side above takes the form
  $\lbr\frv,\overline{\frv}\rbr_{\partial\Omega}$ and satisfies
  $\overline{\lbr\frv,\overline{\frv}\rbr}\!\,_{\partial\Omega} =
  \lbr\overline{\frv},\frv\rbr_{\partial\Omega} =
  -\lbr\frv,\overline{\frv}\rbr_{\partial\Omega}$ which means that they are
  pure imaginary numbers, i.e.~$\imagi\lbr\frv,\overline{\frv}\rbr_{\partial\Omega}\in\RR$.
  As a consequence 
  \begin{equation}\label{DecompositionImaginaryPart}
      2\imagi \Im\{\lbr \mathsf{A}_\kappa^{\Omega}(\fru),\overline{\fru}\rbr_{\partial\Omega}\} =
       +\lbr \gamma^{\Omega}(\psi),\gamma^{\Omega}(\overline{\psi})\rbr_{\partial\Omega}
       -\lbr \gamma_c^{\Omega}(\psi),\gamma_c^{\Omega}(\overline{\psi})\rbr_{\partial\Omega}.
  \end{equation}
  We examine each term in the right-hand side of this identity.
  Both $\psi$ and $\overline{\psi}$ satisfy a homogeneous Helmholtz
  equation in $\RR^d\setminus\overline{\Omega}$ and, since it is a
  bounded domain, we can apply Green's formula in
  $\RR^d\setminus\overline{\Omega}$. This implies that 
  the second term in the right-hand side of
  \eqref{DecompositionImaginaryPart} vanishes: 
  \begin{equation*}
    \begin{aligned}
      \lbr \gamma_c^{\Omega}(\psi),\gamma_c^{\Omega}(\overline{\psi})\rbr_{\partial\Omega}
      & = \int_{\partial\Omega}\gamma_{\dir,c}^{\Omega}(\psi)\gamma_{\neu,c}^{\Omega}(\overline{\psi}) -
      \gamma_{\neu,c}^{\Omega}(\psi)\gamma_{\dir,c}^{\Omega}(\overline{\psi}) d\sigma\\
      & = \int_{\RR^d\setminus\overline{\Omega}} \psi(\Delta\overline{\psi}+\kappa^{2}\overline{\psi}) - \overline{\psi}(\Delta\psi+\kappa^{2}\psi) \, d\bx = 0.
    \end{aligned}
  \end{equation*}
  To study the first term in \eqref{DecompositionImaginaryPart}
  choose $\rho>0$ large enough to have $\RR^{d}\setminus\Omega\subset \mrm{B}_{\rho}$, where $\mrm{B}_{\rho}$ is the open ball of center $0$ and radius $\rho$. 
  Applying Green's formula in $\Omega\cap\mrm{B}_{\rho}$ gives
  \begin{equation*}
    \begin{aligned}  
      \lbr \gamma^{\Omega}(\psi),\gamma^{\Omega}(\overline{\psi})\rbr_{\partial\Omega}
      &  = \int_{\partial\Omega}\gamma_{\dir}^{\Omega}(\psi)\gamma_{\neu}^{\Omega}(\overline{\psi}) -
      \gamma_{\neu}^{\Omega}(\psi)\gamma_{\dir}^{\Omega}(\overline{\psi})\;d\sigma \\
      & = \int_{\Omega\cap\mrm{B}_\rho} 
      \psi(\Delta\overline{\psi}+\kappa^{2}\overline{\psi})
      - \overline{\psi}(\Delta\psi+\kappa^{2}\psi)\, d\bx
       \, +\int_{\partial \mrm{B}_\rho} \overline{\psi}\partial_{\rho}\psi - \psi\partial_{\rho}\overline{\psi}\;d\sigma_{\rho}
    \end{aligned}
  \end{equation*}
  where $\partial_{\rho}\psi$ is the Neumann trace on $\partial
  \mrm{B}_{\rho}$ and $d\sigma_\rho$ is the surface measure on
  $\partial \mrm{B}_{\rho}$.  The volume terms vanish because
  $\Delta\psi+\kappa^{2}\psi = 0$ in $\Omega$. Multiplying
  this identity by $-\imagi\kappa$ then leads to 
  \begin{equation*}
    \begin{aligned}
      -\imagi \kappa  \lbr \gamma^{\Omega}(\psi),\gamma^{\Omega}(\overline{\psi})\rbr_{\partial\Omega}
      & =  \int_{\partial \mrm{B}_\rho} \overline{\imagi\kappa\psi}\,\partial_{\rho}\psi + \imagi\kappa\psi\,\partial_{\rho}\overline{\psi}\;d\sigma_\rho
      = 2\Re\{\int_{\partial \mrm{B}_\rho}\imagi\kappa\psi\,\partial_{\rho}\overline{\psi}\;d\sigma\}\\
      & = -\int_{\partial \mrm{B}_\rho}\vert \partial_{\rho}\psi-i\kappa \psi\vert^{2} d\sigma +
      \int_{\partial \mrm{B}_\rho}\vert \partial_{\rho}\psi\vert^{2}+ \kappa^{2} \vert\psi\vert^{2} d\sigma\\
       &\geq -\int_{\partial \mrm{B}_\rho}\vert \partial_{\rho}\psi-i\kappa \psi\vert^{2} d\sigma. 
    \end{aligned}
  \end{equation*}
  Since this inequality must hold for any $\rho>0$ large enough, we can pass
  to the limit $\rho\to +\infty$ and, by the Sommerfeld radiation
  condition satisfied by $\psi = \mathsf{G}_\kappa^{\Omega} (\fru)$, we finally
  conclude that $\Im\{\lbr \mathsf{A}_\kappa^{\Omega}(\fru),\overline{\fru}\rbr_{\partial\Omega}\} = 
  -\imagi  \lbr \gamma^{\Omega}(\psi),\gamma^{\Omega}(\overline{\psi})\rbr_{\partial\Omega}/2\in\lbr 0,+\infty)$.

  To conclude the proof, let us consider the case where $\Omega$ is bounded, and denote $\Omega^{c}:=\RR^d\setminus\overline{\Omega}$.
  Because $\bn_{\Omega^c} = -\bn_{\Omega}$, we conclude that $\gamma^{\Omega^c} = -\theta\circ\gamma_c^{\Omega}$,
  $\gamma^{\Omega^c}_c = -\theta\circ\gamma^{\Omega}$, and 
  $\mathsf{G}_\kappa^{\Omega^c} = \mathsf{G}_\kappa^{\Omega}\circ\theta$, and hence
  $\mathsf{A}_\kappa^{\Omega} = -\theta\circ\mathsf{A}_\kappa^{\Omega^c}\circ\theta$. The domain $\Omega^{c}$
  is unbounded, so we can apply the first part of the present proof, which finally yields 
  \begin{equation*}
      \Im \{\lbr \mathsf{A}_\kappa^{\Omega}(\fru),\overline{\fru}\rbr_{\partial\Omega}\}
       = - \Im \{\lbr \theta \circ\mathsf{A}_\kappa^{\Omega^c}\circ\theta(\fru),\overline{\fru}\rbr_{\partial\Omega}\}
       = + \Im \{\lbr\mathsf{A}_\kappa^{\Omega^c}\circ\theta(\fru),\overline{\theta(\fru)}\rbr_{\partial\Omega}\}\geq 0.
  \end{equation*}  
\end{proof}

\end{appendix}

\bibliographystyle{amsplain}
\bibliography{paper_fembem}

\end{document}